\documentclass{amsart}
\usepackage{amsmath,amssymb,amsthm,amsfonts,amscd,mathrsfs,mathtools}
\usepackage[utf8]{inputenc}
\usepackage[hyphens]{url}
\usepackage[pagebackref=true]{hyperref}
\usepackage{tikz-cd}
\usepackage{tikz}

\usepackage{caption}
\usepackage{stmaryrd}
\usepackage{enumitem}

\setlist[itemize,description]{leftmargin=*}

\theoremstyle{plain}
    \newtheorem{thm}{Theorem}[section]
    \newtheorem{prop}[thm]  {Proposition}
    \newtheorem{lem}[thm]   {Lemma}
    \newtheorem{cor}[thm]   {Corollary}
    
    \newtheorem{athm}{Theorem}
    
\newtheorem{acor}[athm]{Corollary}

\theoremstyle{definition}
    \newtheorem{rem}[thm]   {Remark}
    \newtheorem{defn}[thm]  {Definition}
    \newtheorem{nota}[thm]  {Notation}
    \newtheorem{ex}[thm]    {Example}

\newcommand{\cA}{\mathcal{A}} \newcommand{\cAh}{\mathcal{A}^\heartsuit}
\newcommand{\cB}{\mathcal{B}} 
\newcommand{\bB}{\mathbb{B}} 
\newcommand{\bbB}{\mathbf{B}} 
\newcommand{\fb}{\mathfrak{b}}
\newcommand{\C}{\mathbb{C}} 
\newcommand{\cC}{\mathcal{C}} 
\newcommand{\cE}{\mathcal{E}} 
\newcommand{\fF}{\mathfrak{F}} 
\newcommand{\fg}{\mathfrak{g}} 
\newcommand{\cH}{\mathcal{H}} 
\newcommand{\cI}{\mathcal{I}} 
\newcommand{\I}{I} 
\newcommand{\fj}{\mathfrak{j}}
\newcommand{\bK}{\mathbf{K}} 
\newcommand{\bL}{\mathbf{L}} 
\newcommand{\fM}{\mathfrak{M}} 
\newcommand{\bN}{\mathbb{N}} 
\newcommand{\fo}{\mathfrak{o}}
\newcommand{\bP}{\mathbb{P}}
\newcommand{\fp}{\mathfrak{p}}
\newcommand{\fP}{\mathfrak{P}} 
\newcommand{\bQ}{\mathbb{Q}}
\newcommand{\R}{\mathbb{R}} 
\newcommand{\fS}{\mathfrak{S}} 
\newcommand{\cF}{\mathcal{F}} 
\newcommand{\mcF}{\mathring{\cF}} 
\newcommand{\fU}{\mathfrak{U}} 
\newcommand{\bZ}{\mathbb{Z}} 

\newcommand{\Sp}{\mathrm{Sp}} 
\newcommand{\SP}{\mathrm{SP}} 
\newcommand{\Vect}{\mathrm{Vec}}
\newcommand{\op}{^{\mathrm{op}}}
\newcommand{\CAlg}{\mathrm{CAlg}}
\newcommand{\Alg}{\mathrm{Alg}}
\newcommand{\shift}{\mathrm{sh}}

\newcommand{\SO}{\mathrm{SO}}
\newcommand{\Homeo}{\mathrm{Homeo}}
\newcommand{\Aut}{\mathrm{Aut}}
\newcommand{\Hom}{\mathrm{Hom}}
\newcommand{\CHom}{\dot{\mathrm{C}}\mathrm{Hom}}
\newcommand{\MG}{\mathbf{MG}} 
\newcommand{\TMG}{\mathbf{TMG}} 
\newcommand{\Fin}{\mathrm{Fin}}
\newcommand{\Fun}{\mathrm{Fun}}
\newcommand{\FS}{\mathrm{FS}} 

\usepackage{mathbbol}
\DeclareSymbolFontAlphabet{\mathbb}{AMSb}
\DeclareSymbolFontAlphabet{\mathbbl}{bbold}
\newcommand{\one}{\mathbbl{1}} 
\newcommand{\uone}{\underline{\one}}

\newcommand{\tp}{^{\mathrm{top}}}
\newcommand{\tpp}{^{\mathrm{top},\infty}}
\newcommand{\Ptp}{\bP^{1,\mathrm{top}}}
\newcommand{\geo}{^{\mathrm{geo}}}
\newcommand{\rel}{{\,\text{rel}\,}}
\newcommand{\et}{{\,\text{\'et}\,}}
\newcommand{\bt}{\mathbf{t}} 
\newcommand{\ulambda}{\underline{\lambda}}
\newcommand{\uphi}{\underline{\phi}}
\newcommand{\trsp}
{\mathrm{trsp}} 
\newcommand{\uk}{\underline{k}}
\newcommand{\uzero}{\underline{0}}
\newcommand{\unu}{\underline{\nu}}
\newcommand{\umu}{\underline{\mu}}
\newcommand{\utau}{\underline{\tau}} 
\newcommand{\lift}{\mathrm{lift}} 
\newcommand{\udll}{\underline{l}}

\newcommand{\udlS}{\underline{S}}
\newcommand{\ccH}{\check{\cH}}
\newcommand{\dcH}{\ddot{\cH}}
\newcommand{\dq}{\ddot{q}}
\newcommand{\dr}{\ddot{r}}
\newcommand{\brv}{\mathbf{brv}} 
\newcommand{\tbrv}{\widetilde{\brv}} 
\newcommand{\Ran}{\mathrm{Ran}}
\newcommand{\ER}{\mathrm{ERan}}
\newcommand{\Hur}{\mathrm{Hur}} 
\newcommand{\CHur}{\mathrm{CHur}}
\newcommand{\dCHur}{\dot{\mathrm{C}}\mathrm{Hur}} 
\newcommand{\Bra}{\mathrm{Bra}}
\newcommand{\dCBra}{\ddot{\mathrm{C}}\mathrm{Bra}}
\newcommand{\dw}{\ddot w}
\newcommand{\CBra}{\mathrm{CBra}}
\newcommand{\SmP}{\Sigma\setminus P}
\newcommand{\SmB}{\Sigma\setminus B}
\newcommand{\Sma}{\Sigma\setminus\set{a}}
\newcommand{\uU}{\underline{U}}
\newcommand{\SmU}{\Sigma\setminus\uU}
\newcommand{\FmbP}{\cF\setminus\bar P}
\newcommand{\Conf}{\mathrm{Conf}}
\newcommand{\mD}{\mathring{\Delta}} 

\newcommand{\pa}[1]{\left(#1\right)}
\newcommand{\set}[1]{\left\{#1\right\}}
\newcommand{\sca}[1]{[\! [#1]\! ]}

\renewcommand{\phi}{\varphi}
\renewcommand{\epsilon}{\varepsilon}

\DeclareMathOperator{\colim}{colim}
\newcommand{\coarse}{\mathrm{crs}}
\newcommand{\vcoarse}{\mathrm{vcrs}}
\newcommand{\Sing}{\mathrm{Sing}}
\newcommand{\Pic}{\mathrm{Pic}}

\usetikzlibrary{decorations.markings}
\tikzset{
midarrow/.style={
postaction={decorate,decoration={markings,mark=at position 0.5 with {\arrow{stealth}}}}}
}

\title[Stable cohomology splitting of moduli of branched covers]{Stable rational cohomology splitting of moduli of branched covers of curves}
\author{Andrea Bianchi}
\thanks{
}
\email{andrea.bianchi37@unibo.it}
\address{Mathematics Department, University of Bologna\newline
Piazza di Porta San Donato 5, 40126, Bologna, Italy}  
\date{\today}

\keywords{Moduli of curves, Hurwitz stack, stable cohomology,  factorisation homology, 
configuration space}
\subjclass[2020]{
14D22 
18G31 
14H10  	
14H15  	
14H51 
55N25  
55R80 
}

\begin{document}
\begin{abstract}
We consider the moduli stack $\mathcal{H}_{g,d}$ parametrising maps $\theta\colon C\to L$ of degree $d$ between complex curves of genera $g$ and $0$. We provide a splitting of the stable rational cohomology of $\mathcal{H}_{g,d}$, for $g\to\infty$, in terms of degrees of intermediate covers, and identify each direct summand in terms of factorisation homology.
For $2\le d\le 5$, we further interpret the stable rational cohomology in terms of configuration spaces with labels in partial abelian monoids; we show in particular the existence of stable odd-dimensional cohomology classes for $d=4,5$.
\end{abstract}

\maketitle
\section{Introduction}
For $d\ge1$ and $g\ge0$ we denote by $\cH_{g,d}$ the moduli stack of maps $\theta\colon C\to L$ of degree $d$ between connected smooth complex projective curves, with $C$ of genus $g$ and $L$ of genus $0$.
We are interested in computing the rational $\ell$-adic cohomology of $\cH_{g,d}$. If $|\cH_{g,d}\tp|$ denotes the underlying homotopy type of the topological stack associated with $\cH_{g,d}$ in the sense of Noohi \cite{Noohi05}, then the equivalence
\[
H^*_{\text{\'et}}(\cH_{g,d};\bQ_\ell)\cong H^*(|\cH_{g,d}\tp|;\bQ)\otimes_{\bQ}\bQ_\ell
\]
allows us to focus on the computation of the rational cohomology of $|\cH_{g,d}\tp|$.
Our first result asserts that $H^*(|\cH_{g,d}\tp|;\bQ)$ splits, in a range of degrees growing with $g$, as the direct sum of the cohomologies of the underlying homotopy types of certain substacks of $\cH_{g,d}$, which we next introduce 
\begin{nota}
\label{nota:ccH}
We denote by $\ccH_{g,d}\subseteq\cH_{g,d}$ the open substack parametrising maps $\theta\colon C\to L$ admitting at least one simple branch value in $L$.

For all $h\ge0$ and all divisors $e\mid d$ with $e<d$, we denote by $\fF_{h,e}\ccH_{g,d}\subseteq\ccH_{g,d}$ the locally closed substack parametrising maps $\theta\colon C\to L$
that satisfy the following:
\begin{itemize}
\item $\theta$ factors as a composition $C\to C'\xrightarrow{\theta'} L$ with $C'$ of genus $h$ and $\theta'$ of degree $e$;
\item $f$ does not factor through a map to $L$ of degree strictly between $e$ and $d$.
\end{itemize}
\end{nota}
\begin{nota}
\label{nota:rho}
We abbreviate $\rho(h,e,d)=(h+e-1)(d/e-1)\ge0$.
\end{nota}
\begin{athm}
\label{thm:A}
Let $d\ge2$. Then there exist constants $I_d,J_d>0$, only depending on $d$, such that for $g\ge1$ there is an isomorphism of rational vector spaces
\begin{equation}
\label{eq:thmA}
H^*(|\cH_{g,d}\tp|;\bQ)\cong\bigoplus_{\overset{h\ge0}{1\le e<d,\ e\mid d}} H^{*-4\rho(h,e,d)}(|\fF_{h,e}\ccH_{g,d}\tp|;\bQ)
\end{equation}
in the range of cohomological degrees $*<(2(g+d-1)-J_d)/I_d$.
\end{athm}
\begin{rem}
For fixed $g\ge1$ only finitely many summands in \eqref{eq:thmA} may contribute non-trivially in the given range of cohomological degrees.
\end{rem}
\begin{rem}
Our proof of Theorem \ref{thm:A} relies on the validity of homology stability for Hurwitz spaces with monodromies in $\fS_d$. To the best of our knowledge, besides the case $d=2$, the only results in the literature that may be applied in our context are given by the recent work of Landesman--Levy \cite{LL1,LL2}, and in particular \cite[Theorem 1.4.9]{LL2}.
We gather in Notation \ref{nota:goodconstantsIJ} a precise list of properties that the constants $I_d,J_d$ are required to satisfy to guarantee the validity of Theorem \ref{thm:A} and of all other results in the article involving constants ``$I$'' and ``$J$''.
\end{rem}
\begin{rem}
In their works \cite{CL1,CL2}, Canning--Larson partially shift the focus from $\cH_{g,4}$ to the open substack $\cH^{\mathrm{nf}}_{g,4}\subseteq\cH_{g,4}$ parametrising those degree 4 maps $\theta\colon C\to L$ that do not factor as a composition of two maps of degree 2: they show that the Chow ring of $\cH^{\mathrm{nf}}_{g,4}$ is tautological in a stable range, and compute this Chow ring in a stable range. Up to the mild\footnote{This modification consists in removing a high codimension closed substack, so it does not affect Chow or cohomology groups in a stable range. See Example \ref{ex:ccHcH}.} modification imposing the presence of at least one simple branch value, this corresponds to our substack $\fF_{0,1}\ccH_{g,4}$. 
We hope that the splitting in Theorem \ref{thm:A}, though at the level of cohomology and not of Chow rings, may serve as evidence of the inherent obstacles that factoring maps give rise to; in particular, it would be  interesting if there were a splitting as in Theorem \ref{thm:A} for rational Chow rings, or even better, a splitting at the motivic level.
\end{rem}
Our second result expresses each $(h,e)$-summand in Theorem \ref{thm:A} in terms of the  stack $\cH_{h,e}$ and of factorisation homology with coefficients in the rational commutative algebra $\cA(d/e)$: the latter is obtained from the group algebra $\bQ[\fS_{d/e}]$ by passing to a certain associated graded and then taking the $\fS_{d/e}$-invariants under conjugation; see Definition \ref{defn:cA} for details.
Recall that any projective curve $C$ of genus $h$ has an underlying topological surface $C\tp$ with underlying homotopy type $|C\tp|$, and we may compute the factorisation homology $\int_{|C\tp|}\cA(d/e)$; as $\cA(d/e)$ is naturally a commutative algebra in the category $\Sp_\bQ^\bN$ of graded rational spectra, $\int_{|C\tp|}\cA(d/e)$ is naturally a graded rational spectrum, so we may select its part $(\int_{|C\tp|}\cA(d/e))_n$ of grading $n$ for any $n\in\bN$. 

More generally, if $\fp\colon \cE\to \cB$ is a family of genus $h$ curves over a stack $\cB$ over $\C$, classified by a map $\cB\to\fM_h$ to the moduli stack of smooth complex curves of genus $h$, we may compute the factorisation homology over each curve in the family, and assemble them into a parametrised graded rational spectrum $\int_{|\fp\tp|}\cA(d/e)$ over the underlying homotopy type $|\cB\tp|$ of the topological stack $\cB\tp$ associated with $\cB$; in other words, $\int_{|\fp\tp|}\cA(d/e)$ is a functor from $|\cB\tp|$ to $\Sp_\bQ^\bN$.
\begin{nota}
\label{nota:upsilon}
We abbreviate $\upsilon(g,l,h)=g-1-(h-1)l$.
\end{nota}
\begin{nota}
\label{nota:universal_curve}
We denote by $\fp_h\colon\cC_h\to\fM_h$ the universal genus $h$ curve, and by
$\fp_{h,e}\colon\cC_{h,e}\to\cH_{h,e}$ its pullback along the forgetful map $\epsilon_{h,e}\colon\cH_{h,e}\to\fM_h$ sending $(\theta\colon C\to L)\mapsto C$.
\end{nota}
\begin{athm}
\label{thm:B}
Let $d\ge2$ and let $1\le e<d$ be a divisor of $d$. Then there exist constants $I_{d/e},J_{d/e}>0$, only depending on $d/e$, such that for all $g\ge1$ and $h\ge0$ there is a map of rational spectra
\[
|\cH_{h,e}\tp|_*\pa{\int_{|\fp_{h,e}\tp|}\cA(d/e)}_{2\upsilon(g,d/e,h)}\to |\fF_{h,e}\ccH_{g,d}\tp|_*\bQ
\]
inducing an isomorphism on homotopy groups $\pi_*$ in degrees 
\[
*>-(2\upsilon(g,d/e,h)-J_{d/e})/I_{d/e}.
\]
In particular we have an isomorphism
\[
H^*\pa{|\cH_{h,e}\tp|;\int_{|\fp_{h,e}\tp|}\cA(d/e)}_{2\upsilon(g,d/e,h)}\cong H^*(|\fF_{h,e}\ccH_{g,d}\tp|;\bQ)
\]
in cohomological degrees $*<(2\upsilon(g,d/e,h)-J_{d/e})/I_{d/e}$.
\end{athm}
\begin{acor}
\label{cor:C}
Combining Theorems \ref{thm:A} and \ref{thm:B} we obtain the following isomorphism of $\bQ$-vector spaces, holding in degrees $*<(2(g+d-1)-J_d)/I_d$:
\begin{equation}
\label{eq:put_together}
H^*(|\cH_{g,d}\tp|;\bQ)\cong\bigoplus_{\overset{h\ge0}{1\le e<d,\ e\mid d}} H^{*-4\rho(h,e,d)}\pa{|\cH_{h,e}\tp|;\int_{|\fp_{h,e}\tp|}\cA(d/e)}_{2\upsilon(g,d/e,h)}.
\end{equation}
\end{acor}
Corollary \ref{cor:C}  expresses the \emph{stable} cohomology of $|\cH_{g,d}\tp|$, for $g\to\infty$, in terms of the \emph{unstable} cohomologies of all homotopy types $|\cH_{h,e}\tp|$ (with suitable local coefficients), for arbitrary $h\ge0$ and all divisors $e\mid d$  \emph{with $e$ strictly smaller than $d$}. When $d$ is a prime, the only allowed value of $e$ is 1, and we obtain the following simplified stable isomorphism
\[
H^*(|\cH_{g,d}\tp|;\bQ)\cong H^*\pa{|\cH_{0,1}\tp|;\int_{|\fp_{0,1}\tp|}\cA(d)}_{2(g+d-1)}
\]
where we use that $\cH_{h,1}=\emptyset$ for $h\ge1$. Here $\cH_{0,1}\cong\fM_0\cong\bbB\Aut(\bP^1)$ is the moduli stack of genus 0 curves, and we have explicitly $|\cH_{0,1}\tp|\simeq\bbB\SO(3)$.
\begin{rem}
\label{rem:ring_structures?}
The stable cohomology isomorphisms provided in Corollary \ref{cor:C} are isomorphisms of $\bQ$-vector spaces, but are not multiplicative. It would be interesting to understand the ring structure on $H^*(|\cH_{g,d}\tp|;\bQ)$, and to determine whether also the ring structure stabilises for $g\to\infty$.
\end{rem}
Corollary \ref{cor:C} shows that the rational cohomology of $|\cH_{g,n}\tp|$ agrees, in a range increasing with $g$, with the direct sum of certain cohomology groups \emph{also depending on $g$}; in fact $g$ enters the constant $2\upsilon(g,d/e,h)$ providing the selected grading of factorisation homology. In order to unambiguously talk about ``rational cohomological stability'' in $g$ for the family of homotopy types $|\cH_{g,d}\tp|$, we need to prove that the direct sum on the right hand side of \eqref{eq:put_together} indeed stabilises for $g\to\infty$. 
For $h\ge0$ and $l\ge2$ we shall introduce in Definition \ref{defn:xiclass} a specific element
\[
\xi_{l,h}\in H^0\pa{|\fM_h\tp|;\int_{|\fp_h\tp|}\cA(l)}_1,
\]
for which we prove the following result.
\begin{athm}
\label{thm:D}
For $h\ge0$, $c\ge2$ and for all $n\ge0$, the map
\[
\xi_{l,h}\cdot-\colon \pa{\int_{|\fp_h\tp|}\cA(l)}_n\to  \pa{\int_{|\fp_h\tp|}\cA(l)}_{n+1}
\]
of rational parametrised spectra over $|\fM_h\tp|$
is a fibrewise equivalence on $\pi_*$ for $*>-n$. Restricting along $|\epsilon_{e,h}\tp|$ and computing cohomology, we obtain an isomorphism
\[ H^*\pa{|\cH_{h,e}\tp|;\int_{|\fp_{h,e}\tp|}\cA(l)}_n\cong
H^*\pa{|\cH_{h,e}\tp|;\int_{|\fp_{h,e}\tp|}\cA(l)}_{n+1}
\]
in cohomological degrees $*<n$, for all $e\ge1$.
\end{athm}

We next study in detail the cases $2\le d\le 5$. By Theorems \ref{thm:A} and \ref{thm:B}, for these values of $d$ the computation of the stable rational cohomology of $|\cH_{g,d}\tp|$ reduces to the computation of the following types of twisted cohomology groups:
\[
H^*\pa{|\cH_{h,2}\tp|;\int_{|\fp_{h,2}\tp|}\cA(l)}\text{ for }l=2;\quad H^*\pa{|\cH_{0,1}\tp|;\int_{|\fp_{0,1}\tp|}\cA(l)}\text{ for }2\le l\le 5. 
\]
We provide a further identification of the previous  cohomology groups in terms of configuration spaces. For this, let $M_2=\set{0,1}$ and $M_3=\set{0,1,2}$, considered as subsets of $\bN$, and let $M_4$ and $M_5$ be the following subsets of $\bN\times\bN$:
\begin{itemize}
\item $(a,b)\in M_4$ if and only if $ab=0$, $a\le 3$ and $b\le1$;
\item $(a,b)\in M_5$ if and only if $a\le 4$, $b\le 1$, and moreover $a\le1$ whenever $b=1$.
\end{itemize}
For $2\le l\le 5$ we consider $M_l$ as a partial abelian monoid, and for an oriented closed surface $\Sigma$ of genus $h\ge0$  we introduce the space $\Conf(\Sigma;M_l)$ of finite configurations of points in $\Sigma$ with partially summable labels in $M_l$. It splits as a disjoint union $\coprod_{a\ge0}\Conf(\Sigma,M_l)_a$ for $l=2,3$, and as a disjoint union $\coprod_{a,b\ge0}\Conf(\Sigma,M_l)_{a,b}$
for $l=4,5$.
There is a natural action of $\Homeo^+(\Sigma)$ on $\Conf(\Sigma;M_l)$. 
\begin{athm}
\label{thm:E}
For an oriented closed surface $\Sigma$ of genus $h\ge0$ and for $n\ge0$ there is a $\Homeo^+(\Sigma)$-equivariant equivalence of rational spectra
\[
\begin{split}
\pa{\int_{|\Sigma|}\cA(l)}_n&\simeq|\Conf(\Sigma;M_l)_n|_*\bQ\quad\text{ for }l=2,3;\\
\pa{\int_{|\Sigma|}\cA(l)}_n&\simeq\bigoplus_{a+2b=n}|\Conf(\Sigma;M_l)_{a,b}|_*\bQ\ [-2b]\quad\text{ for }l=4,5.
\end{split}
\]
\end{athm}
We may now combine Corollary \ref{cor:C} and Theorem 
\ref{thm:E} with the equivalence $|\fM_h\tp|\simeq\bbB\Homeo^+(\Sigma)$, where $\Sigma$ is a closed oriented surface of genus $h$. We obtain for $2\le d\le 5$ a stable splitting of $H^*(|\cH_{g,d}\tp|;\bQ)$ in terms of the topological stacks $\Conf(\fp_h\tp,M_l)$ and $\Conf(\fp_{h,e}\tp,M_l)$
obtained from the maps $\fp_h$ and $\fp_{h,e}$ from Notation \ref{nota:universal_curve} by taking $M_l$-labeled configuration spaces fibrewise. The underlying  homotopy types of these topological stacks are the following:
\[
\begin{split}
|\Conf(\fp_h\tp,M_l)|&\simeq|\Conf(\Sigma,M_l)|\sslash\Homeo^+(\Sigma);\\
|\Conf(\fp_{h,e}\tp,M_l)|&\simeq |\Conf(\fp_h\tp,M_l)|\times_{|\fM_h\tp|}|\cH_{h,e}\tp|.
\end{split}
\]
\begin{acor}
\label{cor:F}
Let $g\ge0$; for $d=2,3,5$ we have stable isomorphisms of graded $\bQ$-vector spaces as follows:
\[
\begin{split}
H^*(|\cH_{g,2}\tp|;\bQ)&\cong H^*(|\Conf(\fp_0\tp,M_2)_{2g+2}|;\bQ);\\
H^*(|\cH_{g,3}\tp|;\bQ)&\cong H^*(|\Conf(\fp_0\tp,M_3)_{2g+4}|;\bQ);\\
H^*(|\cH_{g,5}\tp|;\bQ)&\cong \bigoplus_{a+2b=2g+8}H^{*-2b}(|\Conf(\fp_0\tp,M_5)_{a,b}|;\bQ),
\end{split}
\]
whereas for $d=4$ we have that $H^*(|\cH_{g,4}\tp|;\bQ)$ is stably isomorphic to
\[
\begin{split}
&\bigoplus_{h\ge0} H^{*-4(h+1)}(|\Conf(\fp_{h,2}\tp,M_2)_{2g+2-4h}|;\bQ)\\
\oplus& \bigoplus_{a+2b=2g+6}H^{*-2b}(|\Conf(\fp_0\tp,M_4)_{a,b}|;\bQ).
\end{split}
\]
Precisely, the stable isomorphism associated with $d=2,3,4,5$ holds in the range of cohomological degrees $*<(2(g+d-1)-J_d)/I_d$, .
\end{acor}
\begin{rem}
\label{rem:hyperelliptic}
For $h\ge2$ we remark that the stack $\cH_{h,2}$ can be identified with the hyperelliptic locus inside $\fM_h$ along the map $\epsilon_{h,2}\colon\cH_{h,2}\to\fM_h$; in particular this is an equivalence for $h=2$. For $h\le1$ the map $\epsilon_{h,2}$ has positive-dimensional fibres.

For arbitrary $l$ we also remark that $|\Conf(\fp_0\tp,M_l)|$ is equivalent to the homotopy quotient $|\Conf(S^2,M_l)|\sslash\Homeo^+(S^2)\simeq|\Conf(S^2,M_l)|\sslash\SO(3)$.
\end{rem}
\begin{rem}
\label{rem:Zheng}
For $d=2$ it is a classical result that $H^*(|\cH_{g,2}\tp|)\cong \bQ$ for any $g\ge1$, in particular we know the entire cohomology ring. For $d=3$, Zheng \cite{Zheng} proved that $H^*(|\cH_{g,3}\tp|)$ agrees with $\bQ[\kappa_1]/\kappa_1^3$ in degrees $*\le \lfloor g/4\rfloor$; here we denote by $\kappa_1$ the pullback of the first Mumford--Morita--Miller class along $|\epsilon_{g,3}\tp|\colon|\cH_{g,3}\tp|\to|\fM_g\tp|$. Corollary \ref{cor:F} partially recovers these results (see Subsection \ref{subsec:oddexamples}), but only in the stable ranges provided by the constants $I_d,J_d$, and only at the level of vector spaces: the multiplicative structure, as well as the the identification of the generators as powers of $\kappa_1$, can probably be recovered too with some additional effort. In particular, for $d=3$, Zheng's range $*<\lfloor g/4\rfloor$ is better than the current best available range relying on the work of Landesman--Levy.
\end{rem}
\begin{rem}
For $l\le5$ the algebra $\cA(l)$ can be generated with at most two elements and it admits a monomial presentation; this allows us to make a fairly simple analysis of its factorisation homology over a surface $\Sigma$ and obtain Theorem \ref{thm:D}. A complete presentation of $\cA(l)$ is known for $l\le10$ \cite{BCP:Hilb}; for $6\le l\le10$ we have that $\cA(l)$ is minimally generated by at least three elements and admits no monomial presentation, making the analysis of its factorisation homology harder.
\end{rem}
The cohomology groups from Corollary \ref{cor:F} belong to the classical realm of configuration spaces.
More precisely, for $n\ge0$ and $h\ge2$, $\Conf(\fp_{h,2}\tp,M_2)_n$ is the underlying topological stack of the algebraic stack of hyperelliptic curves of genus $h$ with $n$ unordered marked points; its cohomology, at least for stable $h$ and $n$, should be explicitly computable using the results of \cite{BDPW,MPPRW}. 
The other cohomology groups occurring in Corollary \ref{cor:F}, namely those related to $\fp_0$, are instead the $\SO(3)$ equivariant cohomology groups of configuration spaces of points on $S^2$ with prescribed upper bounds on multiplicities.
In this case, a generalisation of the results in the literature \cite{Segal, Vassiliev,CCMM,GKY1,GKY2,Kallel,Kamiyama} could even allow a complete computation. We content ourselves in Subsection \ref{subsec:oddexamples} with a couple of examples showing
in particular that non-trivial
odd-dimensional cohomology classes in $H^*(|\cH_{g,d}\tp|;\bQ)$ exist, for $g$ large enough, for both $d=4,5$. It would be interesting to compute explicitly all cohomology groups in Corollary \ref{cor:F}.
\begin{rem}
The presence of odd-dimensional classes implies that for $d=4,5$, the stable rational cohomology of $|\cH_{g,d}\tp|$ is not fully algebraic, i.e., not in the image of the cycle class map from the Chow ring. This is in stark contrast with the cases $d=2,3$, as seen in Remark \ref{rem:Zheng}, as well as with the situation for $\fM_g$, in which the stable rational cohomology is fully tautological \cite{MadsenWeiss}.

\end{rem}

We conclude the article with a possible strategy, explained in Subsection \ref{subsec:Picard}, to compute the rational Picard group of $\cH_{g,d}$ for all $d\ge4$ and for $g$ large enough (with respect to $d$), to be given by $
\Pic(\cH_{g,d})\otimes\bQ\cong\bQ^2$.
This strategy was informed by a discussion with Aaron Landesman, to whom I am thankful.

\subsection*{Conventions}
Throughout the article, for an integer $l\ge2$ we denote by $\fS_d$ the symmetric group of permutations of the set $\udll=\set{1,\dots,l}$. We usually denote by $\lambda\vdash l$ a partition of $l\ge2$ as a sum of unordered positive integers. The datum of $\lambda$ also determines a conjugacy class in $\fS_l$, spanned by permutations $\sigma$ with cycle type $\lambda$; we write $\sigma\in\lambda$ for such permutations.

We phrased the main problem in terms of algebraic (or Artin) stacks, but mostly work with their corresponding topological stacks and associated underlying homotopy types.
By ``space'' we mean a topological space, whereas by ``homotopy type'' we mean a topological space up to homotopy equivalent replacement.
All schemes and algebraic stacks will be over $\C$. By ``curve'' we usually mean a smooth, connected, projective 1-dimensional variety over $\C$. By ``surface'', in contrast, we usually mean a topological manifold, possibly with boundary, of real dimension 2.

The one point compactification of a space $X$ is denoted $X^\infty$, and the adjoined point at infinity is denoted $\infty$.
We denote by $|X|$ the underlying homotopy type of the topological space or topological stack $X$. 

We denote by $\Sp_\bQ$ the $\infty$-category of rational spectra, whose homotopy category is the classical derived category of rational chain complexes.
For a homotopy type $Y$ we denote by $\Sp^Y_\bQ=\Fun(Y,\Sp_\bQ)$ the $\infty$-category of $\Sp_\bQ$-valued local coefficient systems on $Y$, and we denote by
$Y_*,Y_!\colon\Sp_\bQ^Y\to\Sp_\bQ$ the limit and colimit functor, respectively, also known as ``cochains'' and ``chains''. Similarly, if $\iota\colon Z\to Y$ is a map of homotopy types, then $\iota^*\colon\Sp^Y_\bQ\to\Sp^Z_\bQ$ denotes the restriction functor, and
$(Z,Y)_!=\mathrm{cofib}(Y_!\iota^*\to Z_!)$ denotes the ``relative chains'' functor.

\subsection*{Acknowledgments}
This article was born as a result of fruitful discussions with people met on the occasion of two conferences, having taken place in Cetraro and in Cortona in June 2026. I'm thankful to the organisers of those conferences for providing such a fantastic environment to discuss and learn new mathematics: Gabriel Farkas, Margherita Lelli Chiesa; Karl Christ, Ciro Ciliberto,
Sam Molcho, Edoardo Sernesi,
Sara Torelli. A special thank goes to Aaron Landesman, for many fruitful discussions as well as detailed corrections on a first draft of this article.
I would then also like to thank Luca Battistella, Samir Canning, Simon Gritschacher, Hannah Larson, Ishan Levy, Tommaso Rossi and Angelina Zheng for fruitful discussions. 

Basic AI chatbots that are freely available online have been used for learning, drawing pictures, and searching the literature; the novel mathematical content of the article is solely the fruit of the work of the author.

\tableofcontents
\section{Hurwitz stacks and topological branched covers}
\label{sec:cHgd}

\subsection{Algebraic and topological Hurwitz stacks}

In the following definition we generalise the algebraic stack $\cH_{g,d}$ from the introduction.
\begin{defn}
\label{defn:algHurstack}
Recall Notation \ref{nota:universal_curve}.
For $g,h\ge0$ and $l\ge1$ we denote by
$\cH_{g,l}(\fp_h)$
the moduli stack of degree $l$ maps $\theta\colon C\to D$ between smooth connected complex projective curves, with $C$ of genus $g$ and $D$ of genus $h$. When $h=0$ we abbreviate $\cH_{g,l}:=\cH_{g,l}(\fp_0)$, as in the introduction.

We denote by $q_{g,l,h}\colon \cH_{g,l}(\fp_h)\to\fM_h$ the forgetful map sending $(\theta\colon C\to D)\mapsto D$.
It is smooth, and its fibre over a curve $D$ of genus $h$ is the moduli stack $\cH_{g,l}(D)$ of genus $g$ degree $l$ covers of $D$.

For a family $\fp\colon \cE\to\cB$ of curves of genus $h$ over a stack $\cB$, classified by a map $\epsilon\colon \cB\to \fM_h$, we denote by $\cH_{g,l}(\fp)$ the fibre product
\[
\cH_{g,l}(\fp)=\cH_{g,l}(\fp_h)\times_{\fM_h} \cB,
\]
computed along the maps $\epsilon$ and $q_{g,l,h}$.
\end{defn}

\begin{ex}
\label{ex:cHglhe}
For $e\ge1$ and $h\ge0$ we may consider, as in Notation \ref{nota:universal_curve}, the map $\fp_{h,e}\colon\cC_{h,e}\to\cH_{h,e}$ classified by the map $\epsilon_{h,e}\colon\cH_{h,e}\to\fM_h$. For $g\ge0$ and $l\ge1$ we then have that $\cH_{g,l}(\fp_{h,e})$ is the moduli stack of sequences of maps $C\to D\to L$ of curves of genera $g,h,0$ respectively, with $C\to D$ and $D\to L$ of degrees $l$ and $e$ respectively.
\end{ex}
\begin{nota}
\label{nota:qglhe}
In the setting of Example \ref{ex:cHglhe} we denote by $q_{g,l,h,e}\colon\cH_{g,l}(\fp_{h,e})\to\cH_{h,e}$ the map sending $(C\to D\to L)\mapsto (D\to L)$.
\end{nota}
\begin{rem}
\label{rem:cHgdDMstack}
The case $l=1$ is degenerate and only included in Definition \ref{defn:algHurstack} for uniform notation: the stack $\cH_{g,1}(\fp_h)$ is non-empty if and only if $g=h$, in which case it is isomorphic to $\fM_h$ along $q_{g,1,h}$.

For $l\ge2$, the stack $\cH_{g,l}(\fp_h)$ is non-empty if and only if the number $\upsilon(g,l,h)$ from Notation \ref{nota:upsilon} is $\ge0$. In this case $2\upsilon(g,l,h)$ counts, with multiplicity, the number of branch values of a degree $l$ map $\theta\colon C\to D$ from a genus $g$ curve to a genus $h$ curve.

We also observe that a degree $l$ map $\theta\colon C\to D$ between curves of genera $g$ and $h$ admits finitely many automorphisms unless either $g=h\le1$. Consequently, $\cH_{g,l}(\fp_h)$ is a Deligne--Mumford stack whenever $g\ge2$, or $g=1$ and $h=0$.

For a curve $D$ of genus $h$, the stack $\cH_{g,l}(D)$ is a (possibly empty) Deligne--Mumford stack for all $g\ge0$ and $l\ge1$.
\end{rem}
\begin{rem}
This remark is the fruit of a conversation with Aaron Landesman.
If $d$ is a prime number and $g>(d-1)^2$, then the map $\epsilon_{g,d}\colon\cH_{g,d}\to\fM_g$ is an inclusion, with image the locus of $d$-gonal curves, i.e. those curves of genus $g$ admitting some degree $d$ map to a genus 0 curve, without a specific choice of this datum. In other words, a degree $d$ map $\theta\colon C\to \bP^1$, if it exists, is unique up to automorphism of $\bP^1$.

To see this, we observe that if $\theta_1,\theta_2\colon C\to \bP^1$ are two maps of degree $d$, they combine into a map $\theta_1\times\theta_2\colon C\to \bP^1\times\bP^1$. The image of this map is an irreducible, but possibly singular curve $C'\subseteq\bP^1\times\bP^1$ of bidegree $(e,e)$ for some divisor $e\mid d$, and the map $\theta_1\times\theta_2\colon C\to C'$ has degree $d/e$. If $e=1$ then $C'$ is the graph of an automorphism of $\bP^1$, which means that $\theta_1$ and $\theta_2$ differ by an automorphism of $\bP^1$. If $e=d$ then $C$ is the normalisation of $C'$, and $C'$, being a (possibly singular) curve of bidegree $(d,d)$ in $\bP^1\times\bP^1$, has genus $\le(d-1)^2$.
\end{rem}
We next recast the study of \'etale cohomology of the algebraic stack $\cH_{g,d}$ in terms of topological stacks and their underlying homotopy types. Recall that Noohi 
\cite[Section 20]{Noohi05} associates with every algebraic stack $\cB$ locally of finite type over $\C$ a topological stack $\cB\tp$; examples of such $\cB$ are the Hurwitz stacks from Definition \ref{defn:algHurstack}, . Concretely, given a groupoid presentation $R\rightrightarrows U$ of $\cB$, i.e. $U\to\cB$ is a smooth atlas and $R=U\times_\cB U$, we extract the underlying analytic spaces $R\tp\rightrightarrows U\tp$ and let $\cB\tp$ be the quotient topological stack. Also following Noohi \cite{Noohi12}, with any topological stack $X$ we can associate an underlying homotopy type $|X|$ (this is referred to as ``classifying space'' in Noohi's work). Concretely, we present $X$ by a groupoid object in spaces $R\rightrightarrows U$, form the simplicial topological space having space of $p$-simplices the $p$-fold iterated fibre product $R\times_UR\times_U\dots\times_UR$, pass to the corresponding simplicial homotopy type, having $|R\times_UR\times_U\dots\times_UR|\simeq |R|\times_{|U|}|R|\times_{|U|}\dots\times_{|U|}|R|$ as homotopy type of $p$-simplices, and compute the geometric realisation of the latter.

The comparison theorem \cite[Théorème 4.1]{SGA4IIIxvi} (see \cite[Corollary 5.14]{Carchedi} for a result covering the case of algebraic stacks) asserts that if $\cB$ is an algebraic stack locally of finite type over $\C$, and $\ell$ is a prime number, then there is an isomorphism between the $\ell$-adice \'etale cohomology of $\cB$ and the singular cohomology of $|\cB\tp|$ with coefficients in $\bQ_\ell$:
\[
H^*_{\text{\'et}}(\cB;\bQ_\ell)\cong H^*(|\cB\tp|;\bQ_\ell);
\]
if $\cB$ is actually of finite type, then all (integral) cohomology groups of $|\cB\tp|$ are finitely generated, and therefore $H^*(|\cB\tp|;\bQ_\ell)$ is actually isomorphic to the tensor product $H^*(|\cB\tp|;\bQ)\otimes_\bQ\bQ_\ell$. This may be applied for instance to $\cB=\cH_{g,d}$.
\begin{nota}
For an oriented closed topological surface $\Sigma$, i.e. a compact 2-dimensional topological manifold without boundary, we denote by $\Homeo^+(\Sigma)$ the topological group of orientation-preserving homeomorphisms of $\Sigma$, endowed with the compact-open topology.
\end{nota}

\begin{ex}
To see how this double step sheds topological light into our problem, let us first remind what happens for the algebraic stack $\fM_g$ for $g\ge0$. The topological stack $\fM_g\tp$ associates with a topological space the groupoid $\fM_g\tp(X)$ of smooth oriented fibre bundles $p\colon C\to X$ with fibre a (real) surface of genus $g$, together with a continuous choice of Riemann structures on fibres; morphisms are given by bundle isomorphisms that are fibrewise holomorphic. If we further pass to the underlying homotopy type $|\fM_g\tp|$, we obtain the homotopy type $\bbB\Homeo^+(\Sigma)$ of the classifying space of $\Homeo^+(\Sigma)$, where $\Sigma$ is a closed oriented topological surface of genus $g$; this is also a classifying space for fibre bundles with genus $g$ oriented surfaces as fibres.
\end{ex}

\subsection{Topological branched covers}
In this subsection we start the analysis of the underlying topological stack and homotopy type of $\cH_{g,l}(\fp_h)$, and thereby introduce the topological stacks of topological branched covers from \cite{DasPetersen}.

We start by observing that the topological stack $\cH_{g,l}(\fp_h)\tp$ associates with a space $X$ the groupoid with the following objects and morphisms:
\begin{itemize}
\item an object in $\cH_{g,l}(\fp_h)\tp(X)$ is the datum of two smooth oriented fibre bundles $C\to X$ and $D\to X$ with fibres real surfaces of genera $g$ and $h$, respectively; moreover, fibres are continuously endowed with Riemann structures; moreover, a map $\theta\colon C\to D$ over $X$ is given, and $\theta$ is fibrewise holomorphic and of degree $c$;
\item a morphism in $\cH_{g,l}(\fp_h)\tp(X)$ is  a commutative diagram over $X$ as follows, in which all maps are fibrewise holomorphic:
\[
\begin{tikzcd}
C\ar[r,"\theta"]\ar[d,"\cong"]&D\ar[d,"\cong"]\\
C'\ar[r,"\theta'"]&D'.\\
\end{tikzcd}
\]
\end{itemize}
The description of the underlying homotopy type $|\cH_{g,l}(\fp_h)\tp|$ is best given in terms of the map $q_{g,l,h}$ from Definition \ref{defn:algHurstack}. The fibre of $q_{g,l,h}$ over $D\in\fM_h$ is $\cH_{g,l}(D)$. Passing to topological stacks, we obtain a smooth map $q_{g,l,h}\tp\colon\cH_{g,l}(\fp_h)\tp\to\fM_h\tp$ with fibre $\cH_{g,l}(D)\tp$.

Consider now the topological surface $\Sigma$ underlying $D\in\fM_h$, i.e. $\Sigma=D\tp$ is the analytic space associated with the complex variety $D$. Then we may identify the topological stack $\cH_{g,l}(D)\tp$ with a substack of the topological stack $\Bra^l_\Sigma$ from \cite[§ 4]{DasPetersen}, whose definition we recall here in a generality which is suitable for the purposes of this article.
\begin{defn}
\label{defn:Bra}
Let $\Sigma$ be a closed connected oriented topological surface and let 
$A\subseteq B\subseteq \Sigma$ be finite subsets, with $\#A\le1$.
Let $l\ge2$. The topological stack 
$\Bra^l_{\Sigma\rel A\et B}$ 
associates with a space $X$ the groupoid whose objects are sequences $(Y,p,\theta,\bt)$ where $Y$ is a space, $p\colon Y\to X$ and $\theta\colon Y\to\Sigma$ are maps, and $\bt\colon X\times A\times\udll\cong \theta^{-1}(A)$ is a homeomorphism, satisfying the following properties:
\begin{itemize}
\item $p$ is a fibre bundle whose fibres are possibly disconnected closed surfaces;
\item for each $x\in X$ the restricted map $\theta\colon p^{-1}(x)\to \Sigma$ is a topological $l$-fold branched cover; that is, it is proper \'etale of degree $l$ away from a finite subset of $p^{-1}(x)$ of \emph{branch points}, and it can be modeled up to local homeomorphism by the map $\C\to\C$, $z\mapsto z^r$ near the branch points for some $2\le r\le l$; we require moreover that $p^{-1}(x)\to\Sigma$ is \'etale over $B$;
\item the composition $p\circ\bt\colon X\times A\times\udll\to X$
agrees with the product projection.
\end{itemize}
Morphisms in $\Bra_{\Sigma\rel A\et B}^l(X)$ are homeomorphisms $Y\cong Y'$ over $X\times \Sigma$ and under $X\times A\times\udll$.
When $A=\emptyset$ we remove $A$ and $\bt$ from the notation, whereas when $A=\set{a}$ we write $\rel a$. Similarly, when $A=B$ we remove $B$ from the notation.

We denote by $\CBra^l_{\Sigma\rel A\et B}\subseteq\Bra^l_{\Sigma\rel A\et B}$ the topological substack spanned at $X$ by those sequences $(Y,p,\theta,\bt)$ such that $p$ has connected fibres.
For $g\ge0$ we denote by $\CBra^{g,l}_{\Sigma\rel A\et B}\subseteq\CBra^l_{\Sigma\rel A\et B}$ the topological substack spanned at $X$ by those sequences $(Y,p,\theta,\bt)$ such that fibres of $p$ are connected genus $g$ surfaces.
\end{defn}
For a connected smooth complex projective curve $D$ with underlying topological surface $\Sigma$ of genus $h$,
we may now identify the topological stacks $\cH_{g,l}(D)\tp$ and $\CBra_\Sigma^{g,l}$: the crucial observation is that the Riemann structures on fibres of $p$ may be transfered up along $\theta$ from the Riemann structure on $D$, thereby simultaneously forcing $\theta$ to be fibrewise holomorphic and not only fibrewise a topological $l$-fold branched cover.

Now there is a natural action of the topological group $\Homeo^+(\Sigma)$ on the topological stack $\Bra^l_\Sigma$: loosely, given a space $X$, a map $s\colon X\to\Homeo^+(\Sigma)$ acts on a triple $(Y,p,\theta)\in\Bra^l_\Sigma(X)$ by sending it to the triple $(Y,p,\theta')$, where $\theta'\colon Y\to\Sigma$ is the composite $Y\xrightarrow{p\times\theta}X\times\Sigma\to\Sigma$, using the adjoint of $s$ as second map.
The substack $\CBra^{g,l}_\Sigma$ is $\Homeo^+(\Sigma)$-invariant, and the fibre sequence of homotopy types
\[
|\cH_{g,l}(D)\tp|\hookrightarrow|\cH_{g,l}(\fp_h)|\xrightarrow{|q_{g,l,h}\tp|}|\fM_h\tp|
\]
may be identified with the fibre sequence
\[
|\CBra^{g,l}_\Sigma|\hookrightarrow|\CBra^{g,l}_\Sigma|\sslash\Homeo^+(\Sigma)\to \bbB\Homeo^+(\Sigma).
\]
If $\epsilon\colon\cB\to\fM_h$ is a map of algebraic stacks, we may pull back the previous fibre sequence; in particular for $\epsilon_{h,e}\colon\cH_{h,e}\to\fM_h$ we obtain a fibre sequence
\[
|\CBra^{g,l}_\Sigma|\hookrightarrow|\cH_{g,l}(\fp_{h,e})\tp|\to|\cH_{h,e}\tp|.
\]

\subsection{Configurations with monodromy}
In this subsection, for a fixed 
connected closed oriented surface $\Sigma$ of genus $h$ and a point $a\in\Sigma$, we deepen our study of the topological stack $\Bra^l_{\Sigma\rel a}$ and its substacks $\CBra^{g,l}_{\Sigma\rel a}$.
Since this is a direct generalisation of constructions and arguments from \cite{Bianchi:Hur1,Bianchi:Hur2,Bianchi:Hur4,DasPetersen}, the discussion will be brief.

We start with the crucial observation that
the Deligne--Mumford topological stack $\Bra^{l}_{\Sigma\rel a}$ is a topological space.
To see this,
it suffices to show that the groupoid $\Bra^l_{\Sigma\rel a}(*)$ is discrete. Indeed, if $(Y,p,\theta,\bt)\in \CBra^{g,l}_{\Sigma\rel a}(*)$, then $Y$ is a closed surface and $\theta\colon Y\to\Sigma$ is a degree $l$ branched cover which is \'etale and trivialised over $a$. Such a cover may not admit non-trivial automorphisms. Compare also with \cite[§4.1]{DasPetersen}.
Our next goal is to explicitly describe the points and the topology of the space $\Bra^{l}_{\Sigma\rel a}$, in terms of monodromy data.
\begin{defn}
\label{defn:smallloop}
Given a finite subset $P\subseteq \Sma$, a \emph{small loop} is an element in $\pi_1(\SmP;a)$ whose conjugacy class is represented by an unbased (or free) simple loop $\gamma$ in $\SmP$ that spins once, clockwise\footnote{With respect to the orientation of $\Sigma$.}, around a single point of $P$.
\end{defn}

\begin{defn}
\label{defn:norm}
The \emph{norm} of a permutation $\sigma\in\fS_l$, also known as its \emph{absolute length}, is the minimum $r\ge0$ such that $\sigma$ may be written as a product of $r$ transpositions.
Equivalently, if $\sigma$ consists of $m$ cycles, then $N(\sigma)=l-m$. 
\end{defn}
We observe that the norm is conjugation invariant, satisfies the triangular inequality, and only vanishes at the identity permutation $\one\in\fS_l$.
\begin{nota}
\label{nota:fP}
We denote by $\fP(l)$ the set of all partitions $\lambda\vdash l$, i.e. the set of conjugacy classes in $\fS_l$.
For $\lambda\in\fP(l)$ we denote by $N(\lambda)=N(\sigma)$ for any $\sigma\in\lambda$. We let $\fP(l)_+\subseteq\fP(l)$ denote the subset of non-discrete partitions, or equivalently, the set of conjugacy class of non-identity elements in $\fS_l$.
\end{nota}
\begin{nota}
\label{nota:Nzphi}
Let $P\subseteq\Sma$ be a finite subset, let $\phi\colon\pi_1(\SmP;a)\to\fS_l$ be a group homomorphism, and let $z\in P$. We write $N(z,\phi):=N(\phi(\gamma))\ge0$ for any small loop $\gamma$ spinning around $z$. We observe that up to conjugation $\phi(\gamma)$ only depends on $z$, hence also $N(z,\phi)$ only depends on $z$. 
\end{nota}

Given a point $(Y,p,\theta,\bt)$ in $\CBra^{g,l}_{\Sigma\rel a}$, we may let $P\subseteq\Sma$ denote the finite set of branch values of $\theta\colon Y\to\Sigma$
and let $\phi\colon\pi_1(\SmP;a)\to\fS_l$ denote the monodromy group homomorphism, which is defined using the trivialisation $\bt\colon\theta^{-1}(a)\cong\udll$. 
Viceversa, let $(P,\phi)$ be a pair consisting of a finite subset $P\subseteq \Sma$ and a group homomorphism $\phi\colon\pi_1(\SmP;a)\to\fS_l$ satisfying the following property:
\begin{itemize}
\item each small loop is sent along $\phi$ to a non-trivial permutation.
\end{itemize}
Then we may recover $(Y,p,\theta,\bt)$ from $(P,\phi)$ up to unique isomorphism, and $P$ is indeed the set of branch values of $\theta$. 
The condition $(Y,p,\theta,\bt)\in\CBra^l_{\Sigma\rel a}$ translates into the following condition on $(P,\phi)$:
\begin{itemize}
\item the image of $\phi$ is a subgroup of $\fS_l$ acting transitively on $\udll$;
\end{itemize}
whereas the further condition 
$(Y,p,\theta,\bt)\in\CBra^{g,l}_{\Sigma\rel a}$ 
translates into the following further condition on $(P,\phi)$, using Notation \ref{nota:upsilon}:
\begin{itemize}
\item $\sum_{z\in P}N(z,\phi)=2\upsilon(g,l,h)$.
\end{itemize}

We next describe the topology on $\Bra^l_{\Sigma\rel a}$ from the point of view of pairs $(P,\phi)$ as above.
Intruitively, 
the topology makes sure that points in $P$ can freely move around each other in $\Sma$ while ``dragging'' the monodromy $\phi$; and they may even collide in clusters, but only under suitable ``geodesic'' circumstances.

\begin{defn}
\label{defn:fUsubset}
Let $(P,\phi)\in\Bra^l_{\Sigma\rel a}$, expand $P=\set{z_1,\dots,z_k}$, and let $\uU=(U_1,\dots,U_k)$ be open discs in $\Sma$ with disjoint closures, with $U_i$ intersecting $P$ only in $z_i$. We associate with the data $P,\phi,\uU$ the subset
\[
\mathfrak{U}(P,\phi,\uU)\subseteq\Bra^l_{\Sigma\rel a}
\]
consisting of all pairs $(P',\phi')$ satisfying the following properties:
\begin{enumerate}
\item $P'\subseteq U_1\cup\dots\cup U_k$ and $P'$ interesects each $U_i$ in at least one point;
\item the following composite group homomorphism
agrees with $\phi$:
\[
\pi_1(\SmP;a)\overset{\cong}{\leftarrow}\pi_1(\Sigma\setminus\uU;a)\to\pi_1(\SmP';a)\xrightarrow{\phi'}\fS_l;
\]
\item for all $1\le i\le k$ the following equality holds:
\[
N(z_i,\phi)=\sum_{z'\in P'\cap U_i}N(z',\phi').
\]
\end{enumerate}

\end{defn}

Note that the inequality ``$\le$'' always holds in (3) in Definition \ref{defn:fUsubset}, as a consequence of the triangular inequality enjoyed by the norm $N$. The sets $\fU(P,\phi,\uU)$ form an open basis for the topology of $\Bra^l_{\Sigma\rel a}$.

\subsection{Orbifold structure}
The goal of this subsection is to prove the following proposition. \begin{prop}
\label{prop:orbifold}
Let $\Sigma$ be a connected oriented closed surface of genus $h\ge0$, let $l\ge2$ and $g\ge0$, and let $A\subseteq B\subseteq\Sigma$ be finite subsets as in Definition \ref{defn:Bra}. Recall Notation \ref{nota:upsilon}.
Then the topological stack $\CBra^{g,l}_{\Sigma\rel A\et B}$ 
is
\begin{enumerate}
\item empty, if $\upsilon(g,l,h)<0$;
\item an oriented orbifold of real dimension $4\upsilon(g,l,h)$, if $\upsilon(g,l,h)\ge0$; 
\item a manifold, if moreover $A\neq\emptyset$.
\end{enumerate}
\end{prop}

\begin{proof}
Since $\CBra^{g,l}_{\Sigma\rel A\et B}$ is an open topological substack of $\CBra^{g,l}_{\Sigma\rel A}$, it suffices to consider the case $A=B$.

We first address point (1): if $\upsilon(g,l,h)=g-1-(h-1)l<0$, then the Riemann--Hurwitz formula prevents the existence of a degree $l$ branched cover map between oriented topological closed surfaces of genera $g$ and $h$. From now on we assume $\upsilon(g,l,h)\ge0$, and abbreviate $\upsilon=\upsilon(g,l,h)$.

We next address point (3): let $A=\set{a}$, let $(P,\phi)\in\Bra^l_{\Sigma\rel a}$, expand $P=\set{z_1,\dots,z_k}$ and let $\uU=(U_1,\dots,U_k)$ as in Definition \ref{defn:fUsubset}. Fix $k$ small loops $\gamma_1,\dots,\gamma_k\subseteq\SmU$ based at $a$, with $\gamma_i$ spinning clockwise around $U_i$. Let $\sigma_i=\phi(\gamma_i)\in\fS_l$, and consider $\sigma_i$ as an element of the partially multiplicative quandle $\fS_l\geo$ from \cite[Definition 7.1]{Bianchi:Hur1}, as well as an element of the multiplicative completion $\widehat{\fS_l\geo}$ of the latter. Then a straightforward generalisation of \cite[Proposition 5.1]{Bianchi:Hur2} identifies $\fU(P,\phi,\uU)$ up to homeomorphism with an open subset of the product $\prod_{i=1}^k\Hur^\Delta(\fS_l\geo)_{\sigma_i}$; here $\Hur^\Delta(\fS_l\geo)$ is the Hurwitz space with collisions from \cite[Definition 6.13]{Bianchi:Hur1}, and 
$\Hur^\Delta(\fS_l\geo)_{\sigma_i}$ is the connected component corresponding to $\sigma_i\in\widehat{\fS_l\geo}$. 
By \cite[Theorem B]{Bianchi:Hur4} we then have that $\Hur^\Delta(\fS_l\geo)_{\sigma_i}$ is naturally an open complex manifold of complex dimension $N(\sigma_i)=N(z_i,\phi)$, in particular it carries a canonical orientation. It follows that $\fU(P,\phi,\uU)$, and hence the entire $\CBra^{g,l}_{\Sigma\rel a}$, is naturally a
topological oriented manifold of real dimension $4\upsilon=2\sum_{i=1}^kN(z_i,\phi)$.

We finally address point (2) following \cite[§4.4]{DasPetersen}; it suffices to consider the case $A=\emptyset$.
For varying $a\in\Sigma$, the substacks $\CBra^{g,l}_{\Sigma\et a}$ give an open cover of $\CBra^{g,l}_\Sigma$. We may now consider the action of $\fS_l$ on the manifold $\CBra^{g,l}_{\Sigma\rel a}$ given as follows: a permutation $\sigma\in \fS_l$ sends a point $(Y,\theta,\bt)\in\CBra^{g,l}_{\Sigma\rel a}$ to $(Y,p,\theta,\sigma\circ\bt)$. The action is orientation preserving, and the quotient orbifold may be identified with $\CBra^{g,l}_{\Sigma\et a}$, which is in particular indeed an oriented orbifold.
\end{proof}
\begin{cor}
\label{cor:Braorbifold}
Let $\Sigma$ be a closed connected oriented surface of genus $h\ge0$, and let $g\ge0$ and $l\ge2$ be such that $\upsilon(g,l,h)\ge0$. Then the coarse space $\CBra^{g,l}_{\Sigma,\coarse}$ associated with the Deligne--Mumford topological stack $\CBra^{g,l}_\Sigma$ is an oriented rational homology manifold of dimension $4\upsilon(g,l,h)$.

In particular, by Poincar\'e--Lefschetz duality, we have a $\Homeo^+(\Sigma)$-equivariant equivalence of rational spectra
\[
|\CBra^{g,l}_{\Sigma,\coarse}|_*\bQ\simeq(|\CBra^{g,l,\infty}_{\Sigma,\coarse}|,\infty)_!\bQ\ [-4\upsilon(g,l,h)].
\]
\end{cor}
\begin{proof}
This follows directly from the fact that $\CBra^{g,l}_{\Sigma}$ is an oriented orbifold of said real dimension.
\end{proof}
A point in the space $\CBra^{g,l}_{\Sigma,\coarse}$ is an isomorphism class of $l$-fold branched covers $\theta\colon Y\to \Sigma$ with $Y$ connected of genus $g$. The group $\Homeo^+(\Sigma)$ acts on $\CBra^{g,l}_{\Sigma,\coarse}$ by postcomposing an equivalence class of maps $\theta$ with a homeomorphism of $\Sigma$, and the canonical map of homotopy types
\[
|\CBra^{g,l}_\Sigma|\sslash\Homeo^+(\Sigma)\to
|\CBra^{g,l}_{\Sigma,\coarse}|\sslash\Homeo^+(\Sigma)
\]
is a rational equivalence and compatible with projection to $|\fM_h\tp|\simeq\bbB\Homeo^+(\Sigma)$. Pulling back the projection
\begin{equation}
\label{eq:rvcrs}
|\CBra^{g,l}_{\Sigma,\coarse}|\sslash\Homeo^+(\Sigma)\xrightarrow{r_{g,l,h}}\bbB\Homeo^+(\Sigma)
\end{equation}
along the map $|\epsilon_{h,e}\tp|\colon|\cH_{h,e}\tp|\to|\fM_h\tp|$, we obtain a fibre sequence
\begin{equation}
\label{eq:cHvcrs}
|\CBra^{g,l}_{\Sigma,\coarse}|\hookrightarrow |\cH_{g,l}(\fp_{h,e})\tp_\vcoarse|\xrightarrow{r_{g,l,h,e}}|\cH_{h,e}\tp|,
\end{equation}
where $|\cH_{g,l}(\fp_{h,e})\tp_\vcoarse|$ is defined as the homotopy pullback
\[
|\cH_{g,l}(\fp_{h,e})\tp_\vcoarse|\coloneqq|(\CBra^{g,l}_\Sigma)_\coarse|\sslash\Homeo^+(\Sigma)
\times_{\bbB\Homeo^+(\Sigma)}|\cH_{h,e}\tp|;
\]
the ``v'' suggests that we are taking coarse spaces only in the ``vertical'' direction.

\begin{rem}
We expect $|\cH_{g,l}(\fp_{h,e})\tp_\vcoarse|$ to agree with the underlying homotopy type an algebraic stack ``$\cH_{g,l}(\fp_{h,e})_\vcoarse$'' of finite type over $\C$, parametrising degree $e$ maps $\theta\colon D\to L$ from a curve of genus $h$ to a curve of genus $0$, together with an isomorphism class of genus $g$, degree $l$ covers of $D$. We will not need this algebraic perspective in the following.
\end{rem}
By construction there is a canonical map
\[
|\cH_{g,l}(\fp_{h,e})\tp|\to|\cH_{g,l}(\fp_{h,e})\tp_\vcoarse|
\]
which is a rational equivalence. In particular, we have reduced the study of the rational cohomology of the algebraic stack $\cH_{g,l}(\fp_{h,e})$ to the study of the rational cohomology of the homotopy type $|\cH_{g,l}(\fp_{h,e})\tp_\vcoarse|$.

We may also consider the ``fibrewise one point compactification'' of
$|\cH_{g,l}(\fp_{h,e})\tp_\vcoarse|$, which we define as the following homotopy type:\footnote{We always first take coarse spaces and then one point compactifications.}
\[
|\cH_{g,l}(\fp_{h,e})\tpp_\vcoarse|\coloneqq|\CBra^{g,l,\infty}_{\Sigma,\coarse}|\sslash\Homeo^+(\Sigma)\times_{\bbB\Homeo^+(\Sigma)}|\cH_{h,e}\tp|.
\]
It fits in a fibre sequence as follows, endowed with a section at infinity which for simplicity we denote also $\infty$:
\begin{equation}
\label{eq:cHvcrsplus}
\begin{tikzcd}[column sep=40]
{|\CBra^{g,l}_{\Sigma,\coarse}|}\ar[r,hook]&{|\cH_{g,l}(\fp_{h,e})\tpp_\vcoarse|}\ar[r,"r^\infty_{g,l,h,e}"]&{|\cH_{h,e}\tp|}\ar[l,bend left, hook,"\infty"'].
\end{tikzcd}
\end{equation}

The Poincar\'e--Lefschetz equivalence from Corollary \ref{cor:Braorbifold} may be interpreted as an equivalence of parametrised spectra over $\bbB\Homeo^+(\Sigma)$. By restriction, and composing with the functor $|\cH_{h,e}\tp|_*$, we obtain equivalence of rational spectra
\[
|\cH_{g,l}(\fp_{h,e})\tp_\vcoarse|_*\bQ\simeq|\cH_{h,e}\tp|_*(r_{g,l,h,e})_*\bQ\simeq|\cH_{h,e}\tp|_*(r^\infty_{g,l,h,e},\infty)_!\bQ\ [-4\upsilon(g,l,h)]
\]
and in particular, taking homotopy groups, we obtain an isomorphism
\[
H^*(|\cH_{g,l}(\fp_{h,e})\tp_\vcoarse|;\bQ)\cong H^{*-4\upsilon(g,l,h)}\pa{|\cH_{h,e}\tp|\ ;\ (r^\infty_{g,l,h,e},\infty)_!\bQ}.
\]

\section{Factorisation homology}
\label{sec:facthomology}
In this section we introduce factorisation homology with coefficients in a commutative rational ring spectrum. We will focus on the commutative algebras $\cA(l)$, also appearing in the statement of Theorems \ref{thm:B}, \ref{thm:D} and \ref{thm:E}, as well as on the closely related commutative algebra $\cAh(l)$. The main result will be the following proposition.
\begin{prop}
\label{prop:defnTheta}
Let $\Sigma$ be a closed connected oriented topological surface of genus $h\ge0$, let $l\ge2$, and let $g\ge0$. Assume $\upsilon(g,l,h)\ge1$. Then there is a well-defined and $\Homeo^+(\Sigma)$-equivariant map of rational spectra
\[
\Theta^{g,l}_\Sigma\colon\pa{\int_{\Sigma}\cAh(l)}_{2\upsilon(g,l,h)}\to(|\CBra^{g,l,\infty}_{\Sigma,\coarse}|,\infty)_!\bQ.
\]
\end{prop}
\begin{rem}
\label{rem:upsilonzero}
If $\upsilon(g,l,h)=0$, then the source and target of the map $\Theta^{g,l}_\Sigma$ are canonically equivalent to $\bQ$: we then just let $\Theta^{g,l}_\Sigma$ be the identity of $\bQ$. If $\upsilon(g,l,h)<0$ then both source and target of $\Theta^{g,l}_\Sigma$ are zero: we then let $\Theta^{g,l}_\Sigma$ be the zero map.
\end{rem}

\subsection{Two models for factorisation homology}
\label{subsec:factmodels}
The goal of this subsection is to introduce two explicit chain complex models for factorisation homology. We start with the abstract $\infty$-categorical notion.
\begin{defn}
We denote by $\Fin$ the category of finite sets, endowed with the coproduct (disjoint union) symmetric monoidal structure; we let $\FS$ denote the subcategory of finite non-empty sets and surjective maps; $\FS$ lacks a monoidal unit, but we may still consider it as a non-unital symmetric monoidal subcategory of $\Fin$. For a homotopy type $Y$ we denote by $\Fin_{/Y}$ the $\infty$-category of finite sets with a map to $Y$, and by $\FS_{/Y}$ the fibre product of $\infty$-categories $\FS\times_{\Fin}\Fin_{/Y}$.
\end{defn}
Let $\cC$ be a presentably symmetric monoidal $\infty$-category, and let $R\in\CAlg(\cC)$ be a commutative algebra object, i.e. a symmetric monoidal functor $R\colon\Fin\to\cC$.
Let $Y$ be a connected homotopy type. The factorisation homology of $Y$ with coefficients in $R$, denoted $\int_YR\in\cC$, is defined as the colimit over $\Fin_{/Y}$ of the composite functor $\Fin_{/Y}\to\Fin\xrightarrow{R}\cC$. The main motivation for this definition is perhaps that $\int_YR$ agrees with the underlying object of the colimit of the constant functor $Y\to\CAlg(\cC)$ at $R$. 
\begin{ex}
Our motivating examples for $\cC$ are $\Sp_\bQ$, the $\infty$-category of rational spectra, and $\Sp_\bQ^\bN\coloneqq\Fun(\bN,\Sp_\bQ)$, the $\infty$-category of functors from $\bN$ to $\Sp_\bQ$, which we refer to as \emph{graded spectra}. Here $\bN$ denotes the symmetric monoidal groupoid whose objects are natural numbers, and whose only morphisms are identities; $\Sp_\bQ$ is endowed with the symmetric monoidal structure given by tensor product of spectra, and $\Sp_\bQ^\bN$ with the symmetric monoidal structure given by Day convolution. There is a symmetric monoidal functor $\Sp_\bQ^\bN\to\Sp_\bQ$ sending a graded spectrum $(\omega_n)_{n\in\bN}\in\Sp_\bQ^\bN$ to the direct sum $\bigoplus_{n\in\bN}\omega_n\in\Sp_\bQ$; we refer to it as the functor forgetting grading.
Our motivating examples for $R$ are the algebras $\cAh(l),\cA(l)\in\CAlg(\Sp_\bQ^\bN)$, that we shall introduce in Definitions \ref{defn:cAheart} and \ref{defn:cA}, as well as their images in $\CAlg(\Sp_\bQ)$ along the grading forgetful functor.
\end{ex}

\begin{rem}
\label{rem:FSreduction}
Since we assume that $Y$ is connected, the following hold:
\begin{itemize}
\item the canonical functor $\FS_{/Y}\to\Fin_{/Y}$ is final, and in particular $\int_YR$ agrees with the colimit over $\FS_{/Y}$ of the composite functor $\FS_{/Y}\to\Fin\xrightarrow{R}\cC$;
\item if $\cC$ is stable and $R\to\one_\cC$ is an augmentation, then we may split $R=\one_\cC\oplus R_+$, where $R_+$ is the augmentation ideal. We may consider $R_+\colon\FS\to\cC$ as a non-unital symmetric monoidal functor. Then the colimit over $\FS_{/Y}$ of the composite $\FS_{/Y}\to\FS\xrightarrow{R_+}\cC$ computes the fibre of the augmentation $\int_YR\to\int_Y\one_\cC\simeq \one_\cC$.
\end{itemize}
\end{rem}

We may therefore give the following working definition.
\begin{defn}
\label{defn:int}
Let $Y$ be a connected homotopy type and let $\I\colon\FS\to\cC$ be a non-unital, symmetric monoidal functor with target a presentably symmetric monoidal stable $\infty$-category $\cC$. We define
\[
\int_Y\I\coloneqq \colim(\FS_{/Y}\to\Fin\xrightarrow{\I}\cC)\in\cC.
\]
\end{defn}
Our next aim is to give a first, concrete chain model for $\int_{|X|}\I$ when $X$ and $\I$ satisfy the following:
\begin{itemize}
\item $X$ is a CW complex containing infinitely many points and whose underlying homotopy type $|X|$ is connected;
\item $\I$ is a non-unital commutative ring in $\bQ$-vector spaces (i.e., it is endowed with an associative and commutative product, possibly without a unit) considered as a non-unital commutative rational ring spectrum concentrated in homological degree 0.\footnote{The construction works as well if $\I$ is a non-unital commutative ring in abelian groups, but for our applications $\bQ$-linearity will be anyway guaranteed.}
\end{itemize}
Fix such $X$ and $\I$ for the remainder of the subsection. Our main application later will be with $X=\Sigma$ a topological surface, and $\I=\cAh(l)_+$ as in Definition \ref{defn:cAheart}.
\begin{nota}
\label{nota:skeletalFS}
To avoid set theoretic issues, as well as some non-canonical choices in the following, we model the category $\FS$ by its skeletal full subcategory whose objects are finite nonempty subsets of $X$.
\end{nota}
\begin{nota}
We denote by $\Delta^p=\set{(t_1,\dots,t_p)\in\R^{p+1}\,|\,t_i\ge0,\ \sum_it_i=1}$ the standard geometric $p$-simplex. For $0\le i\le p$ we denote by $\delta_i\colon\Delta^{p-1}\to\Delta^p$ the $i$\textsuperscript{th} face inclusion. In particular $\delta_0\colon(t_0,\dots,t_{p-1})\mapsto(0,t_0,\dots,t_{p-1})$.

For $p\ge0$ and a singular simplex $s\colon \Delta^p\to X$ we denote by $\alpha(s)=s(1,0,\dots,0)\in X$ the initial vertex.
\end{nota}

In the following definition we introduce a model of the $\infty$-category $\FS_{/|X|}$ as a quasi-category, i.e. a simplicial set with the inner horn filling condition.
\begin{defn}
\label{defn:SingFS}
We model the $\infty$-category $\FS_{/|X|}$ as the simplicial set $\Sing^\FS_\bullet(X)$ described as follows: for $p\ge0$, a $p$-simplex in $\Sing^\FS_p(X)$ is a choice of:
\begin{itemize}
\item a sequence $\udlS=(S_0\twoheadrightarrow \dots\twoheadrightarrow S_p)$ of nonempty finite sets as in Notation \ref{nota:skeletalFS} and surjective maps, i.e. a $p$-simplex in the nerve of $\FS$;
\item a continuous map $s\colon\Delta^p\to X^{S_0}$ such that, for all $1\le i\le p$, there exists a (necessarily unique) dashed map making the following diagram of spaces commute:
\[
\begin{tikzcd}[row sep=10]
\Delta^{p-i}\ar[d,"\delta_0^i"]\ar[r,dashed]&X^{S_i}\ar[d,hook,"X^{S_0\twoheadrightarrow S_i}"]\\
\Delta^p\ar[r,"s"]&X^{S_0}.
\end{tikzcd}
\]
\end{itemize}
Face and degeneracy maps in $\Sing^\FS_\bullet(X)$ are given by omitting or repeating the finite subsets $S_i$, and by suitably restricting $s$.
\end{defn}
\label{defn:SingFSR}
We may now model $\int_{|X|}\I$ as the rational chain complex associated with the following simplicial $\bQ$-vector space $\Sing^\FS_\bullet(X,\I)$.
\begin{defn}
The simplicial $\bQ$-vector space $\Sing^\FS_\bullet(X,\I)$ is given in degree $p\ge0$ by
\[
\Sing^\FS_p(X,\I)=\bigoplus_{(\udlS,s)\in\Sing^\FS_p(X)}\I^{\otimes S_0}.
\]
For $1\le i\le p$, respectively $0\le j\le p$, the face map $d_i$, respectively the degeneracy map $s_j$, restricts to the identity of $\I^{\otimes S_0}$ from the $(\udlS,s)$-summand to the $d_i(\udlS,s)$-summand, respectively to the $s_j(\udlS,s)$-summand.
For $i=0$, instead, it sends the $(\udlS,s)$-summand $\I^{\otimes S_0}$ to the $d_0(\udlS,s)$-summand $\I^{\otimes S_1}$ along the multiplication map $\I^{\otimes S_0}\to \I^{\otimes S_1}$ induced by the surjection $S_0\twoheadrightarrow S_1$.
\end{defn}

\begin{rem}
\label{rem:catofsimplices}
The simplicial $\bQ$-vector space $\Sing^\FS_p(X,\I)$ arises concretely by considering the category of simplices of $\Sing^\FS_p(X,\I)$, which we temporarily denote by $W$. There is a functor $W\to\Vect_\bQ$ sending $(\udlS,s)\mapsto \I^{\otimes S_0}$, and $\Sing^\FS_\bullet(X,\I)$ is the left Kan extension of this functor along the projection $W\to\Delta\op$.
\end{rem}
The model $\Sing^\FS_\bullet(X,\I)$ of $\int_{|X|}\I$ is a bit large, so we next aim at providing a smaller model. The following definition recovers the exit path quasi-category of the Ran space of $X$, stratified by cardinality, in the sense of 
\cite[Definition A.6.2]{HA}.
\begin{defn}
\label{defn:Ran}
We let $\Ran(X)$ denote the set of all finite nonempty subsets of $X$. For $n\ge1$ we denote by $\Ran_{\le n}(X)\subseteq\Ran(X)$ the subset of those subsets of $X$ having cardinality at most $n$. We topologise $\Ran_{\le n}(X)$ as a quotient of $X^n$, and $\Ran(X)$ as the union $\bigcup_{n\ge1}\Ran_{\le n}(X)$. We denote by $\#\colon\Ran(X)\to\bN_{\ge1}$ the lower semicontinuous function assigning to a finite nonempty subset of $X$ its cardinality.

For $k\ge1$ we denote by $\Conf_k(X)\coloneqq\#^{-1}(k)\subseteq\Ran(X)$ the $k$\textsuperscript{th} unordered configuration space of $X$.
\end{defn}
\begin{defn}
\label{defn:ER}
We let $\ER_\bullet(X)\subseteq\Sing_\bullet(\Ran(X))$ denote the subsimplicial set spanned at level $p$ by those singular simplices $s\colon\Delta^p\to\Ran(X)$ for which the following property holds: there exist $n_0\ge \dots\ge n_p\in\bN_{\ge1}$ such that for all $0\le i\le p$ and all $(0,\dots,0,t_i,\dots,t_p)\in\Delta^p$ with $t_i\neq0$ we have $\#(s(0,\dots,0,t_i,\dots,t_p))=n_i$.
\end{defn}
There is an injective map of simplicial sets $\kappa\colon\ER_\bullet(X)\hookrightarrow\Sing^\FS_\bullet(X)$: it sends $s\in\ER_p(X)$ to the unique element of the form $(\udlS,t)\in\Sing^\FS_p(X)$ satisfying:
\begin{itemize}
\item $S_i=s(0,\dots,1,\dots,0)$, where the unique occurrence of 1 is in the $i$\textsuperscript{th} position;
\item $s\colon\Delta^p\to\Ran(X)$ agrees with the composite $\Delta^p\xrightarrow{t}X^{S_0}\to\Ran(X)$;
\item $\alpha(t)\in X^{S_0}$ is the inclusion $S_0\hookrightarrow X$.
\end{itemize}
In the following, we treat $\kappa$ as an inclusion of simplicial sets.
\begin{rem}
By \cite[Theorem A.6.4]{HA} we have that $\ER_\bullet(X)$ is in fact a quasi-category. We believe that the functor of quasi-categories $\kappa$ is in fact final, but refrain from proving this more general statement and focus on proving that it preserves the colimit corresponding to factorisation homology.
\end{rem}
\begin{defn}
Recall Definition \ref{defn:SingFSR} and Remark \ref{rem:catofsimplices}.
We let $\ER_\bullet(X,\I)$ be the simplicial $\bQ$-subvector space of $\Sing^\FS_\bullet(X,\I)$ given in degree $p\ge0$ by
\[
\ER_p(X,\I)=\bigoplus_{s\in\ER_p(X)}\I^{\alpha(s)}\subset\Sing^\FS_p(X,\I).
\]
We denote by $\kappa^\I\colon\ER_\bullet(X,\I)\hookrightarrow\Sing^\FS_\bullet(X,\I)$ the canonical inclusion.
\end{defn}
The main result of the subsection is the following lemma.
\begin{lem}
\label{lem:kappaiso}
The map $\kappa^\I$ induces a quasi-isomorphism on chain complexes.
\end{lem}
To prove Lemma \ref{lem:kappaiso} we need to introduce suitable compatible filtrations on the source and target of $\kappa^\I$.
\begin{defn}
\label{defn:starfiltation}
For $k\ge1$ we let $F_k\Sing^\FS_\bullet(X)$ denote the subsimplicial set spanned by those simplices $(\udlS,s)$ with $\#S_0\le k$. We obtain a filtration of simplicial sets $F_\star\Sing^\FS_\bullet(X)$ with $F_k\Sing^\FS_\bullet(X)\subseteq F_{k+1}\Sing^\FS_\bullet(X)$. We restrict this filtration along $\kappa$ and obtain a filtration $F_\star\ER_\bullet(X)$.

Passing to simplicial $\bQ$-vector spaces, we also let $F_k\Sing_\bullet^\FS(X,\I)$ denote the simplicial $\bQ$-subvector space of $\Sing_\bullet^\FS(X,\I)$ spanned by those summands corresponding to simplices in $F_k\Sing_\bullet^\FS(X)$. We also restrict this filtration along $\kappa^\I$.
\end{defn}
\begin{proof}[Proof of Lemma \ref{lem:kappaiso}]
Since $\kappa^\I$ is a filtered map, it suffices to prove it induces a quasi-isomorphism $(F_k/F_{k-1})\kappa^\I$ between filtration quotients for all $k\ge1$. Fix $k$ in the following and let $T\subseteq X$ be a fixed set of $k$ elements. 
Using the nerves of the categories $\FS$ and $\FS_{T/}$,
define the auxiliary simplicial set 
\[
\Sing^\FS_\bullet(X)_{T/}\coloneqq\Sing^\FS_\bullet(X)\times_{N_\bullet(\FS)}N_\bullet(\FS_{T/}),
\]
and following Remark \ref{rem:catofsimplices} define the auxiliary simplicial $\bQ$-vector space $\Sing^\FS_\bullet(X,\I)_{T/}$, given as the left Kan extension to $\Delta\op$ of the functor from the category of simplices of $\Sing^\FS_\bullet(X,\I)_{T/}$ to $\Vect_\bQ$ sending $(\udlS,s,T\twoheadrightarrow S_0)\mapsto \I^{\otimes S_0}$.
Explicitly, we have
\[
\Sing^\FS_p(X,\I)_{T/}=\bigoplus_{(\udlS,s,T\twoheadrightarrow S_0)\in\Sing^\FS_p(X)_{T/}}\I^{\otimes S_0}
\]

Define similarly the pullback simplicial set
\[
\ER_\bullet(X)_{T/}\coloneqq \ER_\bullet(X)\times_{\Sing_\bullet^\FS(X)}\Sing_\bullet^\FS(X)_{T/}\subseteq\Sing^\FS_\bullet(X)_{T/}
\]
and the auxiliary simplicial $\bQ$-vector space $\ER_\bullet(X,\I)_{T/}$, with
\[
\ER_p(X,\I)_{T/}=\bigoplus_{(s,T\twoheadrightarrow \alpha(s))\in\ER_p(X)_{T/}}R^{\alpha(s)}\subseteq\Sing^\FS_\bullet(X,\I)_{T/}.
\]
Then $\kappa^\I$ canonically lifts to an injective map of simplicial $\bQ$-vector spaces
\[
\kappa^\I_{T/}\colon \ER_\bullet(X,\I)_{T/}\hookrightarrow\Sing^\FS_\bullet(X,\I)_{T/},
\]
which we treat as an inclusion. The group $\fS_T$ of automorphisms of $T$ naturally acts on both auxiliary simplicial $\bQ$-vector spaces, and $\kappa^\I_{T/}$ is $\fS_T$-equivariant.

For $n=k-1,k$ denote by $F_n\Sing^\FS_\bullet(X,\I)_{T/}$ the preimage of $F_n\Sing^\FS_\bullet(X,\I)$ along the forgetful map $\Sing^\FS_\bullet(X,\I)_{T/}\to \Sing^\FS_\bullet(X,\I)$, and define  in a similar way $F_n\ER_\bullet(X,\I)_{T/}$. Then $\kappa^\I_{T/}$ is a filtered simplicial $\bQ$-vector space map, and upon quotienting the $\fS_T$-action, the simplicial map $(F_k/F_{k-1})\kappa^\I_{T/}$ becomes the simplicial map $(F_k/F_{k-1})\kappa^\I$. Using that the source and target of $(F_k/F_{k-1})\kappa^\I_{T/}$ are simplicial $\bQ$-vector spaces consisting of (necessarily) projective $\bQ[\fS_T]$-modules, it suffices to prove that $(F_k/F_{k-1})\kappa^\I_{T/}$ is a quasi-isomorphism.

For this, notice that $\ER_\bullet(X,\I)_{T/}$ computes the relative homology of the pair of simplicial sets $\pa{F_k\ER_\bullet(X)_{T/},F_{k-1}\ER_\bullet(X)_{T/}}$ with \emph{constant} coefficients in $\I^{\otimes T}$, and similarly 
$\Sing^\FS_\bullet(X,\I)_{T/}$ computes the relative homology of the pair of simplicial sets $\pa{F_k\Sing^\FS_\bullet(X)_{T/},F_{k-1}\Sing^\FS_\bullet(X)_{T/}}$ with constant coefficients in $\I^{\otimes T}$. It suffices therefore to show that the map of simplicial sets $\kappa_{T/}\colon\ER_\bullet(X)_{T/}\hookrightarrow\Sing^\FS_\bullet(X)_{T/}$ lifting $\kappa$ is a filtered weak equivalence with respect to the two-step filtration given above. This follows by applying
\cite[Corollary A.9.4]{HA} to the space $X^T$, filtered by its diagonals corresponding to all partitions of $T$.
\end{proof}
If $\I$ is a \emph{graded} non-unital commutative $\bQ$-algebra, i.e. a non-unital symmetric monoidal functor $\FS\to\Vect_\bQ^\bN$, then $\ER_\bullet(X,\I)$ naturally upgrades to a simplicial graded $\bQ$-vector space, whose part $\ER_\bullet(X,\I)_n$ of grading $n\in\bN$ models the rational spectrum $(\int_{|X|}\I)_n$. 
\subsection{The algebra \texorpdfstring{$\cA(l)$}{A(l)}}
In this subsection we introduce two commutative algebras in graded rational spectra, which we denote by $\cA(l),\cAh(l)\in\CAlg(\Sp_\bQ^\bN)$; the first appears in the statement of Theorems \ref{thm:B},\ref{thm:D} and \ref{thm:E}, whereas the second is a mild, auxiliary variation in which all homological degrees are forced to be 0, whence the symbol $\heartsuit$. We will actually start introducing this variation.

Let $l\ge2$, and consider the rational group algebra $\bQ[\fS_l]=\bigoplus_{\sigma\in\fS_l}\bQ\cdot\sigma$. It is an associative algebra in $\Vect_\bQ$, and we may consider it as an object in $\Alg(\Sp_\bQ)$ concentrated in degree 0. There is also an action of the group $\fS_l$ on $\bQ[\fS_l]$ by conjugation: it is an action by automorphisms of associative algebras.

Recall Definition \ref{defn:norm}. We may filter $\bQ[\fS_l]$ according to the norm, i.e. we may declare, for $k\ge0$, the $k$\textsuperscript{th} filtration part to be
\[
F_k=F_k\bQ[\fS_l]=\bigoplus_{\sigma\in\fS_l,N(\sigma)\le k}\bQ\cdot \sigma.
\]
This is a multiplicative filtration, i.e. $F_k\cdot F_{k'}\subseteq F_{k+k'}$; it is also preserved by the $\fS_l$-action. We may therefore pass to the associated graded and obtain the graded associative algebra
\[
\mathrm{grad}_F\bQ[\fS_l]=\bigoplus_{\sigma\in\fS_l}\bQ\cdot[\sigma]\in\Alg(\Sp^\bN_\bQ),
\]
in which $[\sigma]$ denotes the image of $\sigma\in F_{N(\sigma)}$ in the quotient $F_{N(\sigma)}/F_{N(\sigma)-1}\subseteq\mathrm{grad}_F\bQ[\fS_l]$. There is still an action of $\fS_l$ on $\mathrm{grad}_F\bQ[\fS_l]$ by automorphisms of graded algebras.
\begin{defn}
\label{defn:cAheart}
We let $\cAh(l)$ denote the subalgebra of $\fS_l$-invariants
\[
\cAh(l)\coloneqq \mathrm{grad}_F\bQ[\fS_l]^{\fS_l}\in\Alg(\Sp_\bQ^\bN).
\]
Its part of grading $0$ is isomorphic to $\bQ$; we let $\cAh(l)_+\subseteq\cAh(l)$ denote the part of strictly positive grading.
\end{defn}
We have that $\cAh(l)$ is\footnote{We should say ``has the structure'' instead of ``is'', but both here, and later for $\cA(l)$, the commutative upgrading is essentially unique.} in fact a commutative algebra in graded rational spectra \cite[Lemma 4.31]{Bianchi:Hur1}. The underlying (ungraded) spectrum of $\cAh(l)$ is the direct sum $\bigoplus_{\lambda}\bQ\cdot\sca{\lambda}$, where the sum is extended over all partitions $\lambda\in\fP(l)$, i.e. over all conjugacy classes in $\fS_l$, and where $\sca{\lambda}\coloneqq\sum_{\sigma\in\lambda}[\sigma]$. The summand $\bQ\cdot\sca{\lambda}$ has grading $N(\lambda)$ but it lives in homological degree 0.
\begin{nota}
\label{nota:bK}
Let $k\ge1$, let $\mu\vdash l$, and let $\ulambda=(\lambda_1,\dots,\lambda_k)$ be a sequence of partitions $\lambda_i\vdash l$. 
For $\sigma\in\mu$ we denote by $\bK(\sigma,\ulambda)$ the set of all sequences $\utau=(\tau_1,\dots,\tau_k)$ of permutations $\tau_i\in\lambda_i$ such that the equality $\sigma=\tau_1\dots\tau_k\in\fS_l$ holds.
We denote by $\#\bK(\mu,\ulambda)\coloneqq\#\bK(\sigma,\ulambda)$ for any $\sigma\in\mu$.
\end{nota}
Notation \ref{nota:bK} allows us to explicitly write a formula for the product
\[
\sca{\lambda_1}\dots\sca{\lambda_k}=\sum_{\mu\in\fP(l)}\#\bK(\mu,\ulambda)\cdot\sca{\mu},
\]
using also Notation \ref{nota:fP}. 
Note that whenever the equality $N(\mu)=\sum_{i=1}^kN(\lambda_i)$ does not hold, we have $\#\bK(\mu,\ulambda)=0$.

We may instead start from an alternative version of the group algebra $\bQ[\fS_l]_N\in\Alg(\Sp_\bQ)$, whose underlying spectrum is $\bigoplus_{\sigma\in\fS_l}\bQ[-2N(\sigma)]\cdot\sigma$; that is, we deliberately put the generator $\sigma$ in homological degree $-2N(\sigma)$, which means cohomological degre $2N(\sigma)$. We still have a multiplicative filtration $F_\bullet$ which is compatible with the $\fS_l$-action given by conjugation.
\begin{defn}
\label{defn:cA}
We define $\cA(l)\in\Alg(\Sp_\bQ^\bN)$ as
\[
\cA(l)\coloneqq\mathrm{grad}_F\bQ[\fS_l]_N^{\fS_l}\in\Alg(\Sp_\bQ^\bN).
\]
\end{defn}
Again, we may regard $\cA(l)\in\CAlg(\Sp_\bQ^\bN)$, and we may decompose the underlying spectrum of $\cA(l)$ as $\bigoplus_{\sigma\in I}\bQ[-2N(\sigma)]\cdot \sca{\lambda}$.

\begin{rem}
\label{rem:evenshift}
We observe that $\cA(l)$ may be obtained from $\cAh(l)$ also in a direct, categorical way. The category $\Sp_\bQ^\bN$ is endowed with a symmetric monoidal shift endofunctor $\shift\colon\Sp_\bQ^\bN\to\Sp_\bQ^\bN$, sending the sequence of spectra $(\omega_n)_{n\in\bN}$ to the sequence $(\omega_n\ [-2n])_{n\in\bN}$. We then have that the commutative algebra $\cA(l)$ is the image of the commutative algebra $\cAh(l)$ along $\shift$. Since $\shift$ is an equivalence, it preserves colimits; being symmetric monoidal, it also preserves factorisation homologies. In particular, for all $n\in\bN$ and naturally in a space $X$ we have an equivalence of rational spectra 
\[
\pa{\int_{|X|}\cA(l)}_n\simeq\pa{\int_{|X|}\cAh(l)}_n\ [-2n].
\]
\end{rem}
By Lemma \ref{lem:kappaiso} and Remark \ref{rem:FSreduction}, for a topological surface $\Sigma$ we may model $\int_{|\Sigma|}\cAh(l)$ by the simplicial graded $\bQ$-vector space $\ER_\bullet(\Sigma,\cAh(l))$; moreover, the fibre of the augmentation $\int_{|\Sigma|}\cAh(l)\to\int_{|\Sigma|}\bQ\simeq\bQ$ is the part $(\int_{|\Sigma|}\cAh(l))_{\ge1}$ of strictly positive grading, and it can be modeled by $\ER_\bullet(\Sigma,\cAh(l)_+)$.

If $g\ge0$ and $l\ge2$ are such that $\upsilon(g,l,h)\ge1$, we obtain in particular that $\ER_\bullet(\Sigma,\cAh(l)_+)_{2\upsilon(g,l,h)}$ is a simplicial $\bQ$-vector space model for the source rational spectrum $(\int_{|\Sigma|}\cAh(l))_{2\upsilon(g,l,h)}$ of the map $\Theta^{g,l}_\Sigma$ in Proposition \ref{prop:defnTheta}

\subsection{The branch values map}
For the reminder of the section we fix a connected closed surface $\Sigma$ of genus $h\ge0$ and numbers $g\ge0$, $l\ge2$, and we abbreviate $\upsilon=\upsilon(g,l,h)$. We assume $\upsilon\ge1$ to avoid trivialities.
In this subsection we introduce the branch values map $\brv$ and analyse its main properties.
\begin{defn}
\label{defn:brv}
We let $\brv\colon\CBra^{g,l}_{\Sigma,\coarse}\to\Ran_{\le2\upsilon}(\Sigma)\subseteq\Ran(\Sigma)$ denote the continuous map sending an isomorphism class of maps $\theta\colon C\to \Sigma$ to its set of branch values in $\Sigma$. For $1\le k\le2\upsilon$ we let $F_k\CBra^{g,l}_{\Sigma,\coarse}=\brv^{-1}(\Ran_{\le k}(\Sigma))$; this defines a filtration of $\CBra^{g,l}_{\Sigma,\coarse}$ by the poset $\bN$.

We define a similar filtration of $\CBra^{g,l,\infty}_{\Sigma,\coarse}$ by the subspaces 
\[
F_k\CBra^{g,l,\infty}_{\Sigma,\coarse}\coloneqq F_k\CBra^{g,l}_{\Sigma,\coarse}\cup\set{\infty},
\]
and for $k=-1,-2$ we also let $F_{-1}\CBra^{g,l,\infty}_{\Sigma,\coarse}=\set{\infty}$ and $F_{-2}\CBra^{g,l,\infty}_{\Sigma,\coarse}=\emptyset$. 

\end{defn}

We observe that $\brv$ is a map with finite fibres between locally compact Hausdorff spaces; moreover, for any $k\ge1$, the restricted map
\[
\brv\colon 
(F_k\setminus F_{k-1})\CBra^{g,l}_{\Sigma,\coarse}
\to \Conf_k(\Sigma)
\]
is finite \'etale. In the following lemma we highlight a technical feature of $\brv$.
\begin{lem}
\label{lem:onelift}
Let $p,k\ge1$, let $s\colon\mD^p\to (F_k\setminus F_{k-1})\CBra^{g,l,\infty}_{\Sigma,\coarse}$ be a continuous map defined on the interior of $\Delta^p$, and assume that the composite $\brv\circ s$ extends to a map $\Delta^p\to\Ran_{\le2\upsilon}(\Sigma)$. Then $s$ extends to a continuous map $\Delta^p\to\CBra^{g,l,\infty}_{\Sigma,\coarse}$.
\end{lem}
\begin{proof}
Let $(x_i)_{i\in\bN}$ be a sequence of points in $\mD^p$ with limit $x_\infty\in\partial\Delta^p$. Let $\brv\circ s(x_\infty)=\set{z_1,\dots,z_h}\subseteq\Sigma$, with $h\le k$. Let $\uU=(U_1,\dots,U_h)$ be disjoint open discs in $\Sigma$ with $z_j\in U_j$. Up to discarding finitely many points $x_i$, we may assume that for all $i\in\bN$ we have $\brv\circ s(x)\subseteq \uU$ for any $x$ lying in the straight interval joining $x_i$ with $x_{i+1}$ in $\mD^p$. Pick representatives $\theta_i\colon C_i\to \Sigma$ of $s(x_i)$; by the previous we have that each $l$-fold branched cover $\theta_i$ restricts to an \'etale cover $\theta_i\colon\theta_i^{-1}(\SmU)\to\SmU$, and that these restricted covers are pairwise isomorphic. In particular, for all $1\le j\le h$, the norm of the monodromy of $\theta_i$ around the boundary of $U_j$ only depends on $j$, but not on $i$: let us denote this number by $n_j$. There are now two cases:
\begin{itemize}
\item if $\sum_{j=1}^hn_j=2\upsilon$, then the sequence $s(x_i)$ admits a limit in $\CBra^{g,l}_{\Sigma,\coarse}$, given by the unique isomorphism class of covers $\theta_\infty\colon C_\infty\to\Sigma$ satisfying both of the following:
\begin{itemize}
\item the image of $\theta_\infty$ along $\brv$ is $\set{z_1,\dots,z_h}$;
\item the restriction of $\theta_\infty$ over $\SmU$ is isomorphic to any restriction of $\theta_i$ over $\SmU$;
\end{itemize}
\item if instead $\sum_{j=1}^hn_j<2\upsilon$, then the sequence $s(x_i)$ tends to $\infty$.
\end{itemize}
In all cases we have found a  limit in $\CBra^{g,l,\infty}_{\Sigma,\coarse}$.
\end{proof}
In the reminder of the subsection we let $s\colon[0,1]\to\Ran_{\le2\upsilon}(\Sigma)$ be a continuous map restricting to a map $[0,1)\to\Conf_k(\Sigma)$, and we let $\theta_1\colon C_1\to \Sigma$ be a representative of elements in $\CBra^{g,l}_{\Sigma,\coarse}$ with $\brv(\theta_1)=s(1)$. We are broadly interested in classifying lifts of $s$ along $\brv$ ending at $\theta$, i.e.
maps $t\colon [0,1]\to\CBra^{g,l}_{\Sigma,\coarse}$ with $\brv\circ t=s$ and $t(1)=\theta_1$. Our goal is to establish the equalities in Lemmas \ref{lem:eulerchar} and \ref{lem:bLequalbK}.

We denote $P=s(1)=\set{z_1,\dots,z_h}\subseteq\Sigma$, with $h\le k$, and we let $\mu_j\vdash l$ denote the conjugacy class of the local monodromy of $\theta_1$ at $z_j\in P$.

Fix a collection of disjoint open discs $\uU=(U_1,\dots,U_k)$ of $\Sigma$ with $z_j\in U_j$. Let $\epsilon>0$ be such that $s(x)\subseteq\uU$ for all $x\in[1-\epsilon,1]$.
Since $\brv$ is \'etale over $\Conf_k(\Sigma)$, lifts defined over $[1-\epsilon,1]$ uniquely extend to lifts defined over the entire $[0,1]$. Up to replacing $s$ with its restriction on $[1-\epsilon,1]$ we may henceforth assume that $s(x)\subseteq\uU$ for all $x\in[0,1]$. In the following, let $P'=s(0)$, let $P'_j=P'\cap U_j$, and for each $1\le j\le h$ let $z'_{j,1},\dots,z'_{j,k_j}$ be the points of $P'_j$, for suitable numbers $k_j>0$ satisfying $\sum_{j=1}^hk_j=k$.

\begin{defn}
\label{defn:bL}
Let $\ulambda=(\lambda_{j,i})_{1\le j\le h,1\le i\le k_j}$ be a sequence of partitions of $l$, and assume $\sum_{i=1}^{k_j}N(\lambda_{j,i})=N(\sigma_j)$. We denote by
$\bL(\theta_1,\uU,\ulambda)$
the set of all isomorphism classes of pairs $(\theta_0,\psi)$ consisting of a map $\theta_0\colon C_0\to \Sigma$ of degree $l$ with source of genus $g$, branched over $P'=s(0)$, and an identification $\psi\colon\theta_0^{-1}(\SmU)\cong\theta_1^{-1}(\SmU)$ over $\SmU$; we further require that the local monodromy of $\theta_0$ at $z'_{i,j}$ is in $\lambda_{i,j}$. Notice that two such pairs $(\theta_0,\psi)$ and $(\theta'_0,\psi')$, if they are isomorphic, are uniquely isomorphic. The group $\Aut(\theta_1)$ of automorphisms of $\theta_1\colon C_1\to \Sigma$ can be identified with the group of automorphisms of the restricted, \'etale cover $\theta_1^{-1}(\SmU)\to\SmU$, so it acts on the set $\bL(\theta_1,\uU,\ulambda)$ by postcomposition on $\psi$.
\end{defn}

We start with some general observations. Let $t\colon[0,1]\to\CBra^{g,l}_{\Sigma,\coarse}$ be any lift of $s$ along $\brv$ ending at $\theta_1$, and denote by $\theta_x\colon C_x\to\Sigma$ a representative of $t(x)$ for all $x\in[0,1]$. Then 
for all $x\in[0,1]$, and in particular for $x=0$, the \'etale coverings $\theta_x^{-1}(\SmU)\to\SmU$ and $\theta_1^{-1}(\SmU)\to\SmU$ are abstractly isomorphic. Moreover, if $\lambda_{j,i}$ denotes the conjugacy class of the local monodromy of $\theta_0$ at $z'_{j,i}\in P'=\brv(t(0))$, then the equalities $N(\mu_j)=\sum_{i=1}^{k_j}N(\lambda_{j,i})$ hold.

Viceversa, let $\theta_0\colon C_0\to\Sigma$ be any genus $g$ degree $l$ branched cover whose restriction over $\SmU$ is abstractly isomorphic to $\theta_1^{-1}(\SmU)$, and let $\lambda_{j,i}$ be the conjugacy class of the local monodromy of $\theta_0$ at $z'_{j,i}$. Since both $C_0$ and $\theta_1^{-1}(\SmU)$ have the same genus $g$, the
preimage $\theta_0^{-1}(\uU)$ must be a disjoint union of discs.
The Riemann--Hurwitz formula then implies that the equalities $N(\mu_j)=\sum_{i=1}^{k_j}N(\lambda_{j,i})$ hold. Therefore, we may extend $\theta_0$ to a map $t\colon[0,1)\to\CBra^{g,l}_{\Sigma,\coarse}$ lifting along $\brv$ the restriction of $s$ on $[0,1)$, and we may then apply Lemma \ref{lem:onelift} to extend $t$ to the entire $[0,1]$; the proof of Lemma \ref{lem:onelift} shows that $t(1)$ will be the isomorphism class of $\theta_1$.
The previous discussion shows the following.
\begin{lem}
\label{lem:eulerchar}
The set $\lift(s,\theta_1)$ of lifts of $s$ along $\brv$ ending at $\theta_1$ is in bijection with the set of orbits $\bL(\theta_1,\uU,\ulambda)/\Aut(\theta_1)$. Moreover, the following equality holds
\[
\sum_{t\in\lift(s,\theta_1)}\frac{1}{\#\Aut(\alpha(t))}=\frac{\#\bL(\theta_1,\uU,\ulambda)}{\#\Aut(\theta_1)},
\]
where $\#\Aut(\alpha(t))$ denotes the cardinality of the group of automorphisms of any representative$\theta_0$ of $\alpha(t)=t(0)$.
\end{lem}

Now fix also a basepoint
$a\in \SmU$, and fix a trivialisation of the finite set $\theta_1^{-1}(a)\cong\udll$; this allows us to represent $\theta_1$ via the monodromy group homomorphism $\phi_1\colon\pi_1(\SmP,a)\cong\pi_1(\SmU,a)\to\fS_l$.

If $(\theta_0,\psi)\in\bL(\theta_1,\uU,\ulambda)$, then we may use $\psi$ to trivialise $\theta_0^{-1}(a)\cong\udll$ and thus represent $\theta_0$ via a group homomorphism $\phi_0\colon\pi_1(\SmP',a)\to\fS_l$; the restriction of $\phi_0$ on the subgroup $\pi_1(\SmU,a)\cong\pi_1(\SmP,a)$ coincides with $\phi_1$, and the equalities $N(z_j,\phi_1)=\sum_{i=1}^{k_j}N(z'_{j,i},\phi_0)$ hold.

Viceversa, if we are able to extend the group homomorphism $\phi_1\colon\pi_1(\SmU,a)\to\fS_l$ to a group homomorphism $\phi_0\colon\pi_1(\SmP',a)\to\fS_l$ such that the equalities $N(z_j,\phi_1)=\sum_{i=1}^{k_j}N(z'_{j,i},\phi_0)$ hold, then $\phi_0$ allows us to define an $l$-fold branched cover $\theta_0\colon C_0\to \Sigma$, branched at $P'$, and the Riemann-Hurwitz formula implies that $C_0$ has genus $g$. By construction, we also have an identification $\psi\colon\theta^{-1}_0(\SmU)\cong\theta_1^{-1}(\SmU)$, so all together we have obtained an element $(\theta_0,\psi)\in\bL(\theta_1,\uU,\ulambda)$.

Our primary goal is therefore to classify extensions of the group homomorphism $\phi_1\colon\pi_1(\SmU,a)\to\fS_l$ to $\pi_1(\SmP',a)$.
Let $\gamma_1,\dots,\gamma_h$ be simple loops in $\SmU$, based at $a$, with $\gamma_j$ bounding a disc in $\Sigma$ containing only $U_j$. Assume also that the $\gamma_j$'s pairwise intersect only at $a$. Fix for all $1\le j\le h$ a factorisation of $\gamma_j\in\pi_1(\SmP',a)$ as a product $\gamma'_{j,1}\dots\gamma'_{j,k_j}$ for suitable small loops in $\pi_1(\SmP',a)$, with $\gamma'_{j,i}$ spinning around $z'_{j,i}$.
Then the Seifert--Van Kampen theorem identifies $\pi_1(\SmP',a)$ with the free amalgamated product of $\pi_1(\SmU,a)$ and the free group on $k$ generators $\gamma'_{j,i}$ along the free group on $h$ generators $\gamma_j$. In particular, an extension of $\phi_1$ to a group homomorphism $\phi_0\colon\pi_1(\SmP',a)\to\fS_l$ is uniquely determined by a choice, for each $1\le j\le h$, of a factorisation of the permutation $\sigma_j\coloneqq\phi_1(\gamma_j)$ as a product $\tau_{j,1}\dots\tau_{j,k_j}$ of $k_j$ permutations $\tau_{j,i}\coloneqq\phi_0(\gamma'_{j,i})$. In order to enforce that the pair $(\theta_0,\psi)$ resulting from $\phi_0$ is indeed an element of $\bL(\theta_1,\uU,\ulambda)$, we further require $\tau_{j,i}\in\lambda_{j,i}$.
The previous discussion gives a proof of the following lemma.
\begin{lem}
\label{lem:bLequalbK}
Recall Notation \ref{nota:bK}.
Then there is a bijection $\bL(\theta_1,\uU,\ulambda)=\prod_{j=1}^h\bK(\sigma_j,\ulambda_j)$, where $\ulambda_j=(\lambda_{j,1},\dots,\lambda_{j,k_j})$. In particular we have
\[
\#\bL(\theta_1,\uU,\ulambda)=\#\bK(\umu,\ulambda)\coloneqq\prod_{j=1}^h\#\bK(\mu_j,\ulambda_j)
\]
\end{lem}

\subsection{Construction of \texorpdfstring{$\Theta^{g,l}_\Sigma$}{Theta gl Sigma}}
We next introduce a convenient simplicial model for the target of $\Theta^{g,l}_\Sigma$ as well. 
\begin{defn}
Similarly as in Definition \ref{defn:ER}, we let
\[
\Sing'_\bullet(\CBra^{g,l,\infty}_{\Sigma,\coarse})\subseteq\Sing_\bullet(\CBra^{g,l,\infty}_{\Sigma,\coarse})
\]
denote the simplicial subset spanned at level $p$ by those singular simplices $s\colon\Delta^p\to\CBra^{g,l,\infty}_{\Sigma,\coarse}$ for which the following property holds: there exist $n_0\ge \dots\ge n_p\in\bZ_{\ge-1}$ such that for all $0\le i\le p$ and all $(0,\dots,0,t_i,\dots,t_p)\in\Delta^p$ with $t_i\neq0$ we have
\[
s(0,\dots,0,t_i,\dots,t_p)\in (F_{n_i}\setminus F_{n_i-1})\CBra^{g,l,\infty}_{\Sigma,\coarse}.
\]
We let $\Sing'_\bullet(\CBra^{g,l,\infty}_{\Sigma,\coarse};\bQ)$ be the simplicial $\bQ$-vector space obtained as the $\bQ$-lineari\-sation of $\Sing'_\bullet(\CBra^{g,l,\infty}_{\Sigma,\coarse})$.
For $k\ge-2$ we denote by 
\[
F_k\Sing'_\bullet\CBra^{g,l,\infty}_{\Sigma,\coarse}\subseteq\Sing'_\bullet\CBra^{g,l,\infty}_{\Sigma,\coarse}
\]
the subsimplicial set spanned by those simplices $s$ such that $\alpha(s)\in F_k\CBra^{g,l,\infty}_{\Sigma,\coarse}$; we define a similar filtration on the $\bQ$-linearisation.
\end{defn}
The inclusion of simplicial sets $\Sing'_\bullet(\CBra^{g,l,\infty}_{\Sigma,\coarse})\hookrightarrow\Sing_\bullet(\CBra^{g,l,\infty}_{\Sigma,\coarse})$ is a weak equivalence \cite[Corollary A.9.4]{HA}. We may therefore model the rational spectrum
$(|\CBra^{g,l,\infty}_{\Sigma,\coarse}|,\infty)_!\bQ$ by the simplicial $\bQ$-vector space
\[
\Sing'_\bullet(\CBra^{g,l,\infty}_{\Sigma,\coarse},\infty;\bQ)\coloneqq\Sing'_\bullet(\CBra^{g,l,\infty}_{\Sigma,\coarse};\bQ)/F_{-1}\Sing'_\bullet(\CBra^{g,l,\infty}_{\Sigma,\coarse};\bQ).
\]
\begin{nota}
For $k\ge0$ we let
\[
F_k\Sing'_\bullet(\CBra^{g,l,\infty}_{\Sigma,\coarse},\infty;\bQ)\coloneqq (F_k/F_{-1})\Sing'_\bullet(\CBra^{g,l,\infty}_{\Sigma,\coarse};\bQ)
.
\]
\end{nota}

\begin{nota}
\label{nota:sscalambda}
Let $p\ge0$ and $s\in\ER_p(\Sigma)$, and let $(\sca{\lambda_z}_{z\in \alpha(s)})$ be a sequence of non-discrete partitions $\lambda_z\in\fP(l)_+$ indexed by the elements of the set $\alpha(s)\subseteq\Sigma$; assume that the equality $\sum_{z\in\alpha(s)}N(\lambda_z)=2\upsilon$ holds. 
With these data we associate the element
\[
(s,(\sca{\lambda_z}_{z\in \alpha(s)})\coloneqq \otimes_{z\in\alpha(s)}\sca{\lambda_z}\ \in\  \cAh(l)_+^{\otimes\alpha(s)}\subseteq\ER_p(\Sigma,\cAh(l)_+)_{2\upsilon},
\]
where the last inclusion picks the summand indexed by $s$.
\end{nota}

We observe that the elements $(s,(\sca{\lambda_z}_{z\in \alpha(s)}))$ form a basis of the $\bQ$-vector space $\ER_p(\Sigma,\cAh(l)_+)_{2\upsilon}$. This is restricts to a basis of each filtration level $F_k\ER_p(\Sigma,\cAh(l)_+)_{2\upsilon}$ for $k\ge1$.
\begin{defn}
\label{defn:gamma}
We define $\Theta^{g,l}_\Sigma$ as the map of simplicial $\bQ$-vector spaces
\[
\Theta^{g,l}_\Sigma\colon \ER_\bullet(\Sigma,\cAh(l)_+)_{2\upsilon}\to \Sing'_\bullet(\CBra^{g,l,\infty}_{\Sigma,\coarse},\infty;\bQ)
\]
sending the basis element $(s,(\sca{\lambda_z}_{z\in \alpha(s)}))$ at level $p\ge0$ to the sum
\[
\Theta^{g,l}_\Sigma(s,(\sca{\lambda_z}_{z\in \alpha(s)}))\coloneqq\sum \frac{1}{\#\Aut(\alpha(t))}\cdot t,
\]
where $t\colon\Delta^p\to \CBra^{g,l,\infty}_{\Sigma,\coarse}$ ranges over all elements of $\Sing'_p(\CBra^{g,l,\infty}_{\Sigma,\coarse})$ satisfying the following properties:
\begin{itemize}
\item the restriction $\mathring{t}\colon\mathring{\Delta}^p\hookrightarrow\Delta^p\xrightarrow{t}\CBra^{g,l,\infty}_{\Sigma,\coarse}$ has image inside $\CBra^{g,l}_{\Sigma,\coarse}$, and the composite map $\brv\circ\mathring{t}$ agrees with $s$;
\item if $\theta\colon C\to\Sigma$ represents the equivalence class of $\alpha(t)$, then the local monodromy of $\theta$ at $z\in\alpha(s)\subseteq\Sigma$ belongs to the conjugacy class $\lambda_z$.
\end{itemize}
Here we denote by $\Aut(\alpha(t))$ the group of automorphisms of $\theta$ as above over $\Sigma$.
\end{defn}
It is evident that $\Theta^{g,l}_\Sigma$ is compatible with the filtrations $F_\star$ on the source and target simplicial $\bQ$-vector spaces.
We are now ready to prove Proposition \ref{prop:defnTheta}, that is, to check that Definition \ref{defn:gamma} is well-posed, and that the resulting map $\Theta^{g,l}_\Sigma$ is indeed $\Homeo^+(\Sigma)$-equivariant.
\begin{proof}[Proof of Proposition \ref{prop:defnTheta}]
The filtration preserving action of $\Homeo^+(\Sigma)$ on the space $\CBra^{g,l,\infty}_{\Sigma,\coarse}$ gives rise to a filtration preserving action of the simplicial group $\Sing_\bullet(\Homeo^+(\Sigma))$ on the simplicial 
set $\Sing'_\bullet(\CBra^{g,l,\infty}_{\Sigma,\coarse})$, and  after $\bQ$-linearisation,
on the simplicial $\bQ$-vector
space $\Sing'_\bullet(\CBra^{g,l,\infty}_{\Sigma,\coarse},\infty;\bQ)$. The latter models the action of $\Homeo^+(\Sigma)$ on the spectrum $(|\CBra^{g,l,\infty}_{\Sigma,\coarse}|,\infty)_!\bQ$.

Similarly, $\Sing_\bullet(\Homeo^+(\Sigma))$ acts on $\ER_\bullet(\Sigma,\cAh(l)_+)_{2\upsilon}$ as follows: a $p$-simplex $t\colon\Delta^p\to\Homeo^+(\Sigma)$
sends the basis element $(s,(\sca{\lambda_z}_{z\in \alpha(s)})))$ 
to the basis element $(t\cdot s,(\sca{\lambda_{\alpha(t)^{-1}\cdot z}}_{z\in \alpha(t\cdot s)})))$.
This models the action of $\Homeo^+(\Sigma)$ on the spectrum $(\int_{|\Sigma|}\cAh(l))_{2\upsilon}$. The map $\Theta^{g,l}_\Sigma$ is evidently levelwise $\Sing_\bullet(\Homeo^+(\Sigma))$-equivariant.

It is left to check that $\Theta^{g,l}_\Sigma$ is indeed a simplicial map, i.e. it is compatible with face and degeneracy maps in $\Delta\op$. Fix a basis element $(s,\sca{\lambda_z}_{z\in\alpha(s)})$ of $\ER_p(\Sigma,\cAh(l)_+)_{2\upsilon}$. We want to show for $1\le i\le p$ the equalities 
\[
\begin{split}
d_i\circ\Theta^{g,l}_\Sigma(s,\sca{\lambda_z}_{z\in\alpha(s)})&=\Theta^{g,l}_\Sigma\circ d_i(s,\sca{\lambda_z}_{z\in\alpha(s)});\\
s_i\circ\Theta^{g,l}_\Sigma(s,\sca{\lambda_z}_{z\in\alpha(s)})&=\Theta^{g,l}_\Sigma\circ s_i(s,\sca{\lambda_z}_{z\in\alpha(s)}).
\end{split}
\]
For all degeneracy maps and for all face maps except $d_0$, this immediately follows from Lemma \ref{lem:onelift} and the fact that $\brv$ is \'etale over the strata $\Conf_k(\Sigma)\subseteq\Ran(\Sigma)$. 
To check the equality involving $d_0$,
we also fix a basis element $t\in\Sing'_{p-1}(\CBra^{g,l,\infty}_{\Sigma,\coarse},\infty;\bQ)$ and aim at checking that $t$ has the same coefficient on both expressions. If $\brv\circ t|_{\mD^{p-1}}\neq d_0(s)_{\mD^{p-1}}$, both coefficients of $t$ are equal because they both trivially vanish, so let us assume that the last equality holds;
by Lemma \ref{lem:onelift}
we then also have $\brv(\alpha(t))= \alpha(d_0((s)))$.  We let $P'=\alpha(s)$ and $P=\alpha(d_0(s))$, and expand $P=\set{z_1,\dots,z_h}$.
The restriction of $s$ along the straight segment in $\Delta^p$ joining $(1,0,\dots0)$ and $(0,1,0,\dots,0)$ gives rise to a map $P'\to P$; we let $P'_j$ be the preimage of $z_j$ along this map.
Let $\lambda_{j,i}$ be the partition $\lambda_{z_{j,i}}$, and denote by $\ulambda$ the list of all partitions $\lambda_{j,i}$. Finally, let $\mu_j\vdash l$ denote the conjugacy class of the local monodromy at $z_j\in P$ of a representative $\theta_1\colon C_1\to\Sigma$ of $\alpha(t)$.

The definition of the product in $\cAh(l)$, together with the definition of $\Theta^{g,l}_\Sigma$, tells us that the coefficient of $t$ in $\Theta^{g,l}_\Sigma\circ d_0(s,\sca{\ulambda})$ is equal to $\#\bK(\umu,\ulambda)/\#\Aut(\theta_1)$, whereas the definition of 
$\Theta^{g,l}_\Sigma$ tells us that the coefficient of $t$ in 
$d_0\circ\Theta^{g,l}_\Sigma(s,\sca{\ulambda})$ is equal to
$\sum_{(\theta_0,\psi)\in\bL(\theta_1,\uU,\ulambda)}1/\#\Aut(\theta_0)$, where $\uU$ is a union of $h$ disjoint discs in $\Sigma$ covering $P$. Combining Lemmas \ref{lem:eulerchar} and \ref{lem:bLequalbK} we obtain that the two coefficients of $t$ are equal.
\end{proof}

\begin{cor}
\label{cor:defnTheta}
In the hypotheses of Proposition \ref{prop:defnTheta}, there is a well-defined and $\Homeo^+(\Sigma)$-equiva\-riant map of rational spectra
\[
\Theta^{g,l}_\Sigma[-4\upsilon]\ \colon
\pa{\int_{\Sigma}\cA(l)}_{2\upsilon(g,l,h)}\to|\CBra^{g,l}_{\Sigma,\coarse}|_*\bQ.
\]
\end{cor}
\begin{proof}
By Corollary \ref{cor:Braorbifold} and Remark \ref{rem:evenshift}, the source and target of the map in the statement are $\Homeo^+(\Sigma)$-equivariantly equivalent to the source and target of the map $\Theta^{g,l}_\Sigma$ from Proposition \ref{prop:defnTheta}, shifted by $-4\upsilon$.
\end{proof}
We conclude the section introducing part of the notation from Theorem \ref{thm:B}, and speculating on a possible deeper interpretation of the map $\Theta^{g,l}_\Sigma$.
\begin{nota}
\label{nota:intfp}
For $h\ge0$ and $l\ge2$ we denote by 
\[
\int_{|\fp_h\tp|}\cA(l)\in(\Sp_\bQ^\bN)^{|\fM_h\tp|}
\]
the graded parametrised rational spectrum over $|\fM_h\tp|\simeq\bbB\Homeo^+(\Sigma)$ corresponding to $\int_{|\Sigma|}\cA(l)$ with its $\Homeo^+(\Sigma)$-action. Here $\Sigma$ is a closed connected topological surface of genus $h$. For $e\ge1$ we further denote
\[
\int_{|\fp_{h,e}|}\cA(l)\coloneqq\epsilon_{h,e}^*\int_{\fp_h}\cA(l)\in(\Sp_\bQ^\bN)^{|\cH_{h,e}\tp|}.
\]
Finally, recalling Equation \ref{eq:cHvcrs}, for $g\ge0$ we denote by 
\[
\Theta^{g,l}_{h,e}\colon\pa{\int_{|\fp_{h,e}|}\cAh(l)}_{2\upsilon(g,l,h)}\to (r_{g,l,h,e}^\infty,\infty)_!\bQ
\]
the image along $\epsilon^*_{h,e}$ of the map of graded parametrised rational spectra over $|\fM_h\tp|$ corresponding to 
$\Theta^{g,l}_\Sigma$.
By Corollary \ref{cor:defnTheta} we also obtain a map
\[
\Theta^{g,l}_{h,e}[-4\upsilon(g,l,h)]\colon\pa{\int_{|\fp_{h,e}|}\cA(l)}_{2\upsilon(g,l,h)}\to (r_{g,l,h,e})_*\bQ.
\]
\end{nota}

\begin{rem}
By \cite[Theorem 6.1]{Bianchi:Hur3}, the algebra $\pi_*(\cA(l))$ is isomorphic to the rational cohomology ring of the space $\bB=\bB((\fS_d\geo)_+,\fS_d)$ used in \cite{Bianchi:Hur3} to model a ``double delooping'' of the $E_1$-monoid $\Hur((0,1)^2;\fS_d\geo)$ of Hurwitz spaces with collisions. Das--Petersen replace $\bB$ with the underlying homotopy type of the topological stack $\Bra^l_{\R^2}$ featuring in \cite[Theorem 6.3]{DasPetersen}. 

Intuitively, $\Bra^l_{\R^2}$ parametrises degree $l$ branched covers of the plane $\R^2$ that have no prescribed behaviour at $\infty$. With some care, it should be possible to associate with an oriented topological surface $\Sigma$ a bundle $p\colon\Bra^l_{T\Sigma}\to\Sigma$ of topological stacks, whose fibre at $x\in \Sigma$ is $\Bra^l_{T_x\Sigma}$.
There should then also be a scanning map 
\[
\Bra^l_\Sigma\to\Gamma(p)=\Gamma(\Bra^l_{T\Sigma}\xrightarrow{p}\Sigma),
\]
which loosely sends $\theta\colon C\to \Sigma$ to the section sending $x\in\Sigma$ to the part of $C$ lying over a small neighbourhood of $x$.

In this light, the map $\Theta^{g,l}_\Sigma$ from Corollary \ref{cor:defnTheta} ``wants to'' be the induced on rational cochains by the previous scanning map. There are still some issues with this interpretation, and it would be interesting if one could resolve them:
\begin{itemize}
\item we are identifying $\Gamma(p)_*\bQ$ with $\int_{|\Sigma|} p_*\bQ$ as functors $\Sigma\to\CAlg(\Sp_\bQ)$; to do this we would need the relevant Eilenberg--Moore spectral sequence to converge, however this is not guaranteed as $\Sigma$ is 2-dimensional and the homotopy type fibre $|\Bra^l_{\R^2}|$ of $p$ is not 2-connected ;
\item we are further identifying $\int_{|\Sigma|} p_*\bQ$ and the constant functor at $|\Bra^l_{\R^2}|_*\bQ$ with each other, as  functors $|\Sigma|\to\CAlg(\Sp_\bQ)$; however the topological tangent bundle of $\Sigma$ is usually non-trivial;
\item we are replacing $|\Bra^l_{\R^2}|_*\bQ$ with its cohomology algebra; this requires the fibration $|p|\colon|\Bra^l_{T\Sigma}|\to|\Sigma|$, with homotopy fibre $|\Bra^l_{\R^2}|$, to be rationally formal \emph{as a fibration}, something which is not clear.
\end{itemize}
\end{rem}

\section{A first stability result}
We fix a closed connected surface $\Sigma$ of genus $h\ge0$, as well as $g\ge0$ and $l\ge2$, and abbreviate $\upsilon=\upsilon(g,l,h)$ throughout the section. We assume $\upsilon\ge1$.

The goal of this section is to prove the following theorem, motivating the introduction of the map $\Theta^{g,l}_\Sigma$ from Proposition \ref{prop:defnTheta}.
\begin{thm}
\label{thm:Thetaisomorphism}
The composite $\Homeo^+(\Sigma)$-equivariant map of rational spectra
\[
\pa{\int_{|\Sigma|}\cAh(l)}_{2\upsilon}\xrightarrow{\Theta^{g,l}_\Sigma}\pa{\CBra^{g,l,\infty}_{\Sigma,\coarse},\infty}_!\bQ \xrightarrow{\dw_\bQ}\pa{\dCBra^{g,l,\infty}_{\Sigma,\coarse},\infty}_!\bQ
\]
induces an isomorphism on homotopy groups in the range of degrees
\[
*>4\upsilon-(2\upsilon-J_l)/I_l.
\]
\end{thm}
\subsection{Transposition monodromy group}
\label{subsec:TMG}
The precise positive constants $I_l,J_l$, also appearing in Theorems \ref{thm:A} and \ref{thm:B}, will be fixed in Notation \ref{nota:goodconstantsIJ}. Besides that, the main missing ingredient in the statement of Theorem \ref{thm:Thetaisomorphism} is the space $\dCBra^{g,l,\infty}_{\Sigma,\coarse}$, we introduce this subspace next and give a characterisation of its points.
\begin{defn}
\label{defn:dCBra}
Let $A\subseteq B\subseteq\Sigma$ be finite subsets as in Definition \ref{defn:Bra}. We denote by $\dCBra^{g,l}_{\Sigma\rel A\et B}\subseteq\CBra^{g,l}_{\Sigma\rel A\et B}$ the open topological substack spanned at $X$ by the sequences $(Y,p,\vartheta,\bt)$ satisfying for all $x\in X$ the following properties:
\begin{itemize}
\item the branched cover $\theta\colon p^{-1}(x)\to\Sigma$ admits at least one simple branch value, i.e., a point of $\Sigma$ whose associated local monodromy is a transposition;
\item the branched cover $\theta\colon p^{-1}(x)\to\Sigma$ does not factor through any intermediate branched cover.
\end{itemize}
When $A=\emptyset$, we denote by $\dCBra^{g,l}_{\Sigma\et B,\coarse}$ the coarse space of $\dCBra^{g,l}_{\Sigma\et B}$.
\end{defn}
The open inclusion $\dCBra^{g,l}_{\Sigma,\coarse}\hookrightarrow\CBra^{g,l}_{\Sigma,\coarse}$ is $\Homeo^+(\Sigma)$-equivariant, and it gives rise to the collapse map 
$\dw\colon\CBra^{g,l,\infty}_{\Sigma,\coarse}\twoheadrightarrow\dCBra^{g,l,\infty}_{\Sigma,\coarse}$ whose associated map in homology relative to $\infty$ is the map $\dw_\bQ$ in the statement of Theorem \ref{thm:Thetaisomorphism}. In particular, by construction and by Proposition \ref{prop:defnTheta} we have that the composite map from Theorem \ref{thm:Thetaisomorphism}
is indeed $\Homeo^+(\Sigma)$ equivariant. We shall henceforth focus on the proof that $\dw_\bQ\circ\Theta^{g,l}_\sigma$ is an isomorphism in the given range of degrees.

Definition \ref{defn:dCBra} can be put into a more general context with the following definition and the subsequent discussion; strictly speaking, this is a digression and the reader may directly proceed to Lemma \ref{lem:minimalfactorisation}.
\begin{defn}
\label{defn:tranmongroup}
Let $a\in\Sigma$ and let $(C,\theta,\bt)\in\CBra^{g,l}_{\Sigma\rel a}$, i.e. $\theta\colon C\to\Sigma$ is an $l$-fold branched cover with connected source of genus $g$, branched at some subset $P\subseteq\Sigma\setminus \set a$, and $\bt\colon\theta^{-1}(a)\cong\udll$ is a trivialisation. Let $\phi\colon\pi_1(\SmP,a)\to\fS_l$ be the monodromy group homomorphism associated with $(C,\theta,\bt)$.

The \emph{monodromy group} of $(C,\theta,\bt)$, denoted $\MG(C,\theta,\bt)=\MG(P,\phi)$, is the subgroup of $\fS_l$ given by the image of $\phi$.
If $\theta$ admits at least one simple branch value, the 
\emph{transposition monodromy group} of $(C,\theta,\bt)$, denoted $\TMG(C,\theta,\bt)=\TMG(P,\phi)$, is the subgroup of $\MG(\theta,\bt)$ generated by all transpositions that arise as images along $\phi$ of small loops based at $a$, in the sense of Definition \ref{defn:smallloop}.
\end{defn}
We observe that $\TMG(C,\theta,\bt)$ is a Young subgroup of $\fS_l$, that is, there is an equivalence relation $\sim$ on $\udll$ such that $\TMG(C,\theta,\bt)$ is the subgroup of permutations $\sigma$ with $\sigma(i)\sim i$ for all $i\in \udll$: namely, $\sim$ is generated by $i\sim j$ whenever there is a based small loop in $\SmP$ whose image along $\phi$ is the transposition $(i,j)$.
Moreover, $\TMG(C,\theta,\bt)$ is a normal subgroup of $\MG(C,\theta,\bt)$; since the latter acts transitively on $\udll$, we deduce that all conjugacy classes of $\sim$ have equal size and are permuted transitively by the quotient group $\MG/\TMG(C,\theta,\bt)$. We also deduce that either $\MG(C,\theta,\bt)$ or $\TMG(C,\theta,\bt)$ agrees with the entire $\fS_l$ if and only if the other does.
Finally, we observe that for fixed $\theta\colon C\to\Sigma$ there are several choices for $a$ and $\bt$, but up to conjugation both $\MG(\theta,\bt)$ and $\TMG(\theta,\bt)$ only depend on $\theta$.

\begin{lem}
\label{lem:minimalfactorisation}
Let $\theta\colon C\to\Sigma$ be a connected $l$-fold cover branched at $P\subseteq \Sigma$ and admitting at least one simple branch value. 
Then there exists an essentially unique factorisation of $\theta$ as a composite of branched covers of oriented, connected topological surfaces
$C\xrightarrow{\eta} \bar C\xrightarrow{\bar\theta}\Sigma
$
satisfying the following properties:
\begin{enumerate}
\item $\eta$ has degree $>1$ and admits at least a simple branch value;
\item if $C\xrightarrow{\breve\eta} D\xrightarrow{\breve\theta}\Sigma$ is any factorisation of $\theta$ with $\breve\eta$ of degree $>1$, then $\breve\eta$ factors as a composite $C\xrightarrow{\eta}\bar C\to D$ in such a way that $\bar C\to D\xrightarrow{\breve\theta}\Sigma$ agrees with $\bar\theta$.
\end{enumerate}
\end{lem}
\begin{proof}
We consider on the space $\theta^{-1}(\SmP)\subseteq C$ the equivalence relation $\sim$ generated by $w\sim w'$ whenever the following hold:
\begin{itemize}
\item $\theta(w)=\theta(w')$; let us temporarily denote by $a$ this point;
\item there exists a small loop $\gamma\subseteq\SmP$, based at $a$, whose local monodromy with respect to $\theta$ is a transposition, and lifting to a path in $\theta^{-1}(\SmP)$ from $w$ to $w'$.
\end{itemize}
We consider $\sim$ as a subset of $\theta^{-1}(\SmP)^2$; the closure $\bar\sim\subseteq C^2$ is again an equivalence relation on $C$, and we let $\bar C=C/\bar\sim$, and factor $\theta$ through the quotient map $\eta\colon C\to \bar C$ followed by a further map $\bar\theta\colon \bar C\to \Sigma$. Note that $\bar C$ is again a topological surface.

The assumption on $\theta$ imples that $\sim$, and hence also $\bar\sim$, is not the trivial equivalence relation, so that $\eta$ has degree $>1$; moreover, if $z\in P$ is a simple branch value for $\theta$, then $\bar\theta$ must be \'etale over $z$ and $\eta$ must be \'etale over all points of $\bar\theta^{-1}(z)$ except one, which is a simple branch value for $\eta$. This proves (1).

To prove (2), we argue that given any factorisation $C\xrightarrow{\breve\eta} D\xrightarrow{\breve\theta} \Sigma$ of $\theta$, the equivalence relation on $C$ yielding $D$ as a quotient must be coarser than $\bar\sim$. To see this, it suffices to show that the equivalence relation on $\theta^{-1}(\SmP)$ yielding $\breve\theta^{-1}(\SmP)$ as quotient is coarser than $\sim$. To wit, let $w\sim w'$ be points in $\theta^{-1}(a)$, for some $a\in\SmP$, and assume that there is a simple loop $\gamma$ in $\SmP$ spinning around a simple branch value $z\in P$ and lifting to a path in $\theta^{-1}(\SmP)$ from $w$ to $w'$. Since $P$ is a simple branch value, $\breve\theta$ must be \'etale over $P$, and in particular the lift of $\gamma$ that starts at $\breve\eta(w)$ also ends at $\breve\eta(w)$. But this lift is the projection along $\breve\eta$ of the lift of $\gamma$ along $\theta$ from $w$ to $w'$, whence $\breve\eta(w)=\breve\eta(w')$.
\end{proof}

The subspace $\dCBra^{g,l}_{\Sigma,\coarse}$ may now be characterised as the subspace of $\CBra^{g,l}_{\Sigma,\coarse}$ comprising those $\theta\colon C\to\Sigma$ whose transposition monodromy group and monodromy group both coincide with $\fS_l$; equivalently, those $\theta$ whose associated $\bar\theta$ as in Lemma \ref{lem:minimalfactorisation} has degree 1. 
\subsection{Coloured configuration spaces}
Recall that $\Theta^{g,l}_\Sigma$ is by definition modeled as a map of simplicial $\bQ$-vector spaces with target $\Sing_\bullet(\CBra^{g,l,\infty}_{\Sigma,\coarse},\infty;\bQ)$. We next model also $\dw_\bQ$ as a map of simplicial $\bQ$-vector spaces.
\begin{nota}
We denote by $\Sing'_\bullet(\dCBra^{g,l,\infty}_{\Sigma,\coarse},\infty;\bQ)$ the quotient of the simplicial $\bQ$-vector space $\Sing'_\bullet(\CBra^{g,l,\infty}_{\Sigma,\coarse},\infty;\bQ)$ by all simplices $t\colon\Delta^p\to  \CBra^{g,l,\infty}_{\Sigma,\coarse}$ such that $\alpha(t)\notin\dCBra^{g,l}_{\Sigma,\coarse}$.
\end{nota}
By
\cite[Corollary A.9.4]{HA},
we may model the map $\dw_\bQ$ from Theorem \ref{thm:Thetaisomorphism} by the quotient map of simplicial $\bQ$-vector spaces
\[
\dw_\bQ\colon \Sing'_\bullet(\CBra^{g,l,\infty}_{\Sigma,\coarse},\infty;\bQ)\to\Sing'_\bullet(\dCBra^{g,l,\infty}_{\Sigma,\coarse},\infty;\bQ).
\]
\begin{nota}
We define a filtration $F_\star$ on $\Sing'_\bullet(\dCBra^{g,l,\infty}_{\Sigma,\coarse},\infty;\bQ)$ as the image of the filtration $F_\star$ on $\Sing'_\bullet(\CBra^{g,l,\infty}_{\Sigma,\coarse},\infty;\bQ)$.
\end{nota}
We have thus modeled the composite from Theorem \ref{thm:Thetaisomorphism} as a composite
\[
\ER_\bullet(\Sigma,\cAh(l)_+)_{2\upsilon}\xrightarrow{\Theta^{g,l}_\Sigma}
\Sing'_\bullet(\CBra^{g,l,\infty}_{\Sigma,\coarse},\infty;\bQ)\xrightarrow{\dw_\bQ}\Sing'_\bullet(\dCBra^{g,l,\infty}_{\Sigma,\coarse},\infty;\bQ)
\]
of filtered maps of simplicial $\bQ$-vector spaces. Theorem \ref{thm:Thetaisomorphism} will follow from the following, which we state as a lemma for future reference.
\begin{lem}
\label{lem:first_reduction}
Assume that for all $1\le k\le 2\upsilon$ the map of simplicial $\bQ$-vector spaces $(F_k/F_{k-1})(\dw_\bQ\circ\Theta^{g,l}_\Sigma)$ is injective on homology groups in all degrees, and an isomorphism in degrees $*>2k-(2k-2\upsilon-J_l)/I_l$. Then Theorem \ref{thm:Thetaisomorphism} holds true.
\end{lem}
\begin{proof}
The map of filtered simplicial $\bQ$-vector spaces $\dw_\bQ\circ\Theta^{g,l}_\Sigma$ gives rise to a comparison map between the spectral sequences computing the homology of the source and target of $\dw_\bQ\circ\Theta^{g,l}_\Sigma$. The hypothesis, together with the assumption $I_l\ge1$ from Notation \ref{nota:goodconstantsIJ}, implies that the map on $E^1$-pages is an isomorphism on $E^1_{k,*-k}$ for all $*>2k-(2k-2\upsilon-J_l)/I_l\ge4\upsilon-(2\upsilon-J_l)/I_l$ and an injection on $E^1_{k,*-k}$ for all $k,*$, whence a standard spectral sequence argument proves the claim.
\end{proof}
Lemma \ref{lem:first_reduction} is the first of a sequence of reductions of Theorem \ref{thm:Thetaisomorphism} that we will make; the subsequent reductions will be Lemmas \ref{lem:second_reduction} and \ref{lem:third_reduction}, and Example \ref{ex:tbrvexplained}.

We next analyse the maps $(F_k/F_{k-1})(\dw_\bQ\circ\Theta^{g,l}_\Sigma)$ in terms of coloured configuration spaces, using the following definition.

\begin{defn}
\label{defn:colconf}
Recall Definition \ref{defn:Ran}. Let $X$ be a space and let $S$ be a finite set.
We denote by $\Conf_k(X,S)$ the space of configurations of $k$ points in $X$ with labels in $S$. It is defined as the quotient
\[
\Conf_k(X,S)\coloneqq\pa{\set{(z_1,\dots,z_k)\in X^k\,|\,z_i\neq z_j}\times S^k}/\fS_k.
\]
It is a finite \'etale cover of $\Conf_k(X)$.
If $S$ carries a grading $S\to\bN$, we denote by $\Conf_k(X,S)_n$ the subspace of configurations whose sum of labels is equal to $n$.
If $\uk=(k_s)_{s\in S}$ is a sequence of natural numbers with $k=\sum_{s\in S}k_s$,
we let $\Conf_{\uk}(X)$ denote the $\uk$-coloured configuration space of $X$, i.e. the subspace of $\Conf_k(X,S)$ containing configurations in which precisely $k_s$ points carry the label $s\in S$.
\end{defn}
\begin{nota}
Recall Notation \ref{nota:fP}. We consider on $\fP(l)_+$ the grading $N\colon\fP(l)_+\to\bN$ as in Definition \ref{defn:norm}
\end{nota}
The simplicial $\bQ$-vector space $(F_k/F_{k-1})\ER_\bullet(\Sigma,\cAh(l)_+)_{2\upsilon}$ may be identified with the simplicial $\bQ$-vector subspace of $\Sing_\bullet(\Conf_k(\Sigma,\fP(l)_+)_{2\upsilon}^\infty,\infty;\bQ)$ spanned by those simplices $s\colon \Delta^p\to\Conf_k(\Sigma,\fP(l)_+)_{2\upsilon}^\infty$ such that $s^{-1}(\infty)$ is empty or it is a ``final subsimplex'' of $\Delta^p$, i.e. it has the form $\delta_0^i(\Delta^{p-i})$ for some $1\le i\le p-1$. By 
\cite[Corollary A.9.4]{HA}, the inclusion
\begin{equation}
\label{eq:factohom_colconf_iso}
(F_k/F_{k-1})\ER_\bullet(\Sigma,\cAh(l)_+)_{2\upsilon}\hookrightarrow\Sing_\bullet(\Conf_k(\Sigma,\fP(l)_+)_{2\upsilon}^\infty,\infty;\bQ)
\end{equation}
is a quasi-isomorphism.
\begin{nota}
\label{nota:Xi}
We denote by $\Xi_k$ the set of all sequences $\uk=(k_\lambda)_{\lambda\in\fP(l)_+}$ of natural numbers with $\sum_{\lambda\in\fP(l)_+}k_\lambda=k$ and $\sum_{\lambda\in\fP(l)_+}k_\lambda\cdot N(\lambda)=2\upsilon$. The set $\Xi_k$ is in natural bijection with $\pi_0(\Conf_k(\Sigma,\fP(l)_+)_{2\upsilon})$.
\end{nota}

We split $(F_k/F_{k-1})\ER_\bullet(\Sigma,\cAh(l)_+)_{2\upsilon}$ as a direct sum of simplicial $\bQ$-vector spaces $(F_k/F_{k-1})\ER_\bullet(\Sigma,\cAh(l)_+)_{\uk}$ for varying $\uk\in\Xi_k$
: the simplicial $\bQ$-vector subspace corresponding to $\uk$ is spanned by basis elements $(s,\sca{\lambda_z}_{z\in\alpha(s)})$ with $\#\alpha(s)=k$ and $\#\set{z\in\alpha(s)\,|\,\lambda_z=\lambda}=k_\lambda$ for all $\lambda\in\fP(l)_+$.
The quasi-isomorphism 
\eqref{eq:factohom_colconf_iso} restricts to quasi-isomorphisms
\[
(F_k/F_{k-1})\ER_\bullet(\Sigma,\cAh(l)_+)_{\uk}\overset{\simeq}{\hookrightarrow}\Sing_\bullet(\Conf_{\uk}(\Sigma)^\infty,\infty;\bQ).
\]

Now recall Definition \ref{defn:brv}: the finite \'etale map 
\[
\brv\colon (F_k\setminus F_{k-1})\CBra^{g,l}_{\Sigma,\coarse}\to\Conf_k(\Sigma)
\]
naturally lifts to a finite \'etale map
\[
\tbrv\colon (F_k\setminus F_{k-1})\CBra^{g,l}_{\Sigma,\coarse}\to\Conf_k(\Sigma,\fP(l)_+)_{2\upsilon},
\]
by recording not only the branch values, but also the conjugacy class of their local monodromies. Our second reduction towards Theorem \ref{thm:Thetaisomorphism} is the following.
\begin{lem}
\label{lem:second_reduction}
Assume that for all $1\le k\le2\upsilon$ the restricted \'etale map
\begin{equation}
\label{eq:tbrvrestricted}
\tbrv\colon(F_k\setminus F_{k-1})\dCBra^{g,l}_{\Sigma,\coarse}\to\Conf_k(\Sigma,\fP(l)_+)_{2\upsilon}
\end{equation}
induces an isomorphism in rational cohomology in the range of cohomological degrees $*<(2k-2\upsilon-J_l)/I_l$. Then the assumptions of Lemma \ref{lem:first_reduction} hold true.
\end{lem}
\begin{proof}
We may identify $(F_k/F_{k-1})\Sing'_\bullet(\CBra^{g,l,\infty}_{\Sigma,\coarse},\infty;\bQ)$ with the simplicial $\bQ$-vector subspace of $\Sing_\bullet\pa{\pa{(F_k\setminus F_{k-1})\CBra^{g,l}_{\Sigma,\coarse}}^\infty,\infty;\bQ}$ spanned by simplices $s\colon\Delta^p\to\pa{(F_k\setminus F_{k-1})\CBra^{g,l}_{\Sigma,\coarse}}^\infty$ such that $s^{-1}(\infty)$ is a terminal subsimplex of $\Delta^p$. A similar remark holds for $\pa{(F_k\setminus F_{k-1})\dCBra^{g,l}_{\Sigma,\coarse}}^\infty$. By
\cite[Corollary A.9.4]{HA}, the map $(F_k/F_{k-1})\dw_\bQ$ models the collapse map
\[
\pa{\pa{(F_k\setminus F_{k-1})\CBra^{g,l}_{\Sigma,\coarse}}^\infty,\infty}_!\bQ\to\pa{\pa{(F_k\setminus F_{k-1})\dCBra^{g,l}_{\Sigma,\coarse}}^\infty,\infty}_!\bQ
\]
which under Poincar\'e--Lefschetz duality corresponds to the shifted restriction map along the open inclusion
\[
\pa{(F_k\setminus F_{k-1})\CBra^{g,l}_{\Sigma,\coarse}}_*\bQ[2k]\to\pa{(F_k\setminus F_{k-1})\dCBra^{g,l}_{\Sigma,\coarse}}_*\bQ[2k].
\]

Similarly, $(F_k/F_{k-1})\Theta^{g,l}_\Sigma$ models a ``weighted'' version of the shifted restriction map along the finite \'etale cover $\tbrv$
\[
\pa{\Conf_k(\Sigma,\fP(l)_+)_{2\upsilon}}_*\bQ\to\pa{(F_k\setminus F_{k-1})\CBra^{g,l}_{\Sigma,\coarse}}_*\bQ,
\]
described precisely as follows. Let $Z\subseteq (F_k\setminus F_{k-1})\CBra^{g,l}_{\Sigma,\coarse}$ be a connected component, and let $\Conf_{\uk}(\Sigma)=\tbrv(Z)$. Observe that if $\theta\colon C\to\Sigma$ and $\theta'\colon C'\to\Sigma$ represent classes in $Z$, then there is a homeomorphism $C\cong C'$ covering a homeomorphism $\Sigma\cong\Sigma$; in particular the number $n=\#\Aut(\theta)\ge1$ is also equal to $\#\Aut(\theta')$, i.e. it only depends on $Z$. Then $\Theta^{g,l}_\Sigma$ restricts to a simplicial map
\[
\ER_\bullet(\Sigma,\cAh(l)_+)_{\uk}\to\Sing'_\bullet(Z^\infty,\infty;\bQ)\subseteq\Sing'_\bullet(\CBra^{g,l,\infty}_{\Sigma,\coarse},\infty;\bQ)
\]
which models a map of rational spectra
\[
\pa{\Conf_{\uk}(\Sigma)^\infty,\infty}_!\bQ\to (Z^\infty,\infty)_!\bQ,
\]
and under Poincar\'e--Lefschetz duality corresponds to the map
\[
\Conf_{\uk}(\Sigma)_*\bQ[2k]\to Z_*\bQ[2k]
\]
given by $1/n$ times restriction along $\tbrv$; note that the number $n$ depends on $Z$.

It follows that $(F_k/F_{k-1})(\dw_\bQ\circ\Theta^{g,l}_\Sigma)$ may be identified with a mild variation of the restriction map associated with \eqref{eq:tbrvrestricted}, in which we multiply by different invertible rational numbers on different components of $(F_k\setminus F_{k-1})\dCBra^{g,l}_{\Sigma,\coarse}$. Being a finite \'etale cover, \eqref{eq:tbrvrestricted} induces injections in rational cohomology groups. Assuming \eqref{eq:tbrvrestricted} induces an isomorphism in the range from the statement, Poincar\'e--Lefschetz duality yields the corresponding isomorphism range for $(F_k/F_{k-1})(\dw_\bQ\circ\Theta^{g,l}_\Sigma)$ in homology.
\end{proof}
There is one further reduction towards Theorem \ref{thm:Thetaisomorphism} that we wish to make now.
\begin{nota}
For finite subsets $A\subseteq B\subseteq\Sigma$ as in Definition \ref{defn:Bra} and for $k\ge0$ we denote by $F_k\dCBra^{g,l}_{\Sigma\rel A\et B,\coarse}$ the preimage of $F_k\dCBra^{g,l}_{\Sigma,\coarse}$ along the map $\dCBra^{g,l}_{\Sigma\rel A\et B,\coarse}\to\dCBra^{g,l}_{\Sigma,\coarse}$ forgetting trivialisations over $A$ and the \'etale property over $B$.
For simplicity, we denote by
$\tbrv$ also the composite map
\[
(F_k\setminus F_{k-1})\CBra^{g,l}_{\Sigma\rel A\et B,\coarse}\to(F_k\setminus F_{k-1})\CBra^{g,l}_{\Sigma\et B,\coarse}\xrightarrow{\tbrv}\Conf_k(\Sigma\setminus B,\fP(l)_+)_{2\upsilon}.
\]
\end{nota}
\begin{lem}
\label{lem:third_reduction}
Assume that the restricted finite \'etale map
\begin{equation}
\label{eq:tbrv_generalised_dBra}
\tbrv\colon(F_k\setminus F_{k-1}) \dCBra^{g,l}_{\Sigma\rel a\et B,\coarse}\to\Conf_k(\Sigma\setminus B,\fP(l)_+)_{2\upsilon}
\end{equation}
induces an isomorphism in rational cohomology in the range of cohomological degrees $*<(2k-2\upsilon-J_l)/I_l$ for any non-empty finite subset $B\subseteq\Sigma$ and any $a\in\Sigma$. Then the assumptions of Lemma \ref{lem:second_reduction} hold true.
\end{lem}
\begin{proof}
Let $\hat B\subseteq \Sigma$ be a subset of cardinality $2\upsilon+1$. For any non-empty subset $B\subseteq \hat B$ we may restrict \eqref{eq:tbrvrestricted} to a finite \'etale map
\begin{equation}
\label{eq:tbrvrestricted_etale_B_nonempty}
\tbrv\colon(F_k\setminus F_{k-1})\dCBra^{g,l}_{\Sigma\et B,\coarse}\to\Conf_k(\Sigma\setminus B,\fP(l)_+)_{2\upsilon}
\end{equation}
The sources and targets of \eqref{eq:tbrvrestricted_etale_B_nonempty} for varying $\emptyset\neq B\subseteq \hat B$ form a hypercover of the source and target of \eqref{eq:tbrvrestricted}, so that by a standard Mayer--Vietoris spectral sequence argument it suffices to show that \eqref{eq:tbrvrestricted_etale_B_nonempty} induces isomorphisms in rational cohomology in degrees $*<(2k-2\upsilon-J_l)/I_l$. For this, we may pick $a\in B$ and consider the composite 
\[
(F_k\setminus F_{k-1})\dCBra^{g,l}_{\Sigma\rel a\et B}\to (F_k\setminus F_{k-1})\dCBra^{g,l}_{\Sigma\et B,\coarse}\xrightarrow{\eqref{eq:tbrvrestricted_etale_B_nonempty}}\Conf_k(\Sigma\setminus B,\fP(l)_+)_{2\upsilon};
\]
by assumption the previous composite induces a rational cohomology isomorphism in degrees $*<(2k-2\upsilon-J_l)/I_l$, but we also know that both maps, being finite \'etale, induce injective morphisms in rational cohomology in all degrees.
\end{proof}

\subsection{Connectivity of Hurwitz spaces and modules}
We next make a digression about Hurwitz spaces in the sense of \cite{LL2}. Although we may restrict our attention to symmetric groups, the arguments we present work in the more general adpoted setting. Our goal in this subsection is to verify that \cite[Theorem1.4.9]{LL2}
may be applied to prove that the assumptions of Lemma \ref{lem:third_reduction} hold true.

In the entire subsection we fix a finite group $G$, and let $c=c_1\sqcup\dots\sqcup c_v$ be a union of $v$ conjugacy classes in $G$, with the requirement that $c_1$ generates $G$. We also fix a compact, connected, oriented topological surface $\cF$ of genus $h\ge0$ with $n+1\ge 1$ boundary curves denoted $\partial_i\cF$ for $0\le i\le n$, and a basepoint $*\in\partial_0\cF$.
\begin{ex}
\label{ex:mainsetting}
Our applications will be with $G=\fS_l$, $c\subseteq \fS_l\setminus\set{\one}$ being a conjugation-invariant subset containing all transpositions, and $c_1$ being the conjugacy class of transpositions.
\end{ex}

\begin{nota}
\label{nota:standardgenerators}
We fix a sequence $\bar P_\infty=\set{\bar z_i\,|\,i\in\bN}$ of distinct points in $\mcF$ that converges to $*\in\partial_0\cF$, and let $\bar P_k\coloneqq\set{\bar z_1,\dots,\bar z_k}$ be our preferred basepoint of $\Conf_k(\mcF)$. For $k=0$ we also set $\bar P_0=\emptyset$ for uniform notation.

Further, we fix for all $k\ge0$ a system of simple loops, all disjoint from $\bar P_\infty$
\[
\alpha_1,\beta_1,\dots,\alpha_h,\beta_h,\delta_0,\dots,\delta_n,\gamma_1,\gamma_2,\gamma_3,\dots\in\pi_1(\FmbP_\infty,*)
\]
as in Figure \ref{fig:loopscF},
so that in particular, for all $k\ge1$, $\pi_1(\FmbP_k,*)$ is freely generated by
the loops $\alpha_1,\beta_1,\dots,\alpha_h,\beta_h,\delta_0,\dots,\delta_n,\gamma_1,\dots,\gamma_k$.
\begin{figure}
\centering
\begin{tikzpicture}
\draw[fill=gray!20] (0,0) rectangle (12,3)node[above left]{$\partial_0\cF$};
\draw[line width=3pt](0,0) -- (12,0);
\node[below left] at (10,0){$*$};
\fill[gray!40] (1.5,1.5) ellipse (1.2cm  and 0.7cm);
\node[above] at (1.1,2){$V_2$};
\draw[dashed, bend left](0.6,0) to node[left]{$\zeta_2$} (0.6,1.2);
\node at (1,1.5){$\bullet\ \bar z_1$};
\node at (2.1,1.3){$\bullet\ \bar z_2$};
\node at (3.2,1.1){$\bullet\ \bar z_3$};
\draw[looseness=1.5, midarrow] (1,0) to[in=-90, out=90](0.6,1.5) to[out=90, in=90] (1.4,1.5)  to[out=-90, in=90] node[above left]{$\gamma_1$} (1,0);
\draw[looseness=1.5, midarrow] (2.1,0) to[in=-90, out=90](1.7,1.3) to[out=90, in=90] (2.5,1.3)  to[out=-90, in=90] node[above left]{$\gamma_2$} (2.1,0);
\draw[looseness=1.5, midarrow] (3.2,0) to[in=-90, out=90](2.8,1.1) to[out=90, in=90] (3.6,1.1)  to[out=-90, in=90] node[above left]{$\gamma_3$} (3.2,0);
\draw[looseness=2] (3.5,3) to[out=90, in=90](5.5,3);
\draw[looseness=2] (3.9,3) to[out=90, in=90](5.1,3);
\fill[looseness=2, gray!20](3.5,2.7)to(3.5,3)to[out=90, in=90](5.5,3)to (5.5,2.7) to (5.1,2.7)to (5.1,3) to[out=90, in=90] (3.9,3) to (3.9,2.7) to(3.5,2.7);
\draw[looseness=2] (3.5,3) to[out=90, in=90](5.5,3);
\draw[looseness=2] (3.9,3) to[out=90, in=90](5.1,3);
\draw[looseness=2,midarrow] (4.5,0) to[in=-90, out=100] (3.7,3)node[below left]{$\alpha_1$} to[out=90, in=90] (5.3,3)  to[out=-90, in=80]  (4.5,0);

\begin{scope}[xshift=0.8cm]
\fill[looseness=2, gray!20](3.5,2.7)to(3.5,3)to[out=90, in=90](5.5,3)to (5.5,2.7) to (5.1,2.7)to (5.1,3) to[out=90, in=90] (3.9,3) to (3.9,2.7) to(3.5,2.7);
\draw[looseness=2] (3.5,3) to[out=90, in=90](5.5,3);
\draw[looseness=2] (3.9,3) to[out=90, in=90](5.1,3);
\end{scope}
\draw[looseness=2,midarrow] (4.5,0) to[in=-90, out=70](6.1,3) node[below left]{$\beta_1$}  to[out=90, in=90] (4.5,3) to[out=-90, in=90] (4.5,0);
\begin{scope}[xshift=3.3cm]
\fill[looseness=2, gray!20](3.5,2.7)to(3.5,3)to[out=90, in=90](5.5,3)to (5.5,2.7) to (5.1,2.7)to (5.1,3) to[out=90, in=90] (3.9,3) to (3.9,2.7) to(3.5,2.7);
\draw[looseness=2] (3.5,3) to[out=90, in=90](5.5,3);
\draw[looseness=2] (3.9,3) to[out=90, in=90](5.1,3);
\draw[looseness=2,midarrow] (4.5,0) to[in=-90, out=100] (3.7,3)node[below left]{$\alpha_2$} to[out=90, in=90] (5.3,3)  to[out=-90, in=80]  (4.5,0);
\begin{scope}[xshift=0.8cm]
\fill[looseness=2, gray!20](3.5,2.7)to(3.5,3)to[out=90, in=90](5.5,3)to (5.5,2.7) to (5.1,2.7)to (5.1,3) to[out=90, in=90] (3.9,3) to (3.9,2.7) to(3.5,2.7);
\draw[looseness=2] (3.5,3) to[out=90, in=90](5.5,3);
\draw[looseness=2] (3.9,3) to[out=90, in=90](5.1,3);
\end{scope}
\draw[looseness=2,midarrow] (4.5,0) to[in=-90, out=70](6.1,3) node[below left]{$\beta_2$}  to[out=90, in=90] (4.5,3) to[out=-90, in=90] (4.5,0);
\end{scope}
\draw[fill=white] (10,1.5) node{$\partial_1\cF$} circle (0.4cm);
\draw[fill=white] (11.3,1.5) node{$\partial_2\cF$} circle (0.4cm);
\draw[looseness=1.5, midarrow] (10,0) to[in=-90, out=90](9.4,1.5) to[out=90, in=90] (10.6,1.5)  to[out=-90, in=90] node[ left]{$\delta_1$} (10,0);
\draw[looseness=1.5, midarrow] (11.3,0) to[in=-90, out=90](10.7,1.5) to[out=90, in=90] (11.9,1.5)  to[out=-90, in=90] node[ left]{$\delta_2$} (11.3,0);
\end{tikzpicture}
\caption{The generators of $\pi_1(\FmbP_k,*)$ in the case $k=3$, $n+1=3$ and $h=2$. The entire bottom horizontal edge represents the basepoint $*$.}
\label{fig:loopscF}
\end{figure}
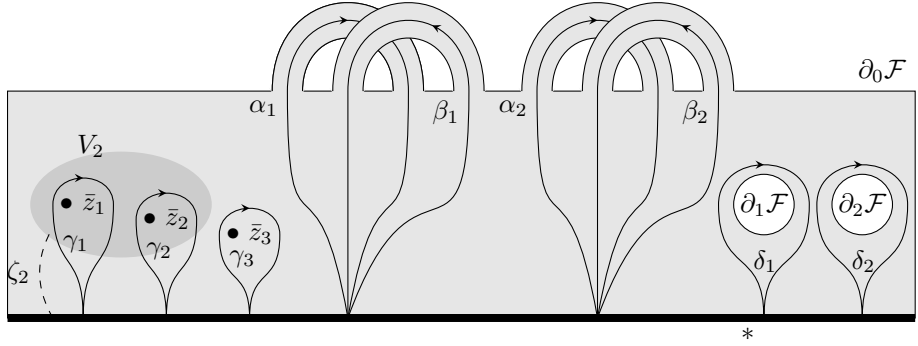
The loop
$\delta_0$ is the \emph{counterclockwise} boundary loop around $\partial_0\cF$, and satisfies the equality
\begin{equation}
\label{eq:delta0formula}    
\delta_0=\gamma_1\dots\gamma_k[\alpha_1,\beta_1]\dots[\alpha_h,\beta_h]\delta_1\dots\delta_n\in\pi_1(\FmbP_k,*),
\end{equation}
where $[\alpha_i,\beta_i]=\alpha_i\beta_i\alpha_i^{-1}\beta_i^{-1}$. We sometimes consider $\pi_1(\cF,*)$ as the subgroup of $\pi_1(\FmbP_k,*)$ generated by the elements $\alpha_1,\beta_1,\dots,\alpha_h,\beta_h,\delta_0,\dots,\delta_n$.

Finally, for each $k\ge1$ we fix a disc $V_k$ whose closure $\bar V_k$ is contained in $\mcF$, and such that $\bar P_\infty\cap V_k=\bar P_\infty\cap\bar V_k=\bar P_k$. We also fix a path $\zeta_k$ joining $*$ with $\partial V_k$ and otherwise disjoint from $V_k\cup\bar P_\infty$. We assume that the concatenation $\zeta_k\star\partial V_k\star\zeta_k^{-1}$ is homotopic, as a based loop in $\FmbP_\infty$, to the product $\gamma_1\dots\gamma_k$.\end{nota}
\begin{nota}
\label{nota:braidgroup}
We denote by $B_k^\cF=\pi_1\pa{\Conf_k(\mcF),\bar P_k}$ the $k$\textsuperscript{th} braid group of $\cF$. It is endowed with a point-pushing map $B_k^\cF\hookrightarrow\pi_0(\Homeo^+(\FmbP_k,\partial\cF))$. For $1\le m\le k$ we identify the standard braid group $B_m$ with the subgroup of the image of the point-pushing map spanned by homotopy classes of homeomorphisms supported on $V_m\setminus \bar P_m$, i.e. restricting to the identity outside $V_m\setminus\bar P_m$.
\end{nota}
\begin{defn}
\label{defn:CHur}
We let $\Hom^c(\pi_1(\FmbP_k,*),G)\subseteq \Hom(\pi_1(\FmbP_k,*),G)$ denote the subset of those group homomorphisms $\phi$ whose evaluation at each loop $\gamma_1,\dots,\gamma_k$ belongs to $c$.
We let $\Hur^{G,c}_k(\cF)$ denote the finite \'etale cover of $\Conf_k(\mcF)$ corresponding to the action of $B_k^\cF$ on $\Hom^c(\pi_1(\FmbP_k,*),G)$ induced by the point-pushing map. Concretely, $\Hur^{G,c}_k(\cF)$ contains pairs $(P,\phi)$ consisting of a configuration $P\in\Conf_k(\mcF)$ and a group homomorphism $\phi\colon\pi_1(\cF\setminus P,*)\to G$ whose evaluation at any small loop spinning clockwise around a point of $P$ belongs to $c$.

We let $\CHom^c(\pi_1(\FmbP_k,*),G)$ denote the subset of surjective group homomorphism; it is $B^\cF_k$-invariant, and we let $\dCHur_k^{G,c}(\cF)\subseteq\Hur^{G,c}_k(\cF)$ denote the corresponding subspace.

If $\uk=(k_1,\dots,k_v)$ is a sequence of natural numbers with $k=\sum_{i=1}^vk_i$, and  $\unu=(\nu_1,\dots,\nu_n)$ is a sequence of conjugacy classes in $G$ (possibly different than those contained in $c$), and $\omega\in G$ is any element, we let $\CHom^c_{\uk,\unu,\omega}(\pi_1(\FmbP_k,*),G)\subseteq\CHom^c(\pi_1(\FmbP_k,*),G)$ denote the subset of group homomorphisms $\phi$ such that:
\begin{itemize}
\item there are precisely $k_i$ values of $1\le j\le k$ for which $\phi(\gamma_j)\in c_j$;
\item $\phi(\delta_i)\in\nu_i$, and $\phi(\delta_0)=\omega$.
\end{itemize}
It is a $B^\cF_k$-invariant subset, and we let $\dCHur^{G,c}_{\uk,\unu,\omega}(\cF)\subseteq\dCHur^{G,c}_k(\cF)$ denote the corresponding subspace. For reasons that will become clear in Example \ref{ex:tbrvexplained}, we also denote by $\tbrv\colon\dCHur^{G,c}_{\uk,\unu,\omega}(\cF)\to\Conf_{\uk}(\mcF)$ the covering map sending $(P,\phi)$ to the set $P$, each of whose points is coloured by the conjugacy class of its local monodromy with respect to $\phi$.
\end{defn}

\begin{ex}
\label{ex:tbrvexplained}
Let $B=\set{a,b_1,\dots,b_n}\subseteq \Sigma$ be a finite subset of cardinality $n+1$, and fix disjoint open discs $\uU=(U_0,\dots, U_n)$ in $\Sigma$, with $a\in U_0$ and $b_i\in U_i$. Let $\cF=\SmU$, and let $*\in\partial_0\cF$ be any point.

Consider the setting of Example \ref{ex:mainsetting}, and 
let $\uk\in\Xi_k$ as in Notation \ref{nota:Xi}, with the number $k_{\trsp}$ associated with the conjugacy class of transpositions being strictly positive. 
Let $\unu=\uone$ be the sequence repeating $n$ times the conjugacy class of $\one\in\fS_l$, and let also $\omega=\one\in\fS_l$. 
Then there is a commutative diagram as follows, in which the horizontal arrows are both open inclusions and weak equivalences:
\[
\begin{tikzcd}
\dCHur^{\fS_l,c}_{\uk,\uone,\one}(\cF)\ar[d,"\tbrv"]\ar[r,hook,"\simeq"]&((F_k\setminus F_{k-1})\dCBra^{g,l}_{\Sigma\rel a\et B})_{\uk}\ar[d,"\tbrv"]\\
\Conf_{\uk}(\mcF)\ar[r,hook,"\simeq"]&\Conf_{\uk}(\SmB).
\end{tikzcd}
\]
Here, the top inclusion sends a pair $(P,\phi)$ to the essentially unique triple $(C,\theta,\bt)$
such that:
\begin{itemize}
\item $\theta\colon C\to\Sigma$ is \'etale over the closure of $\uU$, and in particular we may extend $\bt$ to a trivialisation of $\theta$ over $U_0$;
\item the restriction of $\theta$ over $\cF$, with the given trivialisation over $*\in\partial_0\cF$, is isomorphic to the $l$-fold branched cover trivialised over $*$ associated with $(P,\phi)$.
\end{itemize}
The fact that the top map lands inside $\dCBra^{g,l}_{\Sigma\rel a\et B}\subseteq\CBra^{g,l}_{\Sigma\rel a\et B}$ is a consequence of the discussion from Subsection \ref{subsec:TMG} and the assumption $k_{\trsp}\ge1$. Lemma \ref{lem:third_reduction} invites us to study the cohomology behaviour of the right vertical map, but we may henceforth focus on the left vertical map instead, at least to deal with those $\uk\in\Xi_k$ satisfying $k_{\trsp}\ge1$. If $k_\trsp=0$, the top right space is empty, according to the definition of $\dCBra^{g,l}_{\Sigma\rel a\et B}$ that we gave, and we have no such square.
\end{ex}
\begin{nota}
\label{nota:uphi}
We represent a group homomorphism $\phi\colon\pi_1(\FmbP_k,*)\to G$ by the list $\uphi\coloneqq(\phi(\gamma_1),\dots,\phi(\gamma_k),\phi(\alpha_1),\phi(\beta_1),\dots,\phi(\alpha_h),\phi(\beta_h),\phi(\delta_1),\dots,\phi(\delta_n)
)$ of its values at the generators from Notation \ref{nota:standardgenerators}, taken in this order. We say that $\uphi$ is \emph{standard} if there is $0\le r\le k$ such that $\phi(\gamma_i)\in c_1$ if and only if $i\le r$.
\end{nota}
We next single out a few braids in $B_k^\cF$; our goal is to study their action on the set $\Hom(\pi_1(\FmbP,*),G)\cong G^{k+2h+n}$ to a sufficient extent for our purposes.
\begin{nota}
For $1\le m\le k$ and $\gamma\in\pi_1(\cF\setminus(\bar P_k\setminus \bar P_m),*)$ we denote by $\fb_m(\gamma)\in B_k^\cF$ any braid corresponding to a homeomorphism $\FmbP_k\to\FmbP_k$ obtained by dragging $V_k\setminus\bar P_k$ along $\gamma$. Here we use the paths $\zeta_k$ from Notation \ref{nota:standardgenerators} to let a based loop at $*$ drag $V_k$; the resulting element in $B_k^\cF$ is uniquely determined up to right multiplication by an element in the subgroup $B_m\subseteq B_k^\cF$. In particular, $\fb_1(\gamma)$ is uniquely determined.
\end{nota}
We shall make use of the observations contained for future reference in the following lemma, whose diagramatic proof is standard and left to the reader. See \cite[Remark 8.1.6]{EL} for similar computations.
\begin{lem}
\label{lem:braidactionrules}
Let $\uphi$ be as in Notation \ref{nota:uphi}, and fix $1\le m\le k$ such that there is $\fg\in G$ with $\phi(\gamma_1)=\dots=\phi(\gamma_m)=\fg$. Then the following hold.
\begin{itemize}
\item Suppose first that $\fg^m=\one\in G$, let $\gamma\in\pi_1(\cF\setminus(\bar P_k\setminus \bar P_m),*)$, and let $\tilde\gamma$ be any lift of $\gamma$ to $\pi_1(\cF\setminus(V_m\cup\bar P_k))$; then $\fb_m(\gamma)$ sends $\uphi$ to the sequence obtained from $\uphi$ by replacing the first $m$ entries, all equal to $\fg$, with $m$ entries all equal to $\phi(\tilde\gamma)\fg\phi(\tilde\gamma)^{-1}$; this is in particular independent of the choice of $\fb_m(\gamma)$, since $\uphi$ is fixed by the action of $B_m$;
\item Now suppose $m=1$; then, for the following values of $\gamma$, the image  $\uphi'$ of $\uphi$ along $\fb_1(\gamma)$ satisfies the corresponding properties:
\begin{enumerate}
\item{\bf $\gamma=\gamma_i$ for some $2\le i\le k$:} we have $\phi'(\gamma_i)=\fg\phi(\gamma_i)\fg^{-1}$, and $\uphi'$ coincides with $\uphi$ in all subsequent positions;
\item{\bf $\gamma=\alpha_i$ for some $1\le i \le h$:} we have $\phi'(\alpha_i)=\fg\phi(\alpha_i)\fg^{-1}$, $\phi'(\beta_i)=\phi(\beta_i)\fg^{-1}$, and $\uphi'$ coincides with $\uphi$ in all subsequent positions;
\item{\bf $\gamma=\alpha_i\beta_i\alpha_i^{-1}$ for some $1\le i \le h$:} we have $\phi'(\alpha_i)=\fg\phi(\alpha_i)$, $\phi'(\beta_i)=\phi(\beta_i)$, and $\uphi'$ coincides with $\uphi$ in all subsequent positions;
\item{\bf $\gamma=\delta_i$ for some $1\le i\le n$:} we have $\phi'(\delta_i)=\fg\phi(\delta_i)\fg^{-1}$, and $\uphi'$ coincides with $\uphi$ in all subsequent positions.
\end{enumerate}
\end{itemize}
\end{lem}
\begin{nota}
\label{nota:fj}
We abbreviate by $\fo\ge1$ the order of an element in $c_1\subseteq G$, and we let $\fj$ be the minimal number $r\ge1$ satisfying  the following property:
for any $(\fg_1,\dots,\fg_r)\in c_1^r$ there exists an element of the standard braid group $B_r$ whose image $(\fg'_1,\dots,\fg'_r)$ satisfies the following:
\begin{itemize}
\item $\fg'_1=\dots=\fg'_{\fo}$;
\item the subgroup of $G$ generated by $\fg'_{\fo+1},\dots,\fg'_r$ coincides with the subgroup generated by $\fg'_1,\dots,\fg'_r$.
\end{itemize}
\end{nota}
We remark that $\fj\le(\# c_1\cdot\fo)+1$: indeed, if $r\ge (\# c_1\cdot\fo)+1$ and if $(\fg_1,\dots,\fg_r)\in c_1^r$, then the pigeonhole principle implies that there exist $\fo+1$ equal entries in the previous sequence, say with common value $\fg\in c_1$. By the action of a suitable braid we may transform $(\fg_1,\dots,\fg_r)$ into a sequence beginning with $\fo+1$ instances of $\fg$, and this surely satisfies the requirements from Notation \ref{nota:fj}.
\begin{ex}
In the setting of Example \ref{ex:mainsetting} we have $\fo=2$ and $\fj=l+1$; this can be seen for instance as a consequence of \cite[Proposition 7.13]{Bianchi:Hur1}.
\end{ex}
\begin{lem}
\label{lem:basicreplacement}
Let $\uphi\in\CHom^c_{\uk,\unu,\omega}(\pi_1(\FmbP,*),G)$ for some $\uk,\unu,\omega$ as in Definition \ref{defn:CHur}. Assume that $\uphi$ is standard in the sense of Notation \ref{nota:uphi}. Assume $k_1\ge\fj$. For $1\le i\le n$ fix elements $\fg_i\in\nu_i$, and for $\fj+1\le i\le k$ fix elements $\fg'_i$ conjugate to $\phi(\gamma_i)$. Then there is a braid in $B_k^\cF$ sending $\uphi$ to a sequence $\uphi'$ satisfying $\phi'(\gamma_i)=\fg'_i$ for $i\ge \fj+1$, $\phi'(\alpha_i)=\phi'(\beta_i)=\one$, $\phi'(\delta_i)=\fg_i$, and moreover the elements $\phi'(\gamma_1),\dots,\phi'(\gamma_\fj)$ generate $G$.
\end{lem}
\begin{proof}
We refer to the operations described in Lemma \ref{lem:braidactionrules}.
Up to a braid in $B_{\fj}$ we may assume $\phi(\gamma_1)=\dots=\phi(\gamma_\fo)$ and that the image of $\pi_1(\cF\setminus(V_{\fj}\cup \bar P_k),*)$ along $\phi$ is the entire $G$. For any element $\fg\in c_1$ we may then choose a suitable element $\gamma\in\pi_1(\cF\setminus(\bar P_k\setminus\bar P_{\fo}),*)$ such that $\fb_{\fo}(\gamma)$ sends $\uphi$ to the sequence obtained from $\phi$ by replacing the first $\fo$ entries by $\fg$. Assuming now $\phi(\gamma_1)=\fg$, we may use operation (4) to conjugate $\phi(\delta_i)$ by $\fg$, while keeping untouched the subsequent entries of $\uphi$. Repeating the procedure several times, we may conjugate $\phi(\delta_i)$ by any word in $c_1$, that is, by any element of $G$: we may therefore achieve $\phi(\delta_i)=\fg_i$. We may repeat this process for decreasing $i$, until the equality $\phi(\delta_i)=\fg_i$ holds simultaneously for all $1\le i\le n$. We next operate a similar strategy to adjust first $\phi(\beta_i)$ using operation (2), and then $\phi(\alpha_i)$, by operation (3); we adjust these pairs for decreasing $1\le i\le h$. Finally, we operate similarly, but via operations (1), to adjust $\phi(\gamma_i)=\fg'_i$ for decreasing $\fj+1\le i\le k$. 

Suppose now that the subgroup $H$ generated by $\phi(\gamma_1),\dots,\phi(\gamma_\fj)\in c_1$ is not the entire $G$, and let $\fg\in c_1\setminus H$. By a suitable braid in $B_\fj$ we may achieve that $\phi(\gamma_1)=\dots=\phi(\gamma_\fo)$ and that $\phi(\gamma_{\fo+1}),\dots,\phi(\gamma_\fj)$ generate $H$. By a suitable operation of the first kind from Lemma \ref{lem:braidactionrules} we may now replace the first $\fo$ entries of $\uphi$ by the element $\fg$. After this replacement, the subgroup of $G$ generated by $\phi(\gamma_1),\dots,\phi(\gamma_\fj)$ is strictly larger than $H$, in that it contains $H$ and also $\fg$. Repeating this procedure a suitable number of times gives the result.
\end{proof}
\begin{cor}
\label{cor:CHurconnected}
Consider the setting of Example \ref{ex:mainsetting} and let $\uk,\unu,\omega$ be as in Definition \ref{defn:CHur}. Assume also that $k_\trsp\ge 2l-2$ and that the sum 
\[
\sum_{\lambda\subseteq c}k_\lambda\cdot N(\lambda)+\sum_{i=1}^nN(\nu_i)+N(\omega)
\]
is even. Then the space $\dCHur^{\fS_d,c}_{\uk,\unu,\omega}(\cF)$ is connected.
\end{cor}
\begin{proof}
It suffices to prove that the set $\CHom^{c}_{\uk,\unu,\omega}(\pi_1(\FmbP,*),\fS_l)$ is non-empty, and that it consists of a single orbit under the $B^\cF_k$-action. For $l=2$ this set consists of a single element, so let us assume $l\ge3$ henceforth.

To find some $\uphi\in \CHom^{\fS_l\setminus\one}_{\uk,\unu,\omega}(\pi_1(\FmbP,*),\fS_l)$, we start by choosing arbitrary elements $\phi(\delta_i)\in\nu_i$, we set for simplicity $\phi(\alpha_i)=\phi(\beta_i)=\one$, and we also fix the elements $\phi(\gamma_{2l-1}),\dots,\phi(\gamma_k)$ so that the conjugacy class of transpositions occurs $k_\trsp-2l+2$ times, and each other $\lambda\subseteq c$ occurs $k_\lambda$ times. We may then compute the product
\[
\sigma=\omega\,\phi\pa{\gamma_{2l-1}\dots\gamma_k[\alpha_1,\beta_1]\dots[\alpha_h,\beta_h]\delta_1\dots\delta_n}^{-1}\in\fS_l,
\]
which by the assumptions is an even permutation. We may therefore write $\sigma$ as a product of $2l-2$ transpositions $\phi(\gamma_1)\dots\phi(\gamma_{2l-2})$, and we may even pick the first $l-1$ to be generators of $\fS_l$.

To see that there is a single $B^\cF_k$-orbit, we let $\uphi,\uphi'\in \CHom^{\fS_l\setminus\one}_{\uk,\unu,\omega}(\pi_1(\FmbP,*),\fS_l)$;
up to applying a braids in $B_k$ to $\uphi$ and $\uphi'$, we may assume that $\phi(\gamma_i)$ is conjugate to $\phi'(\gamma_i)$ for all $i$, and moreover that $\phi(\gamma_i)$ and $\phi'(\gamma_i)$ are transpositions for $1\le i\le 2l-2$. By Lemma \ref{lem:basicreplacement}, we may apply braids in $B^\cF_k$ to $\uphi$ and $\uphi'$ and further assume that they coincide past the first $\fj=l+1$ entries; we may moreover assume that both lists $\phi(\gamma_1),\dots,\phi(\gamma_{l+1})$ and $\phi'(\gamma_1),\dots,\phi'(\gamma_{l+1})$ generate $\fS_l$. Since we also have
\[
\phi(\gamma_1\dots\gamma_{l+1})=\omega\,\phi\pa{\gamma_{l+2}\dots\gamma_k[\alpha_1,\beta_1]\dots[\alpha_h,\beta_h]\delta_1\dots\delta_n}^{-1}=\phi'(\gamma_1\dots\gamma_{l+1}),
\]
by \cite[Proposition 7.11]{Bianchi:Hur1} we may find a braid in $B_{l+1}$ sending $\uphi$ to $\uphi'$.
\end{proof}
\begin{lem}
\label{lem:quotientmodule}
Consider on the set $\Hom(\pi_1(\cF),G)$ the equivalence relation generated by $\uphi\sim\uphi'$ if there exist elements $\fg,\fg'\in c_1$ such that the concatenations $(\fg,\uphi),(\fg',\uphi')$, viewed as elements of $\Hom^c(\pi_1(\FmbP_1,*),G)$ in the light of Notation \ref{nota:uphi}, belong to the same $B_1^\cF$-orbit. Then the equivalence classes of $\phi$ are in natural bijection with the set of $n$-tuples $\unu=(\nu_1,\dots,\nu_n)$ of conjugacy classes in $G$, the bijection sending $\uphi$ to the sequence of conjugacy classes of $\phi(\delta_i)$ for $1\le i\le n$.
\end{lem}
\begin{proof}
It is evident that the conjugacy classes of the  $\delta_i$-entries are invariant under the equivalence relation $\sim$. Viceversa, if $\uphi$ and $\uphi'$ satisfy that $\phi(\delta_i)$ is conjugate to $\phi'(\delta_i)$ for all $1\le i\le n$, then we may use a strategy as in the proof of Lemma \ref{lem:basicreplacement}: we may first conjugate $\phi(\delta_n)$ by a sequence of elements of $c_1$, via operations of type (4), until we achieve $\phi(\delta_n)=\phi'(\delta_n)$. We then proceed in a similar way to gradually change $\phi(\delta_{n-1})$ into $\phi'(\delta_{n-1})$, and so on, until $\phi(\delta_i)=\phi'(\delta_i)$ for all $1\le i\le n$. We next adjust the elements $\phi(\beta_i),\phi(\alpha_i)$ for decreasing $1\le i\le h$. The procedure is in fact simpler than the one in Lemma \ref{lem:basicreplacement}, as we may each time directly pick any element $\fg\in c_1$ to concatenate with $\uphi$, instead of producing it via an operation of the first kind described in Lemma \ref{lem:braidactionrules}.
\end{proof}
We are now ready to prove Theorem \ref{thm:Thetaisomorphism}. The ingredients of the proof will justify together Notation \ref{nota:goodconstantsIJ}, and for this reason we have postponed to this point the discussion about the actual values of the constants $I_l,J_l$.
\begin{proof}[Proof of Theorem \ref{thm:Thetaisomorphism}]
By Lemmas \ref{lem:first_reduction}, \ref{lem:second_reduction} and \ref{lem:third_reduction} and Example \ref{ex:tbrvexplained},  it suffices to show, for all $1\le k\le 2\upsilon$, that the map $\tbrv\colon\dCHur^{\fS_l,\fS_l\setminus\one}_{\uk,\uone,\one}(\cF)\to\Conf_{\uk}(\cF)$ is a rational homology isomorphism in degrees $*<(2k-2\upsilon-J_l)/I_l$, for any $\uk$ in the set $\Xi_k\cong\pi_0(\Conf_k(\cF,\fP(l)_+)_{2\upsilon})$ form Notation \ref{nota:Xi}. 
If $2k-2\upsilon<J_l$ there is nothing to prove, so let us assume $2k-2\upsilon\ge J_l$ from now on. 
We next observe that if $\uk\in\Xi_k$, that is, if $\sum_{\lambda\in\fP(l)_+}k_\lambda=k$ and $\sum_{\lambda\in\fP(l)_+}kN(\lambda)=2\upsilon$, then we must have $k_1=k_\trsp\ge 2k-2\upsilon$. The constraint $J_l\ge 2l-2$ in Notation \ref{nota:goodconstantsIJ} will then ensure, together with Corollary \ref{cor:CHurconnected}, that $\dCHur^{\fS_l,\fS_l\setminus\one}_{\uk,\uone,\one}(\cF)$ is connected.

Now let $c\subseteq \fS_l\setminus\set{\one}$ be the union of conjugacy classes $\lambda\in\fP(l)_+$ with $k_\lambda\ge1$, and let $c_1\subseteq c$ denote the conjugacy class of transpositions. We may identify $\dCHur^{\fS_l,c}_{\uk,\uone,\one}(\cF)$ also as a component of $\Hur^{\fS_l,c}_k(\cF)$, which for simplicity we denote
$\dCHur^{\fS_l,c}_{\uk,\uone,\one}(\cF)$, even though we should first remove from $\uk$ all occurrences of 0. We may therefore study the map
\begin{equation}
\label{eq:ultimatetbrv}
\tbrv\colon\dCHur^{\fS_l,c}_{\uk,\uone,\one}(\cF)\to\Conf_{\uk}(\cF).
\end{equation}

We next recall from \cite[§2]{LL2} the notion of Hurwitz module. In particular, by \cite[Example 2.1.3]{LL2}, the sequence of sets $\Hom^c(\pi_1(\FmbP_m,*),\fS_l)$ for $m\ge0$ forms a ``bijective'' Hurwitz module over the rack $c$. We may consider the bijective Hurwitz submodule $S$ spanned by the subsets $\Hom^c_{\uone}(\pi_1(\FmbP_m,*),\fS_l)$, comprising group homomorphisms sending each $\delta_i\mapsto\one$.
We may then identify $\dCHur^{\fS_l,c}_{\uk,\uone,\one}(\cF)$ with a component of $\CHur^{c,S}_k$ in the sense of \cite{LL2}\footnote{We have been notified by the authors that the notation $\mathrm{MHur}^{c,S}_k$ shall be used in place of $\CHur^{c,S}$ in a forthcoming version of the article.}.

We next aim at computing the quotient module $(c/c_1,S/c_1)$ in the sense of \cite[§2.3]{LL2}. The quotient rack $c/c_1$ is the abelian rack whose elements are the conjugacy classes of $\fS_l$ contained in $c$; and by Lemma \ref{lem:quotientmodule} we have that $S/c_1$ has a point as $0$\textsuperscript{th} set , and therefore $(c/c_1)^m$  as $m$\textsuperscript{th} set. This implies that $\Hur_k^{c/c_1,S/c_1}$, in the notation of \cite{LL2}, is just $ \Conf_k(\mcF,c/c_1)$, and in particular we may identify $\Conf_{\uk}(\mcF)$ with a connected component of $\dCHur^{c/c_1,S/c_1}$.

We may now apply \cite[Theorem 1.4.9]{LL2} and conclude that, for large enough constants $I_l,J_l$, the map \eqref{eq:ultimatetbrv} induces a rational homology equivalence in degrees $*<(k_\trsp-J_l)/I_l$, and in particular in degrees $*<(2k-2\upsilon-J_l)/I_l$. Concretely, we identify the source of \eqref{eq:ultimatetbrv} with a component in $\CHur^{c,S}$, as explained above, and identify the target of \eqref{eq:ultimatetbrv} with a component in $\CHur^{c/c_1,S/c_1}$.
\end{proof}
\begin{nota}
\label{nota:goodconstantsIJ}
Relying on \cite[Theorem 1.4.9]{LL2},
we fix two sequences of positive constants $I_l,J_l$, for $l\ge2$, satisfying the following properties:
\begin{itemize}
\item $I_l\ge1$ (see Lemma \ref{lem:first_reduction});
\item $J_l\ge 2l-2$ (see Corollary \ref{cor:CHurconnected});
\item for all $1\le k\le 2\upsilon$ and all $\uk\in\Xi_k$ as in Notation \ref{nota:Xi}
the map
\[
\tbrv\colon\dCHur^{\fS_l,\fS_l\setminus\one}_{\uk,\uone,\one}(\cF)\to\Conf_{\uk}(\mcF)
\]
induces a rational homology isomorphism in degrees $*<(k_\trsp-J_l)/I_l$ (see the proof of Theorem \ref{thm:Thetaisomorphism});
\item 
both sequences $(J_l)_{l\ge2}$ and $(I_l)_{l\ge2}$ are weakly increasing with respect to divisibility: if $l,l'\ge2$ and $l\mid l'$ , then $J_l\le J_{l'}$ and $I_l\le I_{l'}$ (this is only to phrase Theorems \ref{thm:A} and \ref{thm:B} in terms of $I_d,J_d$ instead of $\max_{l\mid d}I_l,\max_{l\mid d}J_l$).
\end{itemize}
\end{nota}

\section{Proofs of the main theorems}
We fix $g\ge0$ and $d\ge2$ in the entire section. 
In this section we prove Theorems \ref{thm:A}, \ref{thm:B}, \ref{thm:D} and \ref{thm:E}; we conclude with Subsection \ref{subsec:oddexamples}, in which in particular we make some explicit computations for $d\le 5$.
\subsection{Proof of Theorem \ref{thm:B}}
We start by generalising Notation \ref{nota:ccH} in the light of Example \ref{ex:cHglhe} and Definition \ref{defn:dCBra}.
\begin{defn}
For $g\ge1$, $l\ge2$ and and $h\ge0$ we denote by $\ccH_{g,l}(\fp_h)\subseteq\cH_{g,l}(\fp_h)$ the open substack parametrising degree $l$ maps $\theta\colon C\to D$ admitting at least one simple branch value; and we denote by $\dcH_{g,l}(\fp_h)\subseteq\ccH_{g,l}(\fp_h)$ the open substack parametrising degree $l$ maps $\theta\colon C\to D$ that furthermore do not factor through any intermediate curve.

We similarly define, for $e\ge1$, the substacks $\dcH_{g,l}(\fp_{h,e})\subseteq\ccH_{g,l}(\fp_{h,e})\subseteq\cH_{g,l}(\fp_{h,e})$ as pullbacks of the previous along the map $\epsilon_{h,e}$. We denote by $\dq_{g,l,h,e}\colon\dcH_{g,l}(\fp_{h,e})\to\cH_{h,e}$ the restriction of the map $q_{g,l,h,e}$ from Notation \ref{nota:qglhe}.
\end{defn}
\begin{lem}
\label{lem:ccHcodimension}
For all $g\ge1$, $l\ge2$, $h\ge0$ and $e\ge1$ the topological stack $\cH_{g,l}(\fp_{h,e})\tp$ is an oriented orbifold of real dimension $4\upsilon(h,e,0)+4\upsilon(g,l,h)-6$. Moreover, $\ccH_{g,l}(\fp_{h,e})\tp$ is the complement of a closed topological stubstack of codimension at least $2\upsilon(g,l,h)$ in $\cH_{g,l}(\fp_{h,e})$. 
In particular, restriction induces an isomorphism $H^*(|\cH_{g,d}(\fp_{h,e})\tp|;\bQ)\cong H^*(|\ccH_{g,d}(\fp_{h,e})\tp|;\bQ)$ in the range of cohomological degrees $*<2\upsilon(g,l,h)-1$.
\end{lem}
\begin{proof}
Consider the map of topological stacks $\CBra^{h,e}_{\Ptp}\to\cH_{h,e}\tp$ given by the inclusion of a fibre of the map $q_{h,e,0}\tp\colon\cH_{h,e}\tp\to\fM_0\tp$. We may pullback the smooth map $q_{g,l,h,e}\tp\colon\cH_{g,l}(\fp_{h,e})\tp\to\cH_{h,e}\tp$ from Notation \ref{nota:qglhe} along this inclusion 
to obtain a smooth map
\begin{equation}
\label{eq:smoothpullback}    
\CBra^{h,e}_{\Ptp}\times_{\cH_{h,e}\tp}\cH_{g,l}(\fp_{h,e})\tp\to \CBra^{h,e}_{\Ptp}
\end{equation}
whose fibre over $\theta\colon D\to \Ptp$ is $\CBra^{g,l}_D$. By Proposition \ref{prop:orbifold} the base and fibre of the smooth map \eqref{eq:smoothpullback} are oriented orbifolds of dimensions $4\upsilon(h,e,0)$ and $4\upsilon(g,l,h)$, hence the fibre product in \eqref{eq:smoothpullback} is an orbifold of dimension $4\upsilon(h,e,0)+4\upsilon(g,l,h)$.
The connected Lie group $\Aut(\bP^1)\tp$ of real dimension 6 acts on the fibre product \eqref{eq:smoothpullback} with finite stabilisers, thanks to the assumptions on $g,l,h,e$ (compare with Remark \ref{rem:cHgdDMstack}). It follows that the quotient topological stack
\[
\cH_{g,l}(\fp_{h,e})\tp\cong \pa{\CBra^{h,e}_{\Ptp}\times_{\cH_{h,e}\tp}\cH_{g,l}(\fp_{h,e})\tp}
\ /\ \Aut(\bP^1)\tp
\]
is an oriented orbifold of real dimension $4\upsilon(h,e,0)+4\upsilon(g,l,h)-6$.

For the second statement, it suffices to estimate codimensions on fibres of $q_{g,l,h,e}\tp$. The complement $Z$ of the intersection of $\ccH_{g,l}(\fp_{h,e})\tp$ with a fibre $\CBra^{g,l}_D$ of $q_{g,l,h,e}\tp$ is the closed substack of 
$\CBra^{g,l}_D$ parametrising degree $l$ maps $\theta\colon C\to D$ admitting no simple branch value: if this is the case, the norm of the local monodromy of $\theta$ around each branch value of $\theta$ is at least 2, so that $\theta$ admits at most $\upsilon(g,l,h)$ branch values. It follows that 
$Z$ has real dimension at most $2\upsilon(g,l,h)$, and hence codimension at least $2\upsilon(g,l,h)$ inside $\CBra^{g,l}_D$.

\end{proof}
\begin{ex}
\label{ex:ccHcH}
By Lemma \ref{lem:ccHcodimension}, the inclusion $|\ccH_{g,d}\tp|\hookrightarrow|\cH_{g,d}\tp|$ induces a rational cohomology isomorphism in degrees less than $2\upsilon(g,d,0)-1=2(g+d-1)-1$. As this number is larger than the range of cohomological degrees from Theorem \ref{thm:A}, we may henceforth focus on computing the stable rational cohomology of $|\ccH_{g,d}\tp|$.
\end{ex}
Recall again Notation \ref{nota:ccH}. In the light of Lemma \ref{lem:minimalfactorisation}, we may characterise $\fF_{h,e}\ccH_{g,d}\tp$ as the locus of maps $\theta\colon C\to L$ admitting at least one simple branch value, and such that the canonical factorisation $\bar\theta\colon D\to L$ as in Lemma \ref{lem:minimalfactorisation} satisfies that $\bar\theta$ has degree $e$ and $D$ has genus $h$. We may then consider the map
\[
\iota_{g,d,h,e}\colon\fF_{h,e}\ccH_{g,d}\tp\hookrightarrow\dcH_{g,d/e}(\fp_{h,e})\tp,
\]
sending $(C\xrightarrow{\theta}L)$ to $(C\to D\xrightarrow{\bar\theta}L)$ as above. It is an open inclusion between orbifolds of the same dimension.
\begin{lem}
\label{lem:iotagdhecodimension}
The complement of the image of
$\iota_{g,d,h,e}$ has codimension at least $2\upsilon(g,d/e,h)$ inside $\dcH_{g,d/e}(\fp_{h,e})\tp$. In particular $|\iota_{g,d,h,e}|$ induces a rational cohomology isomorphism in degrees $*<2\upsilon(g,d/e,h)-1$.
\end{lem}
\begin{proof}
We pick $(D\to L)\in\cH_{h,e}\tp(*)$ and consider the 
fibre $\dCBra^{g,d/e}_D$ of $\dq_{g,d/e,h,e}\tp$ at $(D\to L)$. It suffices to show that the complement $Z$ of the intersection of the image of $\iota_{g,d,h,e}$ with $\dCBra^{g,d/e}_D$ has codimension at least $2\upsilon(g,d/e,h)$ inside $\dCBra^{g,d/e}_D$. If $P\subseteq D$ is the finite set of all preimages of branch values of $D\to L$ (including those that are not branch points), then $Z$ is the locus of $\CBra^{g,d/e}_D$ comprising maps $\theta\colon C\to D$ such that all branch values of $\theta$ belong to $P$ or are not simple branch values. The number of branch values outside $P$ can therefore be at most $\upsilon(g,d/e,h)$. It follows that $Z$ has dimension at most $2\upsilon(g,d/e,h)$ and hence codimension at least $2\upsilon(g,d/e,h)$ inside $\dCBra^{g,d/e}_D$. 
\end{proof}
\begin{nota}
Recall Equations \ref{eq:rvcrs} and \ref{eq:cHvcrs}. We denote by
\[
\dr_{g,l,h,e}\colon|\dcH_{g,l}(\fp_{h,e})\tp_\vcoarse|\to|\cH_{h,e}\tp|
\]
the pullback along $|\epsilon_{h,e}\tp|$ of the projection
\[
|\dCBra^{g,l}_\Sigma|\sslash\Homeo^+(\Sigma)\to\bbB\Homeo^+(\Sigma)\simeq|\fM_h\tp|,
\]
where $\Sigma$ is a connected closed oriented surface of genus $h$.
\end{nota}
We observe that there is a commutative square of homotopy types over $|\cH_{h,e}\tp|$ as follows, with vertical maps being rational homology equivalences
\begin{equation}
\label{eq:diagramdrQequivalence}
\begin{tikzcd}[row sep=10]
{|\dcH_{g,l}(\fp_{h,e})\tp|}\ar[r]\ar[d,"\simeq_\bQ"]&{|\cH_{g,l}(\fp_{h,e})\tp|}\ar[d,"\simeq_\bQ"]\\
{|\dcH_{g,l}(\fp_{h,e})\tp_\vcoarse|}\ar[r]&{|\cH_{g,l}(\fp_{h,e})\tp_\vcoarse|}
\end{tikzcd}
\end{equation}

We are now ready to prove Theorem \ref{thm:B}.
\begin{proof}[Proof of Theorem \ref{thm:B}]
Recall Notation \ref{nota:intfp}, and abbreviate $\upsilon=\upsilon(g,d/e,h)$.
By Theorem \ref{thm:Thetaisomorphism} and Corollary \ref{cor:defnTheta} the following composite map of parametrised rational spectra over $\cH_{h,e}$ induces a fibrewise isomorphism on homotopy groups $\pi_*$ in degrees $*>-(2\upsilon-J_{d/e})/I_{d/e}$:
\[
\pa{\int_{|\fp_{h,e}\tp|}\cA(d/e)}_{2\upsilon}\xrightarrow{\Theta^{g,d/e}_{h,e}[-4\upsilon]}(r_{g,d/e,h,e})*\bQ\to (\dr_{g,d/e,h,e})_*\bQ.
\]
We may then apply $|\cH_{h,e}\tp|_*$ to the previous and obtain a map of rational spectra 
\[
|\cH_{h,e}\tp|_*\pa{\int_{|\fp_{h,e}\tp|}\cA(d/e)}_{2\upsilon}\to|\dcH_{g,l}(\fp_{h,e})\tp_\vcoarse|_*\bQ
\]
which induces an isomorphism on homotopy groups in the same range. The last map may now be composed
with the maps
\[
|\dcH_{g,l}(\fp_{h,e})\tp_\vcoarse|_*\bQ\xrightarrow{\simeq}|\dcH_{g,l}(\fp_{h,e})\tp|_*\bQ\xrightarrow{|\iota_{g,d,h,e}|}|\fF_{h,e}\ccH_{g,d}\tp|_*\bQ
\]
where the first is an equivalence as observed in 
\eqref{eq:diagramdrQequivalence}, whereas the second is an equivalence in the large range of homotopy groups $*>2\upsilon-1$ by Lemma \ref{lem:iotagdhecodimension}.
\end{proof}

\subsection{Proof of Theorem \ref{thm:A}} 
Recall from Lemma \ref{lem:ccHcodimension} and Example \ref{ex:ccHcH} that the coarse space $\ccH_{g,d,\coarse}\tp$, which is open in $\cH_{g,d,\coarse}\tp$, is an oriented rational homology manifold of dimension $4g+4d-10$. Instead of $H^*(|\ccH_{g,d}\tp|;\bQ)\cong H^*(|\ccH_{g,d,\coarse}\tp|;\bQ)$ we may therefore focus on computing the relative homology $H_*(|\ccH_{g,d,\coarse}\tpp|,\infty;\bQ)$.
\begin{nota}
We denote by $\bar\fF_{h,e}\ccH_{g,d,\coarse}\tp$ the closure of the stratum $\fF_{h,e}\ccH_{g,d,\coarse}\tp$ inside
$\ccH_{g,d,\coarse}\tp$. Similarly, we denote by 
$\bar\fF_{h,e}\cH_{g,d,\coarse}\tp$ the closure of the subspace $\fF_{h,e}\ccH_{g,d,\coarse}\tp$ inside
$\cH_{g,d,\coarse}\tp$. It is the locus of those equivalence classes of degree $d$ maps $\theta\colon C\to L$ that admit some factorisation through a degree $e$ map $D\to L$.
\end{nota}
Again by Lemma \ref{lem:ccHcodimension},
each stratum $\fF_{h,e}\ccH_{g,d,\coarse}\tp$, which is identified along $\iota_{g,d,h,e,\coarse}$ with an open in $\cH_{g,d/e}(\fp_{h,e})\tp_\coarse$, is a rational homology manifold of dimension $4\upsilon(h,e,0)+4\upsilon(g,d/e,h)-6$. The codimension of $\fF_{h,e}\ccH_{g,d,\coarse}\tp$ inside $\ccH_{g,d,\coarse}\tp$ is therefore the number $4\rho(h,e,d)$ introduced in Notation \ref{nota:rho}. 

\begin{lem}
\label{lem:surjectivecollapse}
The collapse map
\[
H_*(|(\bar\fF_{h,e}\ccH_{g,d,\coarse}\tp)^\infty|,\infty;\bQ)\to H_*(|(\fF_{h,e}\ccH_{g,d,\coarse}\tp)^\infty|,\infty;\bQ)
\]
is surjective in the range of homological degrees
\[
*>\dim(\fF_{h,e}\ccH_{g,d}\tp)-\ (2\upsilon(g,d/e,h)-J_{d/e})/I_{d/e}
\]
and in particular in the range $*>\dim(\ccH_{g,d}\tp)-(2(g+d-1)-J_d)/I_d$.
\end{lem}
\begin{proof}
Abbreviate by $\delta$ the (real) dimension of $\fF_{h,e}\ccH_{g,d}$.
By Lemma \ref{lem:iotagdhecodimension} and by the proof of Theorem \ref{thm:B}, the composition of the two top horizonal maps and the right vertical map
in the following diagram is surjective in homology relative to $\infty$ in the first given range of degrees, since its precomposition with the map $|\cH_{h,e}\tp|_*\Theta^{g,d/e}_{h,e}[\delta]$ from Notation \ref{nota:intfp} is an isomorphism:
\begin{equation}
\label{eq:compositecollapse}
\begin{tikzcd}
{|\cH_{g,d/e}(\fp_{h,e})\tpp_\vcoarse|}\ar[r,"\simeq_\bQ"] & {|\cH_{g,d/e}(\fp_{h,e})\tpp_\coarse|}\ar[r]\ar[dl,dashed]&{ |\dcH_{g,d/e}(\fp_{h,e})\tpp_\coarse|}\ar[d,"|\iota_{g,d,h,e,\coarse}|"]\\
{|(\bar\fF_{h,e}\cH_{g,d,\coarse}\tp)^\infty|}\ar[r]&{|(\bar\fF_{h,e}\ccH_{g,d,\coarse}\tp)^\infty|}\ar[r]&{|(\fF_{h,e}\ccH_{g,d,\coarse}\tp)^\infty|}.
\end{tikzcd}
\end{equation}
All horizonal maps in the pentagon are collapse maps corresponding to open inclusions.
The dashed diagonal map is constructed as follows: we have a proper forgetful map $\cH_{g,d/e}(\fp_{h,e})\tp_\coarse\to\cH_{g,d,\coarse}\tp$ sending the class of $(C\to D\to L)$ to the class of $(C\to L)$. Its image is contained in (in fact, equal to) $\bar\fF_{h,e}\cH_{g,d}$, and we consider the associated map on one point compactifications. The pentagon commutes by inspection; it follows that the bottom right horizontal collapse map is also surjective in the first given range of homological degrees.

For the final statement, we observe that the first given range of homological degrees $\dim(\fF_{h,e}\ccH_{g,d}\tp)-\ (2\upsilon(g,d/e,h)-J_{d/e})/I_{d/e}$ may be rewritten as the expression 
$\dim(\ccH_{g,d}\tp)-4\rho(h,e,d)-(2\upsilon(g,d/e,h)-J_{d/e})/I_{d/e}$,
and we have inequalities
\[
\begin{split}
&4\rho(h,e,d)+(2\upsilon(g,d/e,h)-J_{d/e})/I_{d/e}\\
\ge&4(h+e-1)(d/e-1)+(2(g-1)-2(h-1)d/e-J_d)/I_d\\
\ge&4(e-1)(d/e-1)+(2(g-1)+2d/e-J_d)/I_d\\
\ge&(2(g+d-1)-J_d)/I_d,
\end{split}
\]
where we used the fourth property in Notation \ref{nota:goodconstantsIJ}, the fact that the second expression is mimimized for $h=0$, and the fact that the third expression is minimized, among values of $e$ between $1$ and $d/2$, for $e=1$.
\end{proof}
\begin{proof}[Proof of Theorem \ref{thm:A}]
We filter the pointed space $\ccH_{g,d,\coarse}\tpp$  by suitable unions of its closed pointed subspaces $(\bar\fF_{h,e}\ccH_{g,d,\coarse}\tp)^\infty$: for $k\ge1$ we denote by $S_k$ the set of pairs $(h,e)$ where $h\ge0$ is arbitrary, and $e\mid d$ is a divisor such that $d/e$ factors as a product of exactly $k$ prime numbers. We then let
$F_k\ccH_{g,d,\coarse}\tpp$ be the union of all subspaces $(\bar\fF_{h,e}\ccH_{g,d,\coarse}\tp)^\infty$ corresponding to pairs $(h,e)\in S_{k'}$ for some $k'\le k$. We also set $F_0\ccH_{g,d,\coarse}\tpp=\set{\infty}$.
We thus obtain an increasing filtration of $\ccH_{g,d,\coarse}\tpp$ by closed pointed subspaces. We may then run a spectral sequence to compute the homology of $\ccH_{g,d,\coarse}\tpp$. By Lemma \ref{lem:surjectivecollapse}, the composite map
\[
\begin{tikzcd}
\bigoplus_{(h,e)\in S_k}H_*(|(\bar\fF_{h,e}\ccH_{g,d\coarse}\tp)^\infty|,\infty;\bQ)\ar[r]& H_*(|F_k\ccH_{g,d,\coarse}\tpp|,\infty;\bQ)\ar[dl]\\
H_*(|F_k\ccH_{g,d,\coarse}\tpp|,|F_{k-1}\ccH_{g,d,\coarse}\tpp|;\bQ)\ar[r,"\cong"]& \bigoplus_{(h,e)\in S_k}H_*(|(\fF_{h,e}\ccH_{g,d\coarse}\tp)^\infty|,\infty;\bQ)
\end{tikzcd}
\]
is surjective in homological degrees $*> \dim(\ccH_{g,d}\tp)-(2(g+d-1)-J_d)/I_d$; in particular the diagonal map is also surjective in the same range.  
It follows that the spectral sequence collapses in this range of degrees, in which we obtain isomorphisms
\[
H_*(|\ccH_{g,d,\coarse}\tpp|,\infty;\bQ)\cong \bigoplus_{\overset{h\ge0}{1\le e<d,\ e\mid d}}H_*(|(\fF_{h,e}\ccH_{g,d\coarse}\tp)^\infty|,\infty;\bQ).
\]
The statement of Theorem \ref{thm:A} now follows
by Poincar\'e--Lefschetz duality.
\end{proof}
\subsection{Proof of Theorem \ref{thm:D}}
\label{subsec:proofthmD}
We fix $l\ge2$ and a closed oriented surface $\Sigma$ of genus $h\ge0$ throughout the subsection.
We start the subsection by introducing the class $\xi_{l,h}$.
Recall Definition \ref{defn:int}, and let $I\colon\FS\to\Sp_\bQ^\bN$ be a non-unital symmetric monoidal functor with underlying object $I(*)\in\Sp_\bQ^\bN$ satisfying $I(*)_0\simeq0$. Then for every finite set $S$ of cardinality at least 2 we have that the part of $I(S)$ of grading 1 vanishes; it follows that for a connected homotopy type $Y$ we have
\[
\pa{\int_YI}_1\simeq(\colim_{\FS^{\le1}_{/Y}}I)_1\simeq \colim_YI(*)_1,
\]
according to the following notation.
\begin{nota}
\label{nota:FSlek}
For $k\ge0$ we denote by $\FS^{\le k}\subseteq\FS$ the (possibly empty) full subcategory spanned by finite sets of cardinality at most $k$. 
\end{nota}
For instance, when $Y=|\Sigma|$ and $I=\cAh(l)_+$, from the equivalence $(\cAh(l)_+)_1\simeq \bQ$ we obtain the equivalence $\pa{\int_{|\Sigma|}\cAh(l)}_1\simeq|\Sigma|_!\bQ$; if we shift this equivalence we then also obtain $\pa{\int_{|\Sigma|}\cA(l)}_1\simeq|\Sigma|_*\bQ$. These equivalences are $\Homeo^+(\Sigma)$-equivariant, so they can be regarded as equivalences of parametrised rational spectra over $|\fM_h\tp|\simeq\bbB\Homeo^+(\Sigma)$. For instance, the last equivalence takes the form $\pa{\int_{|\fp_h\tp|}\cA(l)}_1\simeq|\fp_h\tp|_*\bQ$, using Notation \ref{nota:universal_curve}.

\begin{defn}
\label{defn:xiclass}
Recall Notation \ref{nota:universal_curve}. We denote by 
\[
\xi_{l,h}\in H^0\pa{|\fM_h\tp|;\int_{|\fp_h\tp|}\cA(l)}_1
\]
the class corresponding to the composite map of rational spectra
\[
\bQ\xrightarrow{1}|\cC_h\tp|_*\bQ\simeq|\fM_h\tp|_*|\fp_h\tp|_*\bQ\simeq |\fM_h\tp|_*\pa{\int_{|\fp\tp|}\cA(l)}_1,
\]
where the first map labeled ``1'' is the unit.
\end{defn}
Loosely speaking, $\xi_{l,h}$ picks, in a $\Homeo^+(\Sigma)$-equivariant way, the fundamental class of $\Sigma$, seen as a point in $\pi_0(\int_{|\Sigma|}\cA(l))_1\cong\pi_2(\int_{|\Sigma|}\cAh(l))_1$.

\begin{nota}
\label{nota:intfiltration}
Recall Notation \ref{nota:FSlek}. For a connected homotopy type $Y$, a presentably symmetric monoidal $\infty$-category $\cC$ and a non-unital symmetric monoidal functor $I\colon \FS\to\cC$, we filter $\int_YI$ by the graded spectra
\[
F_k\int_YI\coloneqq \colim_{\FS^{\le k}_{/Y}}I\in\cC,\quad k\ge0.
\]
We obtain a refinement of $\int_YI$ to a commutative algebra in $\bN_{\le}$-filtered objects in $\cC$, i.e. in the functor category $\Fun(\bN_{\le},\cC)$ with Day convolution.
\end{nota}
In the setting $Y=|\Sigma|$ and $I=\cAh(l)_+$, so $\cC=\Sp_\bQ^\bN$, we have the equivalence 
\[
\pa{F_1\int_{|\Sigma|}\cAh(l)}_1\simeq\pa{\int_{|\Sigma|}\cAh(l)}_1\simeq|\Sigma|_!\bQ.
\]
The multiplicativity of the filtration from Notation \ref{nota:intfiltration} implies in particular that for all $n\ge0$ the map $\xi_{l,h}[2]\cdot-\colon\pa{\int_{|\Sigma|}\cAh(l)}_n[2]\to\pa{\int_{|\Sigma|}\cAh(l)}_{n+1}$ refines to a filtered map, up to shifting the filtration in the target. More precisely, we have maps as follows, naturally in $k\in\bN_{\le}$:
\begin{equation}
\label{eq:xilhfiltered}
\xi_{l,h}[2]\cdot-\colon\pa{F_k\int_{|\Sigma|}\cAh(l)}_n[2]\to\pa{F_{k+1}\int_{|\Sigma|}\cAh(l)}_{n+1}.
\end{equation}
\begin{defn}
\label{defn:cAfiltration}
Recall Definition \ref{defn:cAheart}. We consider the two-step filtration on $\cAh(l)$ given by $F_0\cAh(l)=\bQ\cdot\sca{\one}$ and $F_k\cAh(l)=\cAh(l)$ for $k\ge1$. It makes $\cAh(l)$ into a commutative algebra in $\Fun(\bN_{\le},\Sp_\bQ^\bN)$. We denote by 
\[
\cAh(l)^\delta\coloneqq\bigoplus_{k\ge0} (F_k/F_{k-1})\cAh(l).
\]
the associated graded algebra: it is a commutative algebra in $\Sp^{\bN\times\bN}_\bQ$.
We also consider $\cAh(l)_+^\delta$ as a non-unital commutative algebra in $\Sp_\bQ^{\bN\times\bN}$.
\end{defn}
Intuitively, in $\cAh(l)^\delta$ we put $\sca{\one}$ in bigrading $(0,0)$ and $\sca{\lambda}$, for $\lambda\in\fP(l)_+$, in bigrading $(1,N(\lambda))$. Note that for all $\lambda,\lambda'\in\fP(l)_+$ we have $\sca{\lambda}\cdot\sca{\lambda'}=0$ in $\cAh(l)^+$. 

The filtration $F_\star\cAh(l)_+$ from Definition \ref{defn:cAfiltration} gives rise to an analogous filtration
$F_\star\int_{|\Sigma|}\cAh(l)_+$ on factorisation homology, i.e. $\int_{|\Sigma|}\cAh(l)_+$ upgrades to a non-unital commutative algebra in $\Fun(\bN_{\le},\Sp_\bQ^\bN)$ too. This can be identified with the filtration from Notation \ref{nota:intfiltration} in the setting $Y=|\Sigma|$ and $I=\cAh(l)_+$. 
Passing to associated graded algebras, we obtain an equivalence of non-unital, $\Homeo^+(\Sigma)$-equivariant commutative algebras in $\Sp_\bQ^{\bN\times\bN}$ as follows
\[
\bigoplus_{k\ge0}(F_k/F_{k-1})\ \int_{|\Sigma|}\cAh(l)_+\simeq \int_{|\Sigma|}\cAh(l)^\delta_+.
\]
Note that the right side only depends on the algebra $\cAh(l)^\delta_+$, which is endowed with the zero product. We may in fact interpret $\int_{|\Sigma|}\cAh(l)^\delta_+$ in terms of configuration spaces as follows, using that the underlying spectrum of $\cAh(l)^\delta_+$ is freely generated by the bigraded set $\fP(l)_+$, in which $\lambda$ has bigrading $(1,N(\lambda))$.
\begin{defn}
\label{defn:Confmonoid}
Recall Definition \ref{defn:colconf}. For a finite set $S$ endowed with a grading $N\colon S\to\bN$ and for a topological space $X$ we consider the disjoint union
\[
\Conf(X,S)\coloneqq\coprod_{k\ge0}\Conf_k(X,S)
\]
as a bigraded topological space: a pair $(P,f)$ of a finite set $P$ with a map $f\colon P\to S$ is in bigrading $(\#P,\sum_{z\in P}N(f(z)))$.

When $X$ is a compact topological space, by taking one point compactifications in each bigrading, we obtain a bigraded pointed topological space
\[
\Conf(X,S)^\infty\coloneqq\bigvee_{k\ge0}\Conf_k(X,S)^\infty.
\]
\end{defn}
There is a structure of
strictly commutative monoid in bigraded pointed spaces on $\Conf(X,S)$ given by superposition of configurations: the superposition of $(P,f)$ and $(P',f')$ is either $(P\sqcup P',f\sqcup f')$ if $P\cap P'=\emptyset$, or it is $\infty$ (in the correct bigrading).
We may now identify $\int_{|\Sigma|}\cAh(l)_+^\delta$, as a commutative algebra in $\Sp_\bQ^{\bN\times\bN}$, with the image of the commutative monoid $\Conf(\Sigma,\fP(l)_+)$ along the symmetric monoidal functor $(|-|,\infty)_!\bQ$ from bigraded pointed spaces to bigtraded rational spectra. See also Example \ref{ex:intConfidentification} for an application of a more general principle to derive this identification.
\begin{proof}[Proof of Theorem \ref{thm:D}]
Consider the element $\xi_{l,h}[2]\in\pi_2(\int_{|\Sigma|}\cAh(l))_1\cong \pi_2(|\Sigma|_!\bQ)$ corresponding to $\xi_{l,h}$. In few words, $\xi_{l,h}[2]$ is the fundamental class of $\Sigma$. It suffices to prove that the map
\begin{equation}
\label{eq:firstTheoremD}  
\xi_{l,h}[2]\cdot -\colon \pa{\int_{|\Sigma|}\cAh(l)}_n[2]\to \pa{\int_{|\Sigma|}\cAh(l)}_{n+1}
\end{equation}
induces an isomorphism on $\pi_*$ in degrees $*>-n+2(n+1)=n+2$. For $n=0$ we just observe that the map $\bQ[2]\to|\Sigma|_!\bQ$ giving the fundamental class induces an isomorphism between vanishing homotopy groups in degrees $*>2$.

From now on we suppose $n\ge1$: this allows us to replace $\cAh(l)$ by $\cAh(l)_+$ in \eqref{eq:firstTheoremD}; moreover, since \eqref{eq:firstTheoremD} is a filtered map, as observed in \eqref{eq:xilhfiltered}, it suffices to prove for all $k\ge0$ that the associated graded map
\begin{equation}
\label{eq:secondTheoremD}  
\xi_{l,h}[2]\cdot -\colon (F_k/F_{k-1})\pa{\int_{|\Sigma|}\cAh(l)}_n[2]\to (F_{k+1}/F_k)\pa{\int_{|\Sigma|}\cAh(l)}_{n+1}
\end{equation}
induces an isomorphism on $\pi_*$ for $*>n+2$. By the previous discussion, it is equivalent to ask whether for all $k\ge0$ the map
\[
H_{i-2}(\Conf_k(\Sigma,\fP(l)_+)_n^\infty,\infty;\bQ)\to H_{i}(\Conf_{k+1}(\Sigma,\fP(l)_+)_{n+1}^\infty,\infty;\bQ)
\]
given by superposition with the fundamental class of $\Conf_1(\Sigma,\set{\trsp})\cong \Sigma$ induces an isomorphism for $i>n+2$. If $2k+1\le n$, this holds true because the above homology groups vanish for dimension reasons. Otherwise, if $\uk=(k_\lambda)_{\lambda\in\fP(l)_+}$ satisfies $\sum_{\lambda\in\fP(l)_+} k_\lambda=k$ and $\sum_{\lambda\in\fP(l)_+} k_\lambda N(\lambda)=n$, then we may define $\uk'$ by replacing $k_\trsp$ with $k'_\trsp=k_\trsp+1$, and the previous map restricts to a map
\[
H_{i-2}(\Conf_{\uk}(\Sigma)^\infty,\infty;\bQ)\to H_{i}(\Conf_{\uk'}(\Sigma)^\infty,\infty;\bQ)
\]
which by 
\cite{Church} (see in particular \cite[Theorem 1.3]{Knudsen}, and the discussion in \cite[§2.3]{ORW2}) induces an isomorphism for $i>2k-k_1$, and in particular for $i>n$. Note also that every $\uk'$ with $\sum k'_\lambda=k+1$ and $\sum k'_\lambda N(\lambda)=n+1$ must satisfy  $k'_\trsp\ge1$, otherwise we would have $n\ge 2k+1$. 
\end{proof}

\subsection{Proof of Theorem \ref{thm:E}}
Recall Definition \ref{defn:Confmonoid}, and in particular the commutative monoid structure on $\Conf(X,S)^\infty$. We generalise the construction in the following definition.
\begin{defn}
For a space $X$ and for $k\ge0$ we denote by $\SP(X)_k= X^k/\fS_k$ the $k$\textsuperscript{th} symmetric power of $X$; here the quotient is strict. We regard $\SP(X)\coloneqq \coprod_{k\ge0}\SP(X)_k$ as the free strictly commutative topological monoid generated by $X$. Its elements may be written as formal sums $\sum_{i=1}^nk_i\cdot x_i$, with $n\ge0$, distinct points $x_1,\dots,x_n$, and coefficients $k_1,\dots,k_n\ge1$.

Similarly, for a finite set $S$, we regard $\SP(X)^S\cong\SP(X\times S)$ as the free strictly commutative topological monoid generated by $X\times S$. Its elements take the form $\sum_{i=1}^n\uk_i\cdot x_i$, where the points $x_i$ are still distinct, but now now $\uk_i\in\bN^S\setminus\uzero$.

If moreover $\cI\subseteq\bN^S$ is an ideal (i.e. $\uzero\notin\cI$, and the sum $\uk+\uk'$ belongs to $\cI$ as soon as either $\uk$ or $\uk$ belongs to $\cI$), then the difference $\bN^S\setminus\cI$ is a partial abelian monoid, and we define
\[
\Conf(X,\bN^S\setminus\cI)=\coprod_{\uk\in\bN^S}\Conf(X,\bN^S\setminus\cI)_{\uk}\subseteq\SP(X)^S=\coprod_{\uk\in\bN^S}\SP(X)_{\uk}
\]
as the subspace comprising those formal sums $\sum_{i=1}^n\uk_i\cdot x_i$ with all coefficients $\uk_i\in\bN^S\setminus \cI$. We regard $\Conf(X,\bN^S\setminus\cI)$ as an $\bN^S$-graded space. When $X$ is a compact topological space, by taking one point compactifications in a $\bN^S$-graded fashion, we obtain the $\bN^S$-graded pointed space
\[
\Conf(X,\bN^S\setminus\cI)^\infty\coloneqq\bigvee_{\uk\in\bN^S}\Conf(X,\bN^S\setminus\cI)_{\uk}^\infty;
\]
we may regard it as the quotient of $\SP(X)^S$ by the topological ideal spanned by those elements $\sum_{i=1}^n\uk_i\cdot x_i$ with at least one $\uk_i\in\cI$.
\end{defn}
\begin{rem}
\label{rem:SPmanifolds}
If $\Sigma$ is an oriented topological surface, then $\SP(\Sigma)_k$ is an oriented topological $2k$-manifold,
and $\SP(\Sigma)_{\uk}\subseteq\SP(\Sigma)^S$, as well as its open subsets of the form $\Conf(\Sigma,\bN^S\setminus\cI)_{\uk}$, are  oriented topological manifolds of dimension $2\sum_{s\in S}k_s$.
\end{rem}
Out of an ideal $\cI\subseteq\bN^S$ we may also extract a commutative algebra in $\bN^S$-graded $\bQ$-vector spaces $\bQ[\bN^S]/\cI$, which is the quotient of the monoid algebra $\bQ[\bN^S]$ by the ideal generated by the elements in $\cI$.
If $X$ is a CW complex, we then have an identification of commutative algebras in $\Sp_\bQ^\bN$
\[
\int_{|X|}\bQ[\bN^S]/\cI\simeq \pa{|\Conf(X,\bN^S\setminus\cI)^\infty|,\infty}_!\bQ.
\]
\begin{ex}
\label{ex:intConfidentification}
We have used the above identification in Subsection \ref{subsec:proofthmD} for $X=\Sigma$, $S=\fP(l)_+$ and $\cI\subseteq\bN^S$ being the ideal of all $\uk$ with $\sum_\lambda k_\lambda\ge2$. In this case $\Conf(X,\bN^S\setminus\cI)\cong\Conf(X,S)$.
\end{ex}
\begin{proof}[Proof of Theorem \ref{thm:E}]
By \cite[Proposition 6.1]{BCP:Hilb}, there are isomorphisms of commutative algebras in $\Vect_\bQ^\bN$ of the form $\cAh(2)\cong\bQ[x]/x^2$, $\cAh(3)\cong\bQ[x]/x^3$, $\cAh(4)\cong\bQ[x,y]/(x^3-4xy,x^4,y^2)$ and $\cAh(5)\cong\bQ[x,y]/(x^4-5x^2y,x^4-25y^2,x^5)$. In all cases $x$ has grading $1$ and $y$ has grading $2$ with respect to the grading induced by the norm. A change of variables $(x,y)\mapsto(x,x^2-4y)$ leads to an isomorphism $\cAh(4)\cong\bQ[x,y]/(x^4,xy,y^2)$; similarly, a change of variables $(x,y)\mapsto(x,x^2-5y)$ leads to an isomorphism $\cAh(5)\cong\bQ[x,y]/(x^5,x^2y,y^2)$. We have thus found a monomial presentation for $\cAh(l)$ for $l\le5$. Precisely, we have identified $\cAh(l)\cong\bQ[\bN]/(\bN\setminus M_l)$ for $l=2,3$, and $\cAh(l)\cong\bQ[\bN^2]/(\bN^2\setminus M_l)$ for $l=4,5$. This perspective suggests to change our notion of grading in the two cases $l=4,5$: instead of the single $\bN$-grading coming from the norm of permutations, we consider in the following on $\cAh(4)$ and $\cAh(5)$ a bigrading coming from declaring $x$ to be in bigrading $(1,0)$ and $y$ to be in bigrading $(0,1)$.
We may now deduce the following equivalences of commutative algebras in $\Sp_\bQ^\bN$ (for $l=2,3$) or $\Sp_\bQ^{\bN\times\bN}$ (for $l=4,5$)
\[
\int_{|\Sigma|}\cAh(l)\simeq \pa{|\Conf(\Sigma,M_l)^\infty|,\infty}_!\bQ,\quad l=2,3,4,5.
\]
We would now like to deduce a formula for $\int_{|\Sigma|}\cA(l)$. For $l=2,3$ we just operate the endofunctor $\shift$ from Remark \ref{rem:evenshift}: using Remark \ref{rem:SPmanifolds} we obtain for $n\ge0$ the identification of $|\Conf(\Sigma,M_l)_n^\infty|_*\bQ
$ with
\[
(|\Conf(\Sigma,M_l)_n^\infty|,\infty)_!\bQ[-2n]\simeq 
\pa{\int_{|\Sigma|}\cAh(l)}_n[-2n]\simeq
\pa{\int_{|\Sigma|}\cA(l)}_n.
\]
For $l=4,5$ we operate the endofunctor of $\Sp_\bQ^{\bN\times\bN}$ sending $(\omega_{a,b})_{(a,b)\in\bN\times\bN}$ to the object $(\omega_{a,b}\ [-2a-4b])_{(a,b)\in\bN\times\bN}$. For $a,b\ge0$, this allows us to identify the shift $|\Conf(\Sigma;M_l)_{a,b}|_*\bQ\ [-2b]$ with
\[
\pa{|\Conf(\Sigma;M_l)^\infty_{a,b}|,\infty}_!\bQ\ [-2a-4b]\simeq 
\pa{\int_{|\Sigma|}\cAh(l)}_{a,b}[-2a-4b]\simeq
\pa{\int_{|\Sigma|}\cA(l)}_{a,b}.
\]

\end{proof}

\subsection{Even and odd dimensional classes}
\label{subsec:oddexamples}
In this subsection we provide partial computations of the cohomology groups appearing in Corollary \ref{cor:F}.

We start by considering $l=2,3$:
we aim at computing $H^*(|\Conf(\fp_0\tp,M_l)_n|;\bQ)$ for $n\ge2l-1$. We first observe that we have an equivalence 
$|\Conf(\fp_0\tp,M_l)_n|\simeq|\Conf(S^2,M_l)_n|\sslash\SO(3)$. There is moreover a fibre sequence
\[
S^2\simeq\SO(3)/\SO(2)\hookrightarrow|\Conf(S^2,M_l)_n\sslash\SO(2)\xrightarrow{\pi} |\Conf(S^2,M_l)_n\sslash\SO(3);
\]
fibres carry a natural orientation, and since $S^2$ has non-zero Euler characteristic, $\pi$ induces an injective map in rational cohomology, as witnessed by the Becker--Gottlieb transfer along $\pi$. In particular the Serre spectral sequence associated with $\pi$ collapses, so that there is an equivalence of graded vector spaces
\[
H^*\pa{|\Conf(S^2,M_l)_n\sslash\SO(2);\bQ}\cong H^*\pa{|\Conf(S^2,M_l)_n\sslash\SO(3);\bQ}\otimes H^*(S^2;\bQ).
\]
We next focus on computing the left hand side.
We may fix a point $\bar z\in S^2$ fixed by the $\SO(2)$-action and filter $\Conf(S^2,M_l)_n\subseteq\SP(S^2)_n$ according to the multiplicity of $\bar z$: more precisely, for $0\le k\le l-1$, we let $F_k=F_k\Conf(S^2,M_l)_n$ denote the subspace in which $\bar z$ has multiplicity at most $k$. Note that $F_k\setminus F_{k-1}$ is a codimension $2$ closed embedded submanifold in $F_k$ with oriented normal bundle. The stratification is preserved by the action of $\SO(2)$, so we have an induced stratification $|F_\star\Conf(S^2,M_l)_n|\sslash\SO(2)$ on $|\Conf(S^2,M_l)_n|\sslash\SO(2)$. We may thus run a spectral sequence to compute $H^*\pa{|\Conf(S^2,M_l)_n|\sslash\SO(2);\bQ}$, with first page 
\begin{equation}
\label{eq:firstpage}    
E_1^{k,i}=H^{k+i}(|F_k|\sslash\SO(2),|F_{k-1}|\sslash\SO(2);\bQ)\cong H^{k+i-2}(|F_k\setminus F_{k-1}|\sslash\SO(2);\bQ).
\end{equation}

To compute this first page, we first note that the difference $F_k\setminus F_{k-1}$ is $\SO(2)$-equivariantly homeomorphic 
to $\Conf(\R^2,M_l)_{n-k}$.

We next argue that for $n-k\ge l$ the rational cohomology of $\Conf(\R^2,M_l)_{n-k}$ has rank 1 in degrees 0 and $2l-3$, and vanishes in all other degrees. For $n-k=l$ this comes from the equivalence $\Conf(\R^2,M_l)_l\simeq S^{2l-3}$; by a variation of \cite[Theorem 2]{Dold}, the stabilisation maps $\Conf(\R^2,M_l)_n\to\Conf(\R^2,M_l)_{n+1}$ are surjective in cohomology in all degrees; and the stable cohomology for $n\to\infty$ is $H^*(\Omega^2S^{2l-1};\bQ)\cong H^*(S^{2l-3};\bQ)$ \cite{GKY2}.

The Serre spectral sequence computing \eqref{eq:firstpage} has second page given by $\cE_2^{p,q}\cong H^p(\bbB\SO(2);\bQ)\otimes H^q(|F_k\setminus F_{k-1}|;\bQ)$; in particular $\cE_2^{p,q}$ vanishes unless $p$ is even and $q\in\set{0,2l-3}$, in which case $\cE_2^{p,q}\cong\bQ$. 
To understand the differential on the $\cE_{2l-2}$-page, we observe that, by the lower bound on $n$, the Lie group $\SO(2)$ acts on the manifold $\Conf(\R^2,M_l)_{n-k}$ with finite stabilisers; in particular the rational cohomology of $|\Conf(\R^2,M_l)_{n-k}|\sslash\SO(2)$ vanishes in high enough degrees. 
It follows that the $\cE_{2l-2}$-differential is non-trivial, and in particular $\cE_\infty^{p,q}\cong\bQ$ if and only if $0\le p\le 2l-4$ is even and $q=0$. 
This computes \eqref{eq:firstpage}, and in particular it shows that the first page $E^1$ is concentrated in even degrees.
The isomorphism of graded vector spaces
\[
H^*\pa{|\Conf(S^2,M_l)_n\sslash\SO(2);\bQ}\cong\bQ[\kappa]/\kappa^{l-1}\otimes H^*(S^2;\bQ),
\]
with $\kappa$ in cohomological degree 2,
follows immediately, whence we deduce the isomorphism of graded vector spaces
$H^*\pa{|\Conf(S^2,M_l)_n\sslash\SO(3);\bQ}\cong\bQ[\kappa]/\kappa^{l-1}$.

We next turn to the cases $l=4,5$. For these, we aim at checking that 
\begin{equation}
\label{eq:oddgroup}
H^{2l-3}\pa{|\Conf(\fp_0\tp,M_l)_{a,3}|;\bQ}\cong H^*\pa{|\Conf(\Ptp,M_l)_{a,3}|\sslash\Aut(\bP^1)\tp;\bQ}\ncong 0
\end{equation}
for $a\ge l$.
We consider the fibre bundle map $\Conf(\Ptp,M_l)_{a,3}\to\Conf_3(\Ptp)$,
given by only retaining the 3 points of a configuration whose coefficient in $M_l\subseteq\bN\times\bN$ has non-trivial (i.e., equal to 1) second component.
The Lie group $\Aut(\bP^1)\tp$ acts transitively on $\Conf_3(\Ptp)$, with stabiliser $\fS_3$. We may therefore fix the configuration $\bar P=\set{0,1,\infty}$ of three points in $\Ptp$, and  identify the homology group in \eqref{eq:oddgroup} with
$H^{2l-3}(|\Conf(\Ptp,M_l)_{a,\bar P}|;\bQ)^{\fS_3}$, where $\Conf(\Ptp,M_l)_{a,\bar P}$ is the fibre at $\bar P$ of said fibre bundle. 
Passing to homology, it suffices to find a non-trivial, $\fS_3$-invariant class in $H_{2l-3}(|\Conf(\Ptp,M_l)_{a,\bar P}|;\bQ)$.

In the following we identify $S^2\cong\Ptp$, and identify $\fS_3$ as a subgroup of $\SO(3)$.
We observe that $\Conf(\R^2,M_l)_{l,0}\simeq S^{2l-3}$, and in particular we have a ``fundamental class'' in $H_{2l-3}(|\Conf(\R^2,M_l)_{l,0}|;\bQ)$. We may implant this class inside a small disc in $S^2\setminus \bar P$, and pick $a-l$ additional distinct points, in the complement of said disc and away from $\bar P$, all having label $(1,0)\in M_l$: the result is a class $w\in H_{2l-3}(|\Conf(S^2,M_l)_{a,\bar P}|;\bQ)$ which is evidently $\fS_3$-invariant. To prove that it is non-zero, we consider $\Conf(S^2,M_l)_{a,\bar P}$ as a smooth complex variety of complex dimension $a$ and let $W$ denote the closure in $\Conf(S^2,M_l)_{a,\bar P}$
of the locus 
comprising configurations $\sum_{i=1}(a_i,b_i)\cdot z_i$ such that there exist indices $i\neq j$ satisfying the following:
\begin{itemize}
\item $z_i,z_j\in S^2\setminus\bar P=\C\setminus\set{0,1}$;
\item $\Re(z_j)=\Re(z_i)$ and $\Im(z_j)>\Im(z_i)$;
\item $a_i=2$ and $a_j=l-2\ge2$.
\end{itemize}
We have that $W$ is a smooth manifold of codimension $2l-3$ inside $\Conf(S^2,M_l)_{a,\bar P}$ away from a locus of codimension at least 2 inside $W$. In particular $W$ carries a fundamental cohomology class in $H^{2l-3}(|\Conf(S^2,M_l)_{a,\bar P})$, and this cohomology class pairs non-trivially with $w$.

\subsection{Picard groups}
\label{subsec:Picard}
Let $d\ge4$ and let $g$ be such that $2(g+d-1)> 2I_d+J_d$. We aim in the following at computing the cohomology groups 
\begin{equation}
\label{eq:H1H2iso}    
H^1(|\cH_{g,d}\tp|;\bQ)\cong0;\quad H^2(|\cH_{g,d}\tp|;\bQ)\cong\bQ^2.
\end{equation}

Abbreviate $\upsilon=\upsilon(g,d,0)=g+d-1$ in the following.
For $i=1,2$, the only contribution to $H^i(|\cH_{g,d}\tp|;\bQ)$ from Corollary \ref{cor:C} comes from $H^i\pa{|\cH_{0,1}\tp|,\int_{|\fp_{0,1}\tp|}\cA(d)}_{2\upsilon}$.
Since $\pa{\int_{|\fp_{0,1}\tp|}\cA(d)}_{2\upsilon}$ is a rational spectrum concentrated in non-positive degrees, and since $|\cH_{0,1}\tp|\simeq\bbB\SO(3)$ is simply connected and rationally 3-connected, a standard spectral sequence argument reduces to computing $\pi_{-i}\pa{\int_{|S^2|}\cA(d)}_{2\upsilon}$ for $i=1,2$. 
We next recall from \cite[§3.3]{BCP:Hilb} that there is a map of graded commutative ring spectra $\bQ[x,y]\to\cA(d)$ that, in all gradings, is surjective on $\pi_*$ in all degrees, and an isomorphism in degrees $*\ge-5$; here $x$ has degree $-2$ and grading 1, whereas $y$ has degree $-4$ and grading $2$. Since $S^2$ has dimension 2, we obtain that the map $\int_{|S^2|}\bQ[x,y]\to\int_{|S^2|}\cA(d)$ induces an isomorphism on $\pi_*$  in all gradings and in degrees $*>-5+2=-3$. We thus reduce to computing $\pi_{-i}\pa{\int_{|S^2|}\bQ[x,y]}_{2\upsilon}$ for $i=1,2$. We may now identify
\[
\pi_{-i}\pa{\int_{|S^2|}\bQ[x,y]}_{2\upsilon}\cong\bigoplus_{a+2b=2\upsilon}H_{2a+4b-i}(|\SP(S^2)_a\times\SP(S^2)_b|;\bQ)[-2b],\quad i=1,2
\]
from which \eqref{eq:H1H2iso} immediately follows, after noting that only the terms with $b\le1$ may contribute non-trivially for $i=2$, and only the term with $b=0$ for $i=1$.

Assume now that the smooth Deligne--Mumford algebraic stack $\cH_{g,d}$ admits a compactification to a \emph{smooth}, proper Deligne--Mumford stack $\overline{\cH_{g,d}}$.
Then the same arguments used in \cite[§7]{LL1} may be employed to show that $\Pic(\cH_{g,d})\otimes\bQ$ injects into $H^2(|\cH_{g,d}\tp|;\bQ)$ along the cycle class map. In particular $\Pic(\cH_{g,d})\otimes\bQ$ has dimension $\le2$. It was however shown in \cite[Proposition 2.15]{DeopurkarPatel} that $\Pic(\cH_{g,d})\otimes\bQ$ has dimension $\ge2$. 
We could not find a reference providing the desired $\overline{\cH_{g,d}}$, but believe that it would be definitely worthwhile to construct such a compactification, even beyond the scopes of this article.

We conclude by mentioning that for $d\le 5$ and for all $g\ge2$ both the rational Picard group $\Pic(\cH_{g,d})\otimes\bQ$ \cite{DeopurkarPatel,SF} and the integral Picard group $\Pic(\cH_{g,d})$ \cite{ArsieVistoli,BolognesiVistoli,CL3} have been computed.

\bibliography{Bibliography5.bib}
\bibliographystyle{alpha}

\end{document}